\RequirePackage{fix-cm}
\RequirePackage{rotating}
\documentclass[smallextended,numbook]{svjour3}       
\smartqed  
\usepackage{graphicx}
\usepackage{amsmath,amssymb,amsfonts}%
\usepackage{amstext}
\usepackage{enumerate}
\usepackage{multicol,microtype}
\usepackage{rotating} 
\usepackage{color} 
\usepackage{subcaption}
\usepackage{cases}  
\usepackage[active]{srcltx}
\usepackage{algorithm}%
\usepackage{algorithmicx}%
\usepackage{cite} 
\usepackage[colorinlistoftodos,prependcaption,textsize=footnotesize]{todonotes}

\usepackage[misc]{ifsym}
\usepackage{hyperref}
\usepackage{booktabs}
\newcommand*{\affaddr}[1]{#1}
\newcommand*{\affmark}[1][*]{\textsuperscript{#1}}

\spnewtheorem{assumption}{Assumption}{\bfseries}{\itshape}

\def\R{{\rm I\!R}}

\def\argmin{\mathop{\rm arg\,min}}
\def\Argmin{\mathop{\rm Arg\,min}}

\def\Argmax{\mathop{\rm Arg\,max}}

\def\bd{{{\rm bdry}\,}}

\def\xfeas{x^\odot}

\def\d{{\rm dist}}

\def\X{{\mathbb{X}}}
\def\sgn{{\rm Sgn}}
\def\LO{{\cal LO}}
\def\AWO{{\cal AWO}}

\definecolor{blue}{rgb}{0,0,0}
\definecolor{blueblue}{rgb}{0,0,0}

\begin{document}
	
	\title{{\color{blue}Frank-Wolfe-type} methods for {\color{blue}a class of} nonconvex inequality-constrained problems \thanks{{\color{blueblue}Liaoyuan Zeng was supported partly by the National Natural Science Foundation of China (12201389).} Yongle Zhang was supported partly by the National Natural Science Foundation of China (11901414) and (11871359). Guoyin Li  was partially supported by  discovery projects from Australian Research Council (DP190100555 and DP210101025) and a UNSW-SJTU Collaborative Research Seed Grant (RG200965). Ting Kei Pong was supported in part by the Hong Kong Research Grants Council PolyU153004/18p.
}
	}
	
	
	\author{Liaoyuan Zeng \protect\affmark[1]  \and
		Yongle Zhang \protect\affmark[2] \and Guoyin Li \protect\affmark[3] \and Ting Kei Pong \protect\affmark[4] \and {\color{blueblue}Xiaozhou Wang} \protect\affmark[4]  
	}
	
	
	
	\institute{Liaoyuan Zeng \\  \email{zengly@zjut.edu.cn} \\\\
		\Letter $\;$   Yongle Zhang \\  \email{yongle-zhang@163.com} \\\\  Guoyin Li \\   \email{g.li@unsw.edu.au} \\\\
		Ting Kei Pong \\    \email{tk.pong@polyu.edu.hk} \\\\
{\color{blueblue} Xiaozhou Wang \\    \email{xzhou.wang@connect.polyu.hk} }\\\\
		\affaddr{ \affmark[1]College of Science, Zhejiang University of Technology, Hangzhou, Zhejiang, People's Republic of China.\\
			\affmark[2]School of Mathematical Sciences, Visual Computing and Virtual Reality Key Laboratory of Sichuan Province, Sichuan Normal University, Chengdu, Sichuan, People's Republic of China. \\
			\affmark[3]Department of Applied Mathematics, University of New South Wales, Sydney, Australia.\\
			\affmark[4]Department of Applied Mathematics, The Hong Kong Polytechnic University, Hong Kong, People's Republic of China.
		}\\
	}

	\date{Received: date / Accepted: date}

	\maketitle
	
	\begin{abstract}
		The Frank-Wolfe (FW) method, which implements efficient linear oracles that minimize linear approximations of the objective function over a {\it fixed} compact convex set, has recently received much attention in the optimization and machine learning literature. In this paper, we propose a new FW-type method for minimizing a smooth function over a compact set defined {\color{blue}as the level set of a single {\em difference-of-convex} function}, based on new {\em generalized} linear-optimization oracles (LO). We show that these LOs can be computed efficiently with {\it closed-form solutions} in some important optimization models that arise in compressed sensing and machine learning. In addition, under a mild strict feasibility condition, we establish the subsequential convergence of our nonconvex {\color{blue}FW-type} method. Since the feasible region of our generalized LO typically changes from iteration to iteration, our convergence analysis is {\it completely different} from those existing works in the literature on FW-type methods that deal with {\em fixed} feasible regions among subproblems. Finally, motivated by the away steps for accelerating FW-type {\color{blue}methods} for convex problems, we further design an {\it away-step oracle} to supplement our nonconvex {\color{blue}FW-type} method, and establish subsequential convergence of this variant. Numerical results on {\color{blue}the matrix completion problem {\color{blueblue}with standard datasets} are presented to demonstrate the efficiency of the proposed FW-type method and its away-step variant.}

	\end{abstract}

\section{Introduction}

Many optimization problems that arise in application fields such as statistics, computer science and data science can be cast into constrained optimization problems that minimize smooth functions over compact sets:
\begin{equation}\label{P00}
  \begin{array}{rl}
\min\limits_{x\in\X} & f(x) \ \ \ \ {\rm s.t.} \ \ x \in C,
  \end{array}
\end{equation}
where $\X$ is a finite-dimensional Euclidean space, $f:\X\rightarrow\R$ is continuously differentiable and $C$ is a nonempty compact set in $\X$.
When projections onto $C$ can be efficiently computed, the classical gradient projection method \cite{Bertsekas99,Goldstein64,LevitinPolyak66} and its variants are usually the prominent choices of algorithms for solving \eqref{P00}, due to their ease of implementation and scalability.

Projections onto $C$, however, are not necessarily easy to compute; see for example \cite{FrGM17,HaJN15} for some concrete instances that arise in applications.
In this case, scalable first-order methods that do not involve projections may be employed. When the $C$ in \eqref{P00} is \emph{convex}, one popular class of such algorithms is the Frank-Wolfe (FW) method (also called the conditional gradient method) \cite{FrankWolfe56} and its variants. Unlike the
gradient projection methods which require efficient projections onto $C$, FW method, in each iteration, makes use of a linear approximation of $f$, and moves towards a minimizer of this linear function over $C$ along a straight line to generate {\color{blue}the next iterate} in $C$. In particular,
FW method uses a {\em linear-optimization oracle} in each iteration; these {\color{blue}kinds} of oracles can be
much less computationally expensive than projecting onto $C$ in many applications \cite{FrGM17,GarberHazan16,Jaggi13}. Due to their low iteration costs and ease of implementations,  FW method
and its variants have found applications in machine learning and have received much renewed interest in recent years \cite{FrGM17,GarberHazan16,Jaggi13,LaZh16,Lacoste15,Pedregosa20}. For example, when $f$ is also convex and satisfies certain curvature conditions, \cite{Jaggi13} established
the $O(1/k)$ complexity of FW-type methods for \eqref{P00} and presented their powerful applications for
solving sparse optimization models. In \cite{GarberHazan16,Lacoste15,Pedregosa20}, linear convergence results of some FW-type methods
were established under suitable conditions such as strong convexity of
$f$ or $C$ in \eqref{P00}. 
More recently, refinements on the FW method were presented by incorporating
the so-called ``in-face" directions \cite{FrGM17}, which extended the idea of ``{\color{blue}away steps}" proposed earlier in \cite{Guelat86,Wolfe}.
{\color{blue}For the recent development} of FW-type methods, we refer the readers to \cite{BRZ21} for a {\color{blue}nice} survey. 

The previous discussions were on FW-type methods for \eqref{P00} when $C$ is convex.
In the case when $C$ in \eqref{P00} is not convex, the study of FW-type methods are much more limited. Indeed, when $C$ is not convex, a notable difficulty is that one may move \emph{outside of $C$} when moving towards a minimizer of the linear-optimization oracle in an iteration of the FW method.
Despite this difficulty, some important contributions along this direction were given in \cite{RT13,BalashovPolyak20}. Specifically, in \cite{RT13},  FW-type methods have been extended to some optimization models for sparse principal component analysis, whose $C=\{x \in \mathbb{R}^n: \|x\|=1, \, \|x\|_0 \le r\}$, where $\|x\|_0$ denotes the cardinality of $x$ and $r>0$. More recently, \cite{BalashovPolyak20} further discussed how FW-type methods can be developed when $C$ is a sphere or more generally a smooth manifold. Interestingly, the feasible region $C$ considered in \cite{RT13,BalashovPolyak20} {\color{blue}has an} empty interior.

In this paper, different from \cite{RT13,BalashovPolyak20}, we are interested in developing {\color{blue}FW-type} methods for \eqref{P00} when $C$ is nonconvex and has (possibly) nonempty interior. Specifically, we consider the following nonconvex optimization problem:
 \begin{equation}\label{P0}
  \begin{array}{rl}
\min\limits_{x\in\X} & f(x) \ \ \ \ {\rm s.t.} \ \ P_1(x) - P_2(x) \leq \sigma,
  \end{array}
\end{equation}
 where $f:\X\rightarrow\R$ is continuously differentiable,
 $P_1:\X\to \R$ and $P_2:\X\rightarrow\R$ are {\color{blue} real-valued} convex functions, $\sigma >0$ and the feasible set $\mathcal{F} :=\{x\in \X:\; P_1(x) - P_2(x) \leq \sigma\}$ is nonempty and compact.

It is easy to see that \eqref{P0} is a special case of \eqref{P00} with $C=\{x: P_1(x)-P_2(x) \le \sigma\}$.
Notice that, the model \eqref{P0} covers, for example, sparse optimization models whose feasible region can be described as $\{x \in \mathbb{R}^n: P(x) \le r\}$, with $P$ being a weighted difference of $\ell_1$ and $\ell_2$ norms, and $r>0$. For more optimization models of this form, see Section~\ref{sec:LO} below. {\color{blue}Our method} is an FW-type method in the sense that we make use of a (generalized) linear-optimization oracle (see Definition~\ref{LOdef}) where we linearize \emph{both} the objective function $f$ and the concave part, $-P_2$, of the constraint function. We invoke this linear-optimization oracle to obtain a search direction. {\color{blue} We then }use a line search procedure to construct the next iterate and {\color{blue} \emph{maintain its feasibility}}. We would like to point out that, when $P_2\equiv 0$, our (generalized) linear-optimization oracle reduces to the classical linear-optimization oracle used in the classical FW method. As a result, similar to {\color{blue}the classical FW method} for \eqref{P00} with convex $C$, {\color{blue}our method} also does not require projections onto $C$.
On the other hand, a notable difference is that, the feasible regions of the linear-optimization oracles in {\color{blue}our method} \emph{evolve} as the algorithm progresses. This is opposed to the existing {\color{blue}FW-type} methods in the literature where the feasible regions of the linear-optimization oracles are fixed as the feasible region of the original problem. As a result, our convergence analysis is \emph{completely different} from those in the existing literature on FW-type methods.

{\color{blue} It is also worth noting that our model problem is closely related to the difference-of-convex (DC) optimization problems  with DC constraints. In particular, note that any twice continuously differentiable function over a compact set is a DC function. In this case, our model problem can be regarded as a special form of  DC optimization problems with DC constraints considered, for example, in \cite{PRA17}.  DC optimization problems form a large class of nonconvex optimization problems, and have been  studied extensively in the literature \cite{LeDinh14,LeDinh18,PRA17,R2014}. One of the most prominent and widely used algorithms for solving the DC optimization problems is the so-called difference-of-convex algorithm (DCA) and its variants \cite{LeDinh18}. Unlike {\color{blue}FW-type} methods, these algorithms usually utilize the majorization-minimization numerical strategy and rely on efficient computability of convex subproblems with \emph{nonlinear} objective functions. As a result, the analysis and targeted applications of the proposed methods here are also {\it drastically different} from the DCA and its variants in the literature.}

The organization and contribution of this paper are as follows:
\begin{itemize}
\item[{\rm (1)}] After {\color{blue}a quick review of} {\color{blue}notation, basic mathematical tools and the classic FW algorithm for \eqref{P00} with convex sets $C $ in Section~\ref{sec2}, we discuss our motivations and ideas for our {\color{blue}FW-type} algorithms for \eqref{P0}.}
\item[{{\color{blue}\rm (2)}}]{\color{blue}We} formulate the (new) linear-optimization oracle of {\color{blue}the proposed method} in Section~\ref{sec:LO}. We also discuss how these linear-optimization oracle problems can be efficiently solved for several important nonconvex optimization models of the form \eqref{P0} arising from compressed sensing and matrix completion.
\item[{\rm (3)}] In Section~\ref{sec4}, we discuss optimality conditions and define a stationarity measure for \eqref{P0}, which paves the way for the convergence analysis of our new {\color{blue}FW-type} methods later.
\item[{\rm (4)}] In Section~\ref{sec:convergence}, we present a new {\color{blue}FW-type} method for solving \eqref{P0}. Under a suitable strict feasibility condition, we establish that the sequence generated by the proposed method is bounded and any accumulation point is a stationary point of \eqref{P0}. In the case where, additionally, the convex part $P_1$ in the constraint is strongly convex and $f$ is Lipschitz differentiable with nonvanishing gradients on ${\cal F}$, we obtain a $o(1/k)$ complexity in terms of the stationarity measure defined in Section~\ref{sec4}.
\item[{\rm (5)}] In Section~\ref{sec6}, motivated by the {\color{blue}improved empirical performances upon incorporating ``away steps" in} FW method in the convex setting, we further introduce a new {\color{blue}FW-type} method with  ``away steps" for \eqref{P0}, and establish its subsequential convergence under suitable conditions.
\item[{\rm (6)}] {\color{blue} Finally, in Section ~\ref{sec7}, we apply the proposed FW-type method and its ``away-step" variant to solve large matrix completion (MC) {\color{blueblue}problems on standard datasets}. We compare our approaches with the In-face extended Frank-Wolfe method in \cite{FrGM17}, which solves  a convex MC model.}
\end{itemize}

\section{Notation and preliminaries}\label{sec2}
In this paper, we use $ \X $ to denote a finite-dimensional Euclidean space. We denote by $ \langle \cdot, \cdot\rangle $ the inner product on $ \X $ and $ \|\cdot\| $ the associated norm. Next, we let $ \R^n $ denote the Euclidean space of dimension $ n $, and $ \R^{m\times n} $ denote the space of all $ m\times n $ matrices. Moreover, the space of $ n\times n $ symmetric matrices will be denoted by $ S^{n} $ and the cone of $ n\times n $ positive semidefinite matrices will be denoted by $ S_+^{n} $. Finally, we let $ \mathbb N_+ $ denote the set of nonnegative integers.

For a set $ S\subseteq\X $ and an $ x\in\X $, we define the distance from $ x $ to $ S $ as $ {\rm dist}(x, S):= \inf\{\|x-y\|:\,y\in S\} $. The convex hull of $ S $ is denoted by $ {\rm conv}(S) $, and $ \bd S $ is the boundary of $ S $. If $ S $ is a finite set, we denote by $ \lvert S\rvert $ its cardinal number. We use $ B(x, r) $ to denote the closed ball with center $ x $ and radius $ r >0 $, i.e., $  B(x, r) = \{y\in\X: \|x-y\| \le r\} $.

We say that an extended-real-valued function $ f:\X\rightarrow [-\infty, \infty] $ is proper if its effective domain $ {\rm dom}\, f:=\{x\in\X:\, f(x)<\infty \} $ is not empty and $ f(x)>-\infty $ for every $ x\in\X $. A proper function is said to be closed if it is lower semicontinuous. The limiting subdifferential of a proper closed function $ h $ at $ \bar x\in{\rm dom}\, h $ is given as
\[
\partial h(\bar x):= \left\{\upsilon\in\X:\; \exists x^k\overset{h}\rightarrow \bar x \text{ and } \upsilon^k \in \widehat{\partial} h(x^k) \text{ with } \upsilon^k \rightarrow \upsilon\right\},
\]
where $ x^k\overset{h}\rightarrow \bar x $ means $ x^k\rightarrow \bar x $ and $ h(x^k)\rightarrow h(\bar x) $; here, $\widehat{\partial} h$ is the so-called regular (or Fr\'{e}chet) subdifferential, which, for any $x\in {\rm dom\,} h$, is given by
\[
\widehat{\partial} h(x):=\left\{\upsilon\in\X :\; \liminf\limits_{z\rightarrow x,z\neq x}\frac{h(z)-h(x)-\langle \upsilon,z-x\rangle }{\|z-x\|}\ge 0 \right\}.
\]
By convention, we set $\widehat{\partial}h(x)= \partial h(x) = \emptyset$ if $ x\notin {\rm dom}\, h $. It is known that $\partial h(x) = \{\nabla h(x)\}$ if $h$ is continuously differentiable at $x$; see \cite[Exercise~8.8(b)]{RoWe98}. Moreover, when $h$ is proper closed and convex, $\partial h$ reduces to the classical subdifferential in convex analysis; see \cite[Proposition~8.12]{RoWe98}.
Next, for a locally Lipschitz continuous function $ h $, we define its Clarke subdifferential at $ \bar x$ as
\[
\partial^\circ h(\bar x):= \left\{\upsilon\in\X :\; \limsup_{x\rightarrow \bar x, t\downarrow 0} \frac{{\color{blue}h}(x+tw)-{\color{blue}h}(x)}{t}\geq \langle \upsilon, w\rangle \text{ for all } w\in\X\right\};
\]
it holds that $\widehat{\partial}h(\bar x) \subseteq \partial h(\bar x)\subseteq \partial^\circ h(\bar x)$ {\color{blue} (see \cite[Theorem~5.2.22]{BoZh04}); such an $ h $ is said to be \emph{regular} at $ \bar x $ if $ \widehat{\partial} h(\bar x) = \partial^\circ h(\bar x) $.}

We now recall a suitable constraint qualification on the constraint set of \eqref{P0} and present the notion of stationary points of \eqref{P0}.
\begin{definition}[{gMFCQ}]
For \eqref{P0}, we say that the generalized Mangasarian-Fromovitz constraint qualification (gMFCQ) holds at an $x^*\in\mathcal{F}$ if the following implication holds:
\[
{\rm If }\; P_1(x^*) - P_2(x^*) = \sigma,  \;{\rm then }\; 0\not\in\partial^\circ(P_1 - P_2)(x^*).
\]
\end{definition}
Note that gMFCQ reduces to the standard MFCQ when $P_1-P_2$ is smooth.

\begin{definition}[{Stationary point}]\label{Stationary}
We say that an $x^*\in\X$ is a stationary point of \eqref{P0} if there exists $\lambda_*\in\R_+$ such that $(x^*,\lambda_*)$ satisfies
\begin{enumerate}[{\rm (i)}]
	\item $0\in \nabla f(x^*) + \lambda_*(\partial P_1(x^*) - \partial P_2(x^*))$;
	\item $\lambda_*(P_1(x^*) - P_2(x^*) - \sigma)=0$, and $P_1(x^*) - P_2(x^*) \leq \sigma$.
\end{enumerate}
\end{definition}
We next deduce that any local minimizer of \eqref{P0} is a stationary point under gMFCQ.
\begin{proposition}\label{loc-sta}
Consider \eqref{P0}. If the gMFCQ holds at every point in $\mathcal{F}$, then any local minimizer of \eqref{P0} is a stationary point of \eqref{P0}.
\end{proposition}
\begin{proof}
For any local minimizer $x^*$ of \eqref{P0}, we can deduce from \cite[Theorem~10.1]{RoWe98} that
\begin{equation}\label{subdif}
0\in\partial(f + \delta_{[P_1 - P_2\leq\sigma]})(x^*),
\end{equation}
where $ \delta_S $ is the indicator function of the set $ S $. Below, we consider two cases.

{\bf Case 1:} $P_1(x^*) - P_2(x^*) < \sigma$. As $P_1$ and $P_2$ are real-valued convex functions, they are also continuous. By the continuity of $P_1$ and $P_2$ and \eqref{subdif}, we have that $\nabla f(x^*) = 0$. Thus $x^*$ is a stationary point of \eqref{P0} (with $\lambda_* = 0$ in Definition~\ref{Stationary}).

{\bf Case 2:} $P_1(x^*) - P_2(x^*) = \sigma$. Then we have that
\[
\begin{aligned}
0&\in\partial(f + \delta_{[P_1 - P_2\leq\sigma]})(x^*) \overset{\rm (a)} = \nabla f(x^*) + \mathcal{N}_{[P_1 - P_2\leq\sigma]}(x^*)\\
&\overset{\rm (b)}\subseteq \nabla f(x^*) + \bigcup\limits_{\lambda\ge0}\lambda\partial^\circ(P_1 - P_2)(x^*)\overset{\rm (c)}\subseteq \nabla f(x^*) + \bigcup\limits_{\lambda\ge0}\lambda(\partial P_1(x^*) - \partial P_2(x^*)).
\end{aligned}
\]
where $ \mathcal{N}_S(x):= \partial \delta_S(x) $ is the limiting normal cone of the set $ S $ at $ x $, and (a) holds in view of \cite[Exercise~8.8]{RoWe98} and the smoothness of $ f $, (b) follows from \cite[Theorem~5.2.22]{BoZh04}, the first corollary to \cite[Theorem~2.4.7]{Cl90} and the fact that $0\not\in\partial^\circ(P_1 - P_2)(x^*)$ (thanks to the gMFCQ), (c) holds because of \cite[Proposition~2.3.3]{Cl90}, \cite[Proposition~2.3.1]{Cl90} and \cite[Theorem~6.2.2]{BoLe06}. Then, there exists $\lambda_*\ge0$ such that $0\in \nabla f(x^*) + \lambda_*(\partial P_1(x^*) - \partial P_2(x^*))$. Noticing that $P_1(x^*) - P_2(x^*) = \sigma$, we also have $\lambda_*(P_1(x^*) - P_2(x^*) - \sigma)=0$. Therefore, $x^*$ is a stationary point of \eqref{P0}.
\qed \end{proof}

{\color{blue} Next,} we state two lemmas that will be used subsequently in our discussion of stationarity measure in Section~\ref{sec4} and our convergence analysis in Sections~\ref{sec:convergence} and \ref{sec6}. The first lemma is due to Robinson \cite{Rob75} and concerns error bounds for the so-called ${\cal K}-$convex functions: Given a closed convex cone ${\cal K}\subseteq \R^m$, we say that $g: \X\rightarrow \R^m$ is $\mathcal{K}-$convex if
\[
\lambda g(x) + (1 - \lambda)g(y)\in g(\lambda x + (1 - \lambda)y) +\mathcal{K},\ \ \forall  x, y\in \X, \ \lambda\in [0,1].
\]
\begin{lemma}\label{ErrorBounded}
Let $\mathcal{K}\subseteq \R^m$ be a closed convex cone and $g:\X\to \R^m$
be a ${\cal K}-$convex function. Let $\Omega := \{x\in \X:\; 0 \in g(x) + {\cal K}\}$ and suppose there exist $x^s\in \Omega$ and $\delta > 0$ such that $B(0,\delta)\subseteq g(x^s) + \mathcal{K}$.
Then
\[
\d(x,\Omega)\leq \frac{\|x - x^s\|}{\delta}\d(0, g(x) + \mathcal{K}), \ \forall x\in \X.
\]
\end{lemma}

Our next lemma concerns the convergence of descent algorithms with line-search based on the Armijo rule, which will be used in our convergence analysis. The proof is standard and follows a similar idea as in \cite[Proposition~1.2.1]{Bertsekas99}. Here we include the proof for the ease of readers.
\begin{lemma}\label{Armijolemma}
Let $ \Gamma \subseteq \X $ be a compact set. Suppose that $ f:\; \X\rightarrow \R $ is continuously differentiable on an open set containing $\Gamma$. Let $ c\in(0, 1) $, $ \eta\in(0, 1) $ and $ \{\alpha_k^{0}\} $ satisfy $ 0<\inf_k \alpha_k^0\leq \sup_k \alpha_k^0< \infty $. Let $ x^0\in\Gamma $, $\{\alpha_k\}$ be a positive sequence, and $\{x^k\}$ and $ \{d^k\} $ be bounded sequences such that the following conditions hold for each $ k\in\mathbb N_+ $:
	\begin{enumerate}[{\rm (i)}]
		\item $\langle \nabla f(x^k), d^k\rangle < 0$.
		\item $ \alpha_k $ is computed via Armijo line search with backtracking from $\alpha^0_k$, i.e., $\alpha_k = \alpha^0_k\eta^{j_k}$ with
		\begin{equation}\label{Armijols}
		\!\!\!\!\!\!j_k\!:=\! \argmin\left\{j\!\in\! \mathbb{N}_+\!\!:\! \; f(x^k\!+\!\alpha_k^0\eta^jd^k)\!\leq\! f(x^k) \!+\! c\alpha_k^0\eta^j\langle\nabla f(x^k),d^k\rangle\right\}\!.\!\!\!\!  \ \footnote{Here and throughout this paper, we use $\Argmin$ to denote the set of minimizers, and write $\argmin$ when the set of minimizers is a singleton.}
		\end{equation}
		\item $ x^{k+1}\in\Gamma $ and satisfies $ f(x^{k+1})\leq f(x^k+\alpha_kd^k) $.
	\end{enumerate}
   Then, it holds that $ \lim_{k\to\infty} \langle\nabla f(x^k),d^k\rangle =0 $.
\end{lemma}
\begin{proof}
According to \eqref{Armijols} and item (iii), we have for every $k\in\mathbb{N}_+ $,
\[
-c\alpha_k\langle \nabla f(x^k), d^k\rangle\leq f(x^{k}) - f(x^k+\alpha_kd^k) \leq f(x^k) - f(x^{k+1}).
\]
Summing from $ k=0 $ to $ \infty $ on both sides of the above display, we obtain
\[
-c\sum\limits_{k=0}^\infty \alpha_k \langle \nabla f(x^k), d^k\rangle \leq f(x^0) - \inf\limits_{x\in\Gamma} f(x) < \infty;
\]
here, $\inf_\Gamma f$ is finite because $\Gamma$ is compact and  $f$ is continuously differentiable on an open set containing $\Gamma$.  Using this together with item (i), we can deduce that
\begin{equation}\label{sddlim}
	\lim_{k\to\infty} \alpha_k\langle \nabla f(x^k), d^k\rangle = 0.
\end{equation}

Next, notice that $\{x^{k}\}\subseteq \Gamma$ and $\Gamma$ is compact. This together with the boundedness of $\{d^k\} $ and the continuity of $ \nabla f $ implies that $ \{\langle\nabla f(x^k), d^k\rangle\} $ is bounded. Therefore, to prove the desired conclusion, it suffices to show that the limit of any convergent subsequence of $ \{\langle\nabla f(x^k), d^k\rangle\} $ is zero.

To this end, fix any convergent subsequence $ \{ \langle\nabla f(x^{k_t}), d^{k_t}\rangle \} $. Since $ \{\alpha_k\} $ is also bounded, by passing to a further subsequence if necessary, we have $\lim_{t\to\infty} (x^{k_t}, d^{k_t}, \alpha_{k_t}) = (x^*, d^*, \alpha_*)$ for some $ x^*\in\Gamma $, $ d^*\in\X $ and $ \alpha_*\geq 0 $. We consider two cases:

{\bf Case 1}: $ \alpha_*>0 $. Then $ \lim_{t\to\infty} \langle\nabla f(x^{k_t}), d^{k_t}\rangle = 0 $ follows directly from \eqref{sddlim}.

{\bf Case 2}: $\alpha_*=0$. In this case, using the fact that $ \inf \alpha_k^0 > 0 $ and the definition of $\alpha_k$ in \eqref{Armijols}, we see that backtracking must have been invoked for all large $t$, i.e., there exists $T $ such that $j_{k_t} \ge 1$ whenever $t\ge T$. Then $j = j_{k_t}-1$ violates the inequality in \eqref{Armijols}, i.e.,
\[
f(x^{k_t}+\eta^{-1}\alpha_{k_t}d^{k_t}) > f(x^{k_t}) + c\eta^{-1}\alpha_{k_t}\langle\nabla f(x^{k_t}), d^{k_t}\rangle \; \text{for every}\; t\geq T.
\]
Dividing both sides of the above inequality by $ \alpha_{k_t}/\eta $ and rearranging terms, we obtain
\[
\frac{f(x^{k_t}+\eta^{-1}\alpha_{k_t}d^{k_t}) - f(x^{k_t})}{\eta^{-1}\alpha_{k_t}} > c\langle\nabla f(x^{k_t}), d^{k_t}\rangle  \; \text{for every}\; t\geq T.
\]
Passing to the limit and rearranging terms in the above display, we have
\[
(1-c)\langle\nabla f(x^*), d^*\rangle \geq 0.
\]
Since $ c\in(0, 1) $, we see further that
\[
\lim_{t\to\infty} \langle\nabla f(x^{k_t}), d^{k_t}\rangle = \langle\nabla f(x^*), d^*\rangle \geq 0.
\]
On the other hand, we have $\lim_{t\to\infty} \langle\nabla f(x^{k_t}), d^{k_t}\rangle  \leq 0$
because $ \langle\nabla f(x^k), d^k\rangle < 0$ for every $ k $. Thus, we conclude that $ \lim_{t\to\infty} \langle\nabla f(x^{k_t}), d^{k_t}\rangle = 0 $. 
\qed \end{proof}

{\color{blue}Before ending this section, we briefly review the classic FW algorithm (i.e., for \eqref{P00} with convex $ C $) and its variants, discuss some recent extensions of it in nonconvex setting, and describe the basic idea of our approach.
	
	The FW algorithm with Armijo line search is presented in Algorithm~\ref{FW_convex}. In each iteration, one solves the \emph{linear-optimization oracle} \eqref{convexLO} and searches for the next iterate along the \emph{feasible direction} $ u^k-x^k $. The algorithm terminates early when the so-called \emph{duality gap} $ \langle \nabla f(x^k), u^k-x^k\rangle $ becomes zero, meaning that $x^k$ is a stationary point of \eqref{P00} with convex $ C $; otherwise, an infinite sequence is generated whose subsequential convergence can be deduced from Lemma~\ref{Armijolemma} as discussed in \cite[Proposition~2.2.1 and Section~2.2.2]{Bertsekas99}. Efficient implementations of the FW algorithm for applications such as matrix completion with {\color{blue} $ C=\{x\in\R^{m\times n}:\,\|x\|_*\leq \sigma \}$, where $\|x\|_*$ denotes the nuclear norm of the matrix $x \in \R^{m \times n}$,} can be found in \cite{Pedregosa20,Jaggi10,FrGM17}. An important variant of FW algorithm, the so-called FW algorithm with \emph{away {\color{blue}steps}} was proposed in \cite{Wolfe} and further studied in \cite{Guelat86} to tackle the possible zigzag behavior of the original FW iterates.
	For other variants of FW algorithm for \eqref{P00} with convex $ C $, we refer the readers to \cite{Jaggi13,Lacoste15,FrGM17} and the recent comprehensive survey \cite{BRZ21}.
}

\begin{algorithm}[H]
	\captionsetup{labelfont={color=blue},font={color=blue}}
		\caption{\ Frank-Wolfe algorithm for \eqref{P00} with convex $ C $}\label{FW_convex}
		{\color{blue}
	\begin{algorithmic}
		\State \vspace{-0.15 cm}
		\begin{description}
			\item[\bf Step 0.] Choose $x^0\in C$, $c\in (0, 1)$. Set $k = 0$.  \vspace{0.1 cm}
			\item[\bf Step 1.] Compute
			\begin{equation}\label{convexLO}
				u^k\in\Argmin_{x\in C} \langle\nabla f(x^k), x\rangle.
			\end{equation}	
				
			\noindent If $\langle\nabla f(x^k), u^k-x^k\rangle = 0$, terminate.	 \vspace{0.1 cm}
			\item[\bf Step 2.] Find $ \alpha_k\in(0, 1] $ {\color{blue} via} backtracking (starting from $1$) to satisfy the Armijo rule
			\[
			f(x^k + \alpha_k (u^k-x^k)) \le f(x^k) + c\alpha_k\langle\nabla f(x^k),  u^k-x^k \rangle.
			\]	
			\item[\bf Step 3.] Set $ x^{k+1} = x^k + \alpha_k  (u^k-x^k)$. Update $k \leftarrow k+1$ and go to Step 1.
		\end{description}
	\end{algorithmic}
}
\end{algorithm}

{\color{blue}
Compared with the vast literature on the FW algorithm and its variants for \eqref{P00} with $ C $ being convex, there are only few works extending this kind of algorithm for \emph{nonconvex} {\color{blue} sets} $ C $. As mentioned in the introduction, a notable difficulty of extending the FW algorithm for nonconvex $ C $ is that the update rule in Step 3 may lead to an infeasible $x^{k+1}$. Despite this, the FW algorithm has been extended in \cite{RT13,BalashovPolyak20} to some nonconvex settings; these works identified conditions on $f$ and $C$ to guarantee that $\alpha_k = 1$ in Step 2. However, they required $ C $ to be (a certain subset of) the boundary of a strongly convex set, which can be restrictive for applications.

In this paper, we take a different approach to extend the FW algorithm to some nonconvex settings. In particular, we focus on a broad class of nonconvex {\color{blue} feasible sets} $ C $ {\color{blue} which is given as the} level set of a difference-of-convex function, i.e., $ \mathcal{F} $ in \eqref{P0}. For this class of sets, we develop in the next section a \emph{new} linear-optimization oracle (Definition~\ref{LOdef}). Specifically, at a $y\in {\cal F}$ {\color{blue} and $\xi \in \partial P_2(y)$}, we replace $ \mathcal{F} $ by its {\em subset} $ \mathcal{F}(y,\xi)=\{x\in \X:\; P_1(x) - P_2(y)-\langle\xi, x-y\rangle\leq \sigma\} $. {\color{blue} We then define} the linear-optimization oracle as minimizing a suitable linear objective function over $\mathcal{F}(y,\xi)$. If $u$ is an output of such an oracle,
then $u\in\mathcal{F}(y,\xi)$.
Since we also have $ \mathcal{F}(y,\xi)$ being convex and $y\in\mathcal{F}(y,\xi)\subseteq \mathcal F$, we deduce that $ y + \alpha(u-y)\in\mathcal{F}(y,\xi)\subseteq \mathcal F $ for any $ \alpha\in[0,1] $, hence overcoming the \emph{infeasiblity issue}. We showcase scenarios when our new linear-optimization oracles can be efficiently {\color{blue} computed}; see Sections~\ref{gl} and \ref{mc}. We impose further regularity conditions on the representing functions $P_1$ and $P_2$ in \eqref{P0} in Section~\ref{sec4} so that the ``finite termination criterion" $\langle\nabla f(x^k), u^k-x^k\rangle = 0$, with the $u^k$ obtained from our new linear-optimization oracle, can still be connected to the notion of stationarity of \eqref{P0} in Definition~\ref{Stationary}. Equipped with the new linear-optimization oracle and the regularity conditions, we then present our FW-type algorithm for solving \eqref{P0} (with its $P_1$ and $P_2$ satisfying the aforementioned regularity conditions) in Section~\ref{sec:convergence}, where we also establish its well-definedness and subsequential convergence. Finally, we construct in Section~\ref{sec6} a new away-step oracle (Definition~\ref{awdef}) for \eqref{P0} by mimicking the away steps used in the atomic version of the FW algorithm with away steps \cite{Lacoste13}.
}



\section{Linear-optimization oracles: Examples}\label{sec:LO}

Recall that the FW method \cite{FrankWolfe56,DeRu70,Jaggi13,GarberHazan16} can be efficiently employed in some instances of \eqref{P00} when $C$ is \emph{convex} but the projections onto $C$ are difficult {\color{blue}or too expensive to compute}. Concrete examples of such $C$ include the $\ell_p$ norm ball constraint when $p\in(1, \infty)\setminus\{2\}$, the nuclear norm ball constraint that arises in matrix completion problems for recommender systems \cite{FrGM17}, and the total-variation norm ball adopted in image reconstruction tasks \cite{HaJN15}; for these examples, the so-called linear-optimization oracle can be carried out efficiently. Such oracle is used in each iteration of the FW method to generate a test point, and the next iterate of the FW method is obtained as a suitable \emph{convex combination} of the current iterate and the test point. Note that {\color{blue}the \emph{convexity} of $C$} in \eqref{P00} is crucial here so that the next iterate stays feasible. Since the constraint set of \eqref{P0} can be nonconvex in general, it appears that the classic FW method described above cannot be directly applied to solve \eqref{P0}.

%

As a first step towards developing {\color{blue}FW-type methods} for \eqref{P0}, let us define a (new) notion of linear-optimization oracle for the possibly nonconvex constraint set in \eqref{P0}. We will then discuss how to solve the linear-optimization oracles that correspond to some concrete applications. Our new FW-type methods for \eqref{P0} based on this new notion of linear-optimization oracles will be presented as Algorithms~\ref{FW_nonconvex} and \ref{alg:aw-fw} in Sections~\ref{sec:convergence} and \ref{sec6}, respectively.
\begin{definition}[{Linear-optimization oracle}]\label{LOdef}
Let $P_1$, $P_2$ and $\sigma$ be defined in \eqref{P0}, $y\in \mathcal{F}$, $\xi\in\partial P_2(y)$, and define
\begin{equation}\label{Fyxi}
{\cal F}(y,\xi) := \{x\in \X:\; P_1(x) - \langle \xi, x - y\rangle - P_2(y) \leq \sigma\}.
\end{equation}
Let $a\in \X$. A linear-optimization oracle for $(a, y, \xi)$ (denoted by $\LO(a, y, \xi)$ or $\LO$ for brevity) computes a solution of the following problem
\begin{equation}\label{LO}
  \begin{array}{rl}
\min\limits_{x\in\X} & \langle a, x\rangle\ \ \ \ {\rm s.t.}\ \  x\in {\cal F}(y,\xi).
  \end{array}
\end{equation}
\end{definition}
\begin{remark}[{Well-definedness of $\LO$}]\label{welldefined}
Notice that problem \eqref{LO} is well-defined. Indeed, given $y\in \mathcal{F}$ and $\xi\in\partial P_2(y)$, we have
\begin{equation}\label{constrain}
\begin{aligned}
y\in {\cal F}(y,\xi) &= \{x\in \X:\; P_1(x) - P_2(y) - \langle \xi, x- y\rangle \leq \sigma\}\\
&\subseteq \{x\in \X:\; P_1(x) - P_2(x) \leq \sigma\} = {\cal F}\subseteq B(0,M)
\end{aligned}
\end{equation}
{\color{blue} for some $M>0$}, where the first set inclusion follows from the convexity of $P_2$ and the fact that $\xi\in\partial P_2(y)$, and the second set inclusion holds due to the compactness of $\cal F$. Thus, problem \eqref{LO} is minimizing a linear objective function over a compact nonempty constraint set. Hence, its set of optimal solutions is nonempty.
\end{remark}

We now present in the following subsections some concrete examples of $\LO$ (in the sense of Definition~\ref{LOdef}) that can be carried out efficiently. Our first two examples arise from sparsity inducing problems (group sparsity) and the matrix completion problem, respectively. Our third example concerns the case when $P_1$ is strongly convex, where the $\LO$ can be shown to be related to the computation of the proximal mapping of a suitable function.

\subsection{Group sparsity}\label{gl}
In this subsection, we let $ \X =\R^n $ and let $x_J$ denote the subvector of $ x \in \R^n$ indexed by $ J $, where $ J\subseteq \{1,\ldots n\} $. We consider $P_1$ and $P_2$ as in the following assumption and discuss the corresponding $\LO$.
\begin{assumption}\label{assump31}
Let $ \X =\R^n $ and $\mathcal{J}$ be a partition of $\{1,2,\ldots, n\}$. Let $P_1(x) = \sum_{J\in\mathcal{J}}\|x_J\|$ and $P_2$ be a norm such that $P_2\le \mu P_1$ for some $\mu\in[0,1)$.
\end{assumption}
Notice that the choice of $P_1$ and $P_2$ in Assumption~\ref{assump31} (together with $\sigma > 0$) ensures that the constraint set $\{x\in \X:\; P_1(x) - P_2(x)\le \sigma\}$ is compact and nonempty. Thus, in view of Remark~\ref{welldefined}, the corresponding $\LO$ is well-defined.
The choice of $P_1$ in Assumption~\ref{assump31} is known as the group LASSO regularizer \cite{YuLi06}. As in \cite{YinLouHeXin15}, here we consider a natural extension that subtracts a norm $P_2$ from the group LASSO regularizer. An example of $P_2$ satisfying Assumption~\ref{assump31} is $P_2(x)=\mu\|x\|$ with $\mu \in [0,1)$.

With $P_1$ and $P_2$ in \eqref{P0} chosen as in Assumption~\ref{assump31}, for any given $a\in\R^n$, any $y\in {\cal F} $ and any $\xi\in\partial P_2(y)$, it holds that $P_2(y) = \langle\xi,y\rangle$ and hence the corresponding $\LO(a, y, \xi)$ solves a problem of the following form:
\begin{equation}\label{sgl}
\begin{array}{rl}
\min\limits_{x\in\R^n} & \displaystyle \langle a, x\rangle \ \ \ \
{\rm s.t.} \ \ \displaystyle \sum\limits_{J\in\mathcal{J}}\|x_J\| - \langle\xi, x\rangle\leq \sigma.
  \end{array}
\end{equation}
We next derive a closed form formula that describes an output of $\LO(a, y, \xi)$. To proceed with our derivation, we first establish the following lemma.
\begin{lemma}\label{gruopsub}
Let $b\in \R^p\backslash \{0\}$ and $c\in \R^p$ with $\|c\|<1$. Consider 
\begin{equation}\label{sgleq33}
\kappa: = \min\limits_{\|x\| = 1, \, {\color{blue} x \in \R^p}}  \frac{\langle b, x\rangle}{1 - \langle c, x\rangle}.
\end{equation}
Then $\kappa < 0$ and it holds that
\begin{equation*}
c - \frac{\langle b, c\rangle + \sqrt{\langle b, c\rangle^2 - \|b\|^2(\|c\|^2 - 1)}}{\|b\|^2}b \ \in\  \Argmin\limits_{\|x\| = 1}  \frac{\langle b, x\rangle}{1 - \langle c, x\rangle}.
\end{equation*}
\end{lemma}
\begin{proof}
First note that \eqref{sgleq33} is equivalent to the following problem
\begin{equation}\label{sgleq34}
\min\limits_{\|x\| = 1}  \frac{\langle b, x\rangle}{\|x\| - \langle c, x\rangle}.
\end{equation}
Since $\|c\| < 1$, we must have $\|x\| \pm \langle c, x\rangle > 0$ whenever $\|x\|=1$. Thus, the optimal value $\kappa$ of the above optimization problem must be negative.\footnote{Since $b\neq 0$, the objective in \eqref{sgleq34} is negative at $x = -\frac{b}{\|b\|}$.}

Since the objective of \eqref{sgleq34} is (positively) $0$-homogeneous, we see that an $x^*$ solves \eqref{sgleq34} if and only if $x^* = \frac{x_1^*}{\|x_1^*\|}$, where
\begin{equation}\label{sgleq4}
x_1^*\in \Argmin\limits_{x_1\in\R^p\setminus\{0\}}  \frac{\langle b, x_1\rangle}{\|x_1\| - \langle c, x_1\rangle}.
\end{equation}
Using the (positive) $0$-homogeneity of the objective function in \eqref{sgleq4}, one can see further that an optimal solution $x_2^*\in\R^p$ of the following problem must be optimal for \eqref{sgleq4}:
\begin{equation}\label{sgleq5}
\begin{array}{rl}
\min\limits_{x_2\in\R^p} & \langle b, x_2\rangle\ \ \ \ {\rm s.t.} \ \ \|x_2\| - \langle c, x_2\rangle = 1.
\end{array}
\end{equation}
We now turn to solving \eqref{sgleq5}. For any $x_2$ satisfying $\|x_2\| - \langle c, x_2\rangle = 1$, we have $x_2 \neq 0$ and $\frac{x_2}{\|x_2\|} - c \neq 0$ (thanks to $\|c\|<1$). Thus, the LICQ holds for \eqref{sgleq5}. Hence, for any optimal solution $x_2^*$ of \eqref{sgleq5},\footnote{Since $\|c\| < 1$, the feasible set of \eqref{sgleq5} is compact and nonempty. This implies that the set of optimal solutions is nonempty.} there exists $\lambda_*\in\R$ such that
\[
0= b + \lambda_*\left(\frac{x_2^*}{\|x_2^*\|} - c\right).
\]
Since $b\not= 0$, we have that $\lambda_*\not= 0$ and hence $\frac{x_2^*}{\|x_2^*\|} = c - \lambda_*^{-1}b$, which further gives
\begin{equation}\label{findlambda}
  \|c - \lambda^{-1}_*b\|^2 = 1.
\end{equation}
Next, since $c - \lambda^{-1}_*b$ is a positive rescaling of $x_2^*$, we conclude using the relations between \eqref{sgleq34}, \eqref{sgleq4} and \eqref{sgleq5} that $c - \lambda^{-1}_*b$ is an optimal solution of \eqref{sgleq34}. Since $\kappa < 0$, we must then have
\[
0>\langle b, c - \lambda^{-1}_*b\rangle = \langle b, c\rangle - \lambda^{-1}_*\|b\|^2 \quad\Longrightarrow\quad \lambda_*^{-1} > \frac{\langle b, c\rangle}{\|b\|^2}.
\]
Using this together with \eqref{findlambda}, we see that $\lambda^{-1}_*$ must be the larger solution of the quadratic equation $\|c - tb\|^2 = 1$. Solving this quadratic equation, we obtain that
\[
\lambda_*^{-1} = {\color{blue} \frac{\langle b, c\rangle + \sqrt{\langle b, c\rangle^2 - \|b\|^2(\|c\|^2 - 1)}}{\|b\|^2}}.
\]
Combining this with the fact that $c - \lambda^{-1}_*b$ is an optimal solution of \eqref{sgleq34} (and hence \eqref{sgleq33}) completes the proof.
\qed \end{proof}

We now present a closed form solution of problem~\eqref{sgl}. For notational convenience, we define for any $x\in \R^p$ the following sign function,
\begin{equation}\label{sign}
\sgn(x):= \begin{cases}
\frac{x}{\|x\|} & {\rm if}\ x\not=0,\\
\frac{e}{\|e\|} & {\rm if}\ x = 0,
\end{cases}
\end{equation}
where $e$ is the vector of all ones of dimension $p$. {\color{blue} Note that $x = \|x\|{\rm Sgn}(x)$ for all $x\in \R^p$. The proof below makes use of this relation to re-scale the constraint function in \eqref{sgl} and decouple the problem into optimization problems that involve the modulus and the ``angle" respectively; Lemma~\ref{gruopsub} is then invoked to handle the latter optimization problem.}

\begin{theorem}[{Closed form solution for \eqref{sgl}}]
  Consider \eqref{sgl} with $a\neq 0$. For each $J\in {\cal J}$, fix any vector $v_J\in \R^{\lvert J\rvert}$ with $\|v_J\|=1$ and define
\begin{equation}\label{wstar}
w_J^*=
\begin{cases}
  v_J & {\rm if}\  a_J = 0,\\
  \xi_J - \frac{\langle a_J, \xi_J\rangle + \sqrt{\langle a_J, \xi_J\rangle^2 - \|a_J\|^2(\|\xi_J\|^2 - 1)}}{\|a_J\|^2}a_J & {\rm otherwise},
\end{cases}
\end{equation}
{\color{blue}where $\xi$ is defined as in \eqref{sgl}.}
Let $\mathcal{I}:= \Argmin_{J\in {\cal J}}\{\kappa_J\}$, where
  \begin{equation}\label{sgleq3}
\kappa_J := \min\limits_{\|w_J\| = 1}  \frac{\langle a_J, w_J\rangle}{1 - \langle \xi_J, w_J\rangle},
\end{equation}
and fixed any $J_0\in\mathcal{I}$.\footnote{These $\kappa_J$ are readily computable. Indeed, as we will point out later in the proof of this theorem, the $w^*_J$ in \eqref{wstar} solves the minimization problem in \eqref{sgleq3}, thanks to Lemma~\ref{gruopsub}.} Then a solution $x^*= (x^*_J)_{J\in {\cal J}}$ of \eqref{sgl} is given by
\begin{equation}\label{hahahahaha}
x^*_J = \begin{cases}
  \frac{\sigma w_{J_0}^*}{1 - \langle \xi_{J_0}, w_{J_0}^*\rangle} & {\rm if}\ J = J_0,\\
  0 & {\rm otherwise}.
\end{cases}
\end{equation}
\end{theorem}
\begin{proof}
The existence of optimal solutions to \eqref{sgl} follows from Remark~\ref{welldefined}.
Also, for each $J\in {\cal J}$, we have $\|\xi_J\|\leq \mu < 1$ (thanks to the fact that $P_2$ is a norm, $P_2\leq\mu P_1$ and $\mu< 1$). In addition, from \eqref{sign}, then one can check readily that $\|\sgn(x_J)\| = 1$ and $x_J = \|x_J\|\sgn(x_J)$. In particular, we have $1 - \langle\xi_J, \sgn(x_J)\rangle > 0$.

For each $J\in\mathcal{J}$, let $u_J = (1 - \langle\xi_J, \sgn(x_J)\rangle)x_J$. Then we have $\|u_J\| = \big(1 - \langle\xi_J, \sgn(x_J)\rangle\big)\|x_J\|$ and $\sgn(u_J) = \sgn(x_J)$. Thus, we obtain the following equivalences:
\begin{equation*}
\begin{split}
&\sum\limits_{J\in\mathcal{J}}(\|x_J\| \!-\! \langle\xi_J, x_J\rangle)\!\leq\! \sigma
\!\Longleftrightarrow\!\! \sum\limits_{J\in\mathcal{J}}\!(1 \!-\! \langle\xi_J, \sgn(x_J)\rangle)\|x_J\|\!\leq\! \sigma
\!\Longleftrightarrow\!\! \sum\limits_{J\in\mathcal{J}}\! \|u_J\|\! \leq \sigma.
\end{split}
\end{equation*}
Hence, an $\widehat x$ solves \eqref{sgl} if and only if $\widehat x_J = {\color{blue} \frac{\widehat u_J}{1 - \langle\xi_J, \sgn(\widehat u_J)\rangle}}$ for each $J\in\mathcal{J}$, where
\begin{equation}\label{sgleq1}
\widehat u\in\Argmin\limits_{\sum\limits_{J\in\mathcal{J}}\|u_J\| \leq \sigma} \sum\limits_{J\in\mathcal{J}}\left\langle a_J, \frac{u_J}{1 - \langle\xi_J, \sgn(u_J)\rangle} \right\rangle
\end{equation}


We next discuss how to find such a $\widehat u$. Notice that any $u_J\in\R^{\lvert J\rvert}$ can be written as $u_J = r_Jw_J$ for some $r_J\geq 0$ and $w_J\in \R^{\lvert J\rvert}$ satisfying $\|w_J\| = 1$. Consequently, $\widehat u$ satisfies \eqref{sgleq1} if and only if $\widehat u_J = \widehat r_J\widehat w_J$, for each $J\in\mathcal{J}$, where
\begin{equation}\label{sgleq2}
\begin{aligned}
(\widehat r, \widehat w)\in \Argmin\limits_{(r, w)\in\R^{\lvert \mathcal{J}\rvert}\times\R^n}& \displaystyle\sum\limits_{J\in\mathcal{J}} \left\langle a_J, \frac{ w_J}{1 - \langle\xi_J, w_J\rangle} \right\rangle r_J\\ 
  {\rm s.t.} & \displaystyle\sum_{J\in {\cal J}}r_J \le \sigma,\\
  & \displaystyle r_J\geq 0, \ \ \|w_J\| = 1, \  \ \forall J\in\mathcal{J}.
\end{aligned}
\end{equation}
To solve \eqref{sgleq2}, we start with the minimization with respect to $w_J$. This amounts to solving the optimization problems \eqref{sgleq3} for each $J\in\mathcal{J}$. There are two cases:
\begin{itemize}
  \item If $a_J = 0$, then $\kappa_J = 0$ and the minimum in \eqref{sgleq3} is achieved at any feasible $w_J$; in particular, one can take $\widehat w_J = v_J$ (with $v_J$ as in \eqref{wstar}).
  \item If $a_J \not= 0$, we can apply Lemma~\ref{gruopsub} with $p := \lvert J\rvert$, $b := a_J$ and $c := \xi_J$, to deduce that $\kappa_J < 0$ and a minimizer is given as in \eqref{wstar}.
\end{itemize}
Thus, one can take $\widehat w _J = w^*_J$ in \eqref{wstar} as a minimizer (with respect to $w_J$) in \eqref{sgleq2}.

Now, to solve \eqref{sgleq2}, it remains to consider the following problem to find an $\widehat r$:
\begin{equation}\label{sgleq6}
\begin{array}{rl}
  \min\limits_{r\in\R^{\lvert\cal{J}\rvert}} & \displaystyle\sum\limits_{J\in\mathcal{J}} \kappa_Jr_J\ \ \ \ \ \ {\rm s.t.}\ \ \ \displaystyle\sum_{J\in {\cal J}}r_J \le \sigma,\ \ r_J \ge 0,\ \ \forall J\in {\cal J}.
\end{array}
\end{equation}
Recall that $\mathcal{I}= \Argmin_{J}\{\kappa_J\}$ and $J_0$ is a fixed element in $\mathcal{I}$. Since $\kappa_{J} \leq 0$ for any $J\in\mathcal{J}$, we see that $r^* = (r^*_J)_{J\in {\cal J}}$ defined below is an optimal solution of \eqref{sgleq6}:
\begin{equation}\label{rstar}
r^*_J = \begin{cases}
  \sigma & {\rm if}\ J = J_0,\\
  0 & {\rm otherwise}.
\end{cases}
\end{equation}
Hence, a solution to \eqref{sgleq2} is given by $\widehat r = r_J^*$ and $\widehat w = w_J^*$, for each $J$, defined in \eqref{rstar} and \eqref{wstar} respectively.

Finally, invoking the relationship between \eqref{sgleq1} and \eqref{sgleq2}, we see that $\widehat u$ with $\widehat u_J:= r^*_Jw^*_J$ for each $J$ solves \eqref{sgleq1}. Recall that \eqref{sgl} and \eqref{sgleq1} are related via $u_J = (1 - \langle\xi_J, \sgn(x_J)\rangle)x_J$ and $x_J = \frac{1}{1 - \langle \xi_J, \sgn(u_J)\rangle}u_J$ for all $J$. Thus, we conclude that an optimal solution $x^*= (x^*_J)_{J\in {\cal J}}$ of \eqref{sgl} is given by:
\begin{equation*}
x^*_J = \frac{1}{1 - \langle \xi_J, \sgn(r_J^*w^*_J)\rangle}r_J^*w^*_J = \begin{cases}
  \frac{\sigma w_{J_0}^*}{1 - \langle \xi_{J_0}, w_{J_0}^*\rangle} & {\rm if}\ J = J_0,\\
  0 & {\rm otherwise}.
\end{cases}
\end{equation*}
This completes the proof.
\qed \end{proof}

\begin{remark}[Element-wise sparsity]\label{cssparsity}
On passing, we discuss an interesting special case where every $J\in\mathcal{J}$ is a singleton. In this case, we have $P_1 = \|\cdot\|_1$ and $P_2:\R^n\rightarrow\R$ is a norm satisfying $P_2\le \mu P_1$ for some $\mu\in[0,1)$. Such a setting arises in compressed sensing \cite{LouYan18,YinLouHeXin15}, where $P_2$ can be chosen as $P_2 = \mu\|\cdot\|$ for some $\mu \in [0,1)$, resulting in the difference of $\ell_1$ and (a positive multiple of) $\ell_2$ norm regularizer. A closed form formula that describes an output of the corresponding $\LO(a, y, \xi)$ can be readily deduced from \eqref{hahahahaha} as follows, upon invoking \eqref{wstar} and \eqref{rstar}:
\begin{equation}\label{xstar}
x^*_i=\begin{cases}
\displaystyle\frac{-\sigma{\rm Sgn}(a_{i_0})}{1 + \xi_{i_0}{\rm Sgn}(a_{i_0})} & i=i_0,\\
\displaystyle 0 & i\neq i_0,
\end{cases}
\end{equation}
where $ \mathcal{I}:= \Argmin_{i}\left\{ \frac{-\lvert a_i\rvert }{1 + \xi_i {\rm Sgn}(a_i)} \right\} $ with $ i_0 $ being any fixed element of $\mathcal{I} $, and ${\rm Sgn}$ is defined in \eqref{sign}.
\end{remark}

\subsection{Matrix completion}\label{mc}
In this subsection, we let $\X = \R^{m\times n}$. For an $ {\color{blue}x}\in\X $, we denote by $ \|{\color{blue}x}\|_* $ and $ \|{\color{blue}x}\|_F $ its nuclear norm and Frobenius norm, respectively. We consider $P_1$ and $P_2$ as in the following assumption and discuss the corresponding $\LO$.
\begin{assumption}\label{assump32}
Let $\X = \R^{m\times n}$, $P_1({\color{blue}x}) = \|{\color{blue}x}\|_*$ and $P_2$ be a norm function such that $P_2\le \mu P_1$ for some $\mu\in[0,1)$.
\end{assumption}
Observe that under the choice of $P_1$ and $P_2$ in Assumption~\ref{assump32} (together with $\sigma > 0$), the set $\{{\color{blue}x}\in \X:\; P_1({\color{blue}x}) - P_2({\color{blue}x})\le \sigma\}$ is nonempty and compact. Hence, the corresponding $\LO$ is well-defined thanks to Remark~\ref{welldefined}.
An example of $P_2$ is $P_2({\color{blue}x})= \mu\|{\color{blue}x}\|_F$ with $\mu \in [0,1)$. In this case, the regularization $ P({\color{blue}x})=\|{\color{blue}x}\|_* - \mu\|{\color{blue}x}\|_F $ has been used in low rank matrix completion, which can be viewed as an extension of the $ \ell_{1-2} $ {\color{blue} regularizer used} in compressed sensing. Exact and stable recovery {\color{blue}conditions} and numerical advantages of this class of nonconvex nonsmooth {\color{blue} regularizer} with $ \mu=1 $ are discussed in \cite{MaLou17}.

Now, with $P_1$ and $P_2$ in \eqref{P0} chosen as in Assumption~\ref{assump32}, for any given ${\color{blue}a}\in\R^{m\times n}$, any ${\color{blue}y}\in {\cal F} $ and any ${\color{blue}\xi}\in \partial P_2({\color{blue}y})$, it holds that $P_2({\color{blue}y}) = \langle{\color{blue}\xi,y}\rangle$ and thus the corresponding $\LO({\color{blue}a, y, \xi})$ solves a problem of the form:
\begin{equation}\label{lo-mat}
\min\limits_{{\color{blue}x}\in\R^{m\times n}} \langle {\color{blue}a, x }\rangle \quad \text{s.t.}\; \|{\color{blue}x}\|_* -\langle {\color{blue}\xi, x}\rangle \leq \sigma.
\end{equation}We next present a closed form solution of \eqref{lo-mat}. For notational simplicity, we write\footnote{\color{blue}In this subsection, for notational clarity, we use upper case letters to denote square matrices of size $(m+n)\times(m+n)$, and use lower case letters to denote matrices and vectors of other dimensions.}
\begin{equation}\label{AandXi}
  \widetilde{A} = \begin{bmatrix}
	0 & {\color{blue}a} \\ {\color{blue}a}^T & 0
	\end{bmatrix} {\color{blue} \in S^{m+n}}\ \ \ {\rm and}\ \ \ \widetilde{\Xi}= \begin{bmatrix}
	0 & {\color{blue}\xi} \\ {\color{blue}\xi}^T & 0
	\end{bmatrix} {\color{blue} \in S^{m+n}}.
\end{equation}
{\color{blue} The proof below involves the reformulation of \eqref{lo-mat} into a semidefinite programming (SDP) problem, and makes use of the well-known result \cite[Theorem~2.2]{Pataki98} concerning the rank of extreme points of the solution set of an SDP problem.}

\begin{theorem}\label{cfs-mat}
	Consider \eqref{lo-mat} with ${\color{blue}a}\neq 0$. Let $ z\in\R^{m+n} $ be a generalized eigenvector of the smallest generalized eigenvalue of the matrix pencil $ (\widetilde A, I - \widetilde\Xi) $, and satisfy $z^T(I - \widetilde\Xi)z = 1$, where $\widetilde A$ and $\widetilde\Xi$ are given in \eqref{AandXi}. Then $ {\color{blue}x}^*=2\sigma z_1z_2^T $ is an optimal solution of \eqref{lo-mat}, where $ z=\begin{bmatrix}
  z_1^T & z_2^T
\end{bmatrix}^T $ with $ z_1\in\R^m $ and $ z_2\in\R^n $.
\end{theorem}
\begin{proof}
	For notational convenience, we write
	\begin{equation}\label{Zstar}
	Z^*:= 2\sigma zz^T=\begin{bmatrix}
	{\color{blue}u}^* & {\color{blue}x}^* \\ ({\color{blue}x}^*)^T & {\color{blue}v}^*
	\end{bmatrix}\succeq 0,
	\end{equation}
	where $ u^* = 2{\sigma}z_1z_1^T $ and $ v^* =  2{\sigma}z_2z_2^T $.
	
	We first claim that $ Z^* $ is a solution of the following optimization problem:\footnote{Note that $I - \widetilde \Xi \succ 0$ because the spectral norm of {\color{blue}$\xi$} is at most $\mu < 1$. Thus, the feasible set of \eqref{auxp} is nonempty and compact, and hence the set of optimal solution is nonempty.}
	\begin{equation}\label{auxp}
	\min\limits_{Y\in{\mathcal S}^{m+n}_+}\; \langle \widetilde{A}, Y\rangle \quad \text{s.t.}\; \langle I-\widetilde{\Xi}, Y \rangle \leq 2\sigma.
	\end{equation}
	Indeed, by adding a slack variable $ \alpha\in\R_+ $, \eqref{auxp} can be written as
	\begin{equation}\label{auxp1}
	\min\limits_{Y\in{\mathcal S}^{m+n}_+, \alpha \geq 0}\; \langle \widetilde{A}, Y\rangle \quad \text{s.t.}\; \langle I-\widetilde{\Xi}, Y \rangle +\alpha = 2\sigma.
	\end{equation}
	Note that there is only one equality constraint in the above semidefinite programming problem, and there must be a solution $ (Y^*, \alpha_*) $ that is an extreme point of the feasible set of \eqref{auxp1} (as the solution set of \eqref{auxp1} does not contain a line). According to \cite[Theorem~2.2]{Pataki98}, the rank $ r_Y $ of $ Y^* $ and the rank $ r_\alpha $ of $\alpha_* $ satisfy
	\[
	r_Y(r_Y+1) + r_\alpha (r_\alpha +1 ) \leq 2.
	\]
	Since $ {\color{blue}a}\neq 0 $, we must have $ Y^*\neq 0 $. This fact together with the above display implies that $ r_Y = 1 $ and $r_{\alpha}=0$. Therefore, we can write the rank-1 solution $ Y^* $ of \eqref{auxp1} as $ Y^* = 2\sigma y^*(y^*)^T $ for some $ y^*\in\R^{m+n}$ that solves
	\begin{equation*}
	\min_{y\in\R^{m+n}} y^T\widetilde{A}y \quad \text{s.t.}\; y^T(I-\widetilde{\Xi})y = 1.
	\end{equation*}
	Such a $y^*$ can be obtained as a generalized eigenvector that corresponds to the smallest generalized eigenvalue of the matrix pencil $ (\widetilde A, I-\widetilde\Xi) $ and satisfies $ y^T(I-\widetilde{\Xi})y = 1 $. Now, recalling the definitions of $ z $ and $ Z^* $, we see that $ Z^* $ is a solution of \eqref{auxp}.
	
	We are now ready to prove that $ {\color{blue}x}^* $ is a solution of \eqref{lo-mat}. First, recall from \cite{ReFa10} that the nuclear norm of a matrix $ {\color{blue}x}\in\R^{m\times n} $ can be represented as:
	\begin{equation}\label{nucn}
	\!\!\!\!\!\|{\color{blue}x}\|_* = \min\limits_{{\color{blue}u, v}} \left\{\frac12\left(
	{\rm tr}({\color{blue}u})+{\rm tr}({\color{blue}v})\right):\; \begin{bmatrix}
	{\color{blue}u} & {\color{blue}x} \\ {\color{blue}x}^T & {\color{blue}v}
	\end{bmatrix}\succeq 0,\; {\color{blue}u}\in S^{m},\; {\color{blue}v}\in S^n
	\right\}.\footnote{We would like to point out that the minimum is attainable according to the discussions in \cite{ReFa10}.}
	\end{equation}
	One can then deduce that
	\begin{equation}\label{feasXstar}
	\begin{split}
	&\|{\color{blue}x}^*\|_*-\langle {\color{blue}\xi}, {\color{blue}x}^* \rangle \overset{\rm (a)}\leq \frac12\left({\rm tr}({\color{blue}u}^*)+{\rm tr}({\color{blue}v}^*)\right) - \langle {\color{blue}\xi}, {\color{blue}x}^*\rangle\\
	& \overset{\rm (b)}=
	\frac12{\rm tr}(Z^*)-\frac12 \langle \widetilde{\Xi}, Z^*\rangle = \frac12 \langle I-\widetilde{\Xi}, Z^*\rangle \leq
	\sigma,
	\end{split}
	\end{equation}
	where (a) follows from \eqref{nucn} and the definition of ${\color{blue}u}^*$ and ${\color{blue}v}^*$ in \eqref{Zstar}, and (b) uses the definitions of $ \widetilde\Xi $ and $ Z^* $. This shows that ${\color{blue}x}^*$ is feasible for \eqref{lo-mat}.
	
	Next, for any $ {\color{blue}\widehat x}\in\R^{m\times n} $ satisfying $ \| {\color{blue}\widehat x}\|_*-\langle {\color{blue}\xi}, {\color{blue}\widehat x}\rangle\leq \sigma $, we define $ ( {\color{blue}\widehat u}, {\color{blue}\widehat v}) $ as the minimizer in \eqref{nucn} corresponding to $\|{\color{blue}\widehat x}\|_*$.
	Let
\[{\color{blue}\widehat Z}=\begin{bmatrix}
	{\color{blue}\widehat u} & {\color{blue}\widehat x} \\ {\color{blue} \widehat x^T} & {\color{blue}\widehat v}
	\end{bmatrix}.
 \]
 We can check directly that $ {\color{blue}\widehat Z} $ is feasible for \eqref{auxp} by using the feasibility of $ {\color{blue}\widehat x} $ for \eqref{lo-mat} and the definitions of $ \widetilde A $ and $ \widetilde \Xi $. Then we have
	\[
	\langle {\color{blue}a}, {\color{blue}x}^*\rangle = \frac12 \langle \widetilde{A}, Z^*\rangle \overset{\rm (a)}\leq \frac12\langle \widetilde{A}, {\color{blue}\widehat{Z}}\rangle \overset{\rm
		(b)}= \langle {\color{blue}a}, {\color{blue}\widehat{x}}\rangle,
	\]
	where (a) uses the optimality of $ Z^* $ and the feasibility of $ {\color{blue}\widehat Z }$ for \eqref{auxp}, and (b) uses the definition of $ \widetilde{A} $. This together with \eqref{feasXstar} shows that $ {\color{blue}x}^*=2 \sigma z_1z_2^T $ solves \eqref{lo-mat}.
\qed \end{proof}

\begin{remark}
To obtain a closed form solution of \eqref{lo-mat}, we {\color{blue} need} to compute a generalized eigenvector $ z $ as shown in Theorem~\ref{cfs-mat}. Noticing that $I - \widetilde \Xi \succ 0$ (thanks to the fact that the spectral norm of ${\color{blue}\xi}$ is at most $\mu < 1$), such a generalized eigenvector $z$ can be found efficiently by  {\sf eigifp} \cite{YeGolub02}: {\sf eigifp} is an iterative solver based on Krylov subspace methods and only requires matrix vector multiplications $ \widetilde{A}{\color{blue}z} $ and $ \widetilde{\Xi}{\color{blue}z}$ in each iteration.
\end{remark}

\subsection{The case where $P_1$ is strongly convex}\label{sec:scsubp}
In this subsection, we assume that $ P_1 $ in \eqref{P0} is strongly convex with modulus $ \rho>0 $. We will argue that the corresponding $\LO$ involves a linear-optimization problem over a strongly convex set. Also, under suitable constraint qualifications, its solution involves computation of proximal mapping.

To this end, consider any given $ a\in\X\backslash\{0\} $, any $ y \in {\cal F} $ and any $ \xi\in\partial P_2(y) $.  Then the corresponding $ \LO(a, y, \xi) $ solves a problem of the following form:
\begin{equation}\label{LOsc}
	\begin{array}{cl}
	\min\limits_{x\in\X} & \langle a, x \rangle\ \ \ \ {\rm s.t.} \ \  \widetilde P_1(x) + \frac{\rho}{2}\|x\|^2 -\langle \xi, x\rangle \leq \widetilde{\sigma},
	\end{array}
\end{equation}
where $ \widetilde P_1 := P_1 -\frac{\rho}2\|\cdot\|^2 $ and $ \widetilde\sigma = \sigma + P_2(y) - \langle\xi, y\rangle $.
Moreover, $\widetilde P_1$ is convex.

Now, suppose that Slater's condition holds for the constraint set in \eqref{LOsc}.
Let $ x^* $ be a solution of \eqref{LOsc}. Since $a\neq 0$, we see from \cite[Corollary~28.1]{Ro70} and \cite[Theorem~28.3]{Ro70} that there exists $ \lambda_* > 0 $ 
such that
\begin{subnumcases}{}
	 x^* = \argmin_{x\in\X} \;\left\{\langle a, x\rangle + \lambda_*\left(\widetilde P_1(x) + \frac{\rho}2\|x\|^2- \langle \xi, x\rangle - \widetilde{\sigma}\right)\right\},	 \label{xlambda} \\
	P_1(x^*) - \langle \xi, x^*\rangle  = \widetilde{\sigma}. \label{eqconst}
\end{subnumcases}
Let $ \iota_* = \lambda_*^{-1} $. One can then deduce from \eqref{xlambda} that
\begin{equation}\label{xprox}
x^* = {\rm Prox}_{\frac 1{\rho}\widetilde P_1}\left(\frac{1}{\rho}(\xi - \iota_* a)\right),
\end{equation}
where $ {\rm Prox}_{g}(y):={\color{blue} \Argmin_{x\in \X}}\left\{ g(x) + \frac{1}{2}\|x-y\|^2\right\}$ is the proximal mapping of the proper closed function $ g $ at $ y $. Substituting the above expression into \eqref{eqconst}, we obtain a one-dimensional nonlinear equation in $ \iota_* $ as follows:
\begin{equation*}
P_1\left( {\rm Prox}_{\frac 1{\rho}\widetilde P_1}\left(\frac{1}{\rho}(\xi - \iota_* a)\right)\right) - \left\langle \xi,  {\rm Prox}_{\frac 1{\rho}\widetilde P_1}\left(\frac{1}{\rho}(\xi - \iota_* a)\right)\right\rangle  = \widetilde{\sigma}.
\end{equation*}
By standard root-finding procedures, one can solve for $ \iota_*>0 $. Then a solution $ x^* $ to \eqref{LOsc} can be obtained as \eqref{xprox}.

{\color{blue}
\begin{remark}\label{rem4}
We would like to point out that any DC function $ P_1-P_2 $ can be rewritten as the difference of strongly convex functions: Indeed, we trivially have for any $\rho > 0$ that $P_1 - P_2 = [P_1 + \frac\rho2\|\cdot\|^2] - [P_2 + \frac\rho2\|\cdot\|^2]$. Thus, the discussions in this section can be applied to the examples in Sections~\ref{gl} and \ref{mc} after transforming the DC functions therein to the difference of strongly convex functions. However, the $\LO$ involved will then require computing proximal mappings, which can be inefficient compared with the oracles described in Sections~\ref{gl} and \ref{mc}.
Specifically, as we shall discuss in Section~\ref{sec7} for the matrix completion problem, the corresponding output of \eqref{lo-mat} can be computed efficiently using {\sf eigifp} \cite{YeGolub02} even when $m$ and $n$ are huge, and can be chosen to have \emph{rank one}. This allows the use of the efficient SVD rank-one update technique \cite{Brand06} for updating the iterates. In contrast, computing the corresponding \eqref{LOsc} will involve the proximal operator in \eqref{xprox} with $ P_1$ being the nuclear norm in $ \R^{m\times n} $, which can be prohibitively expensive when $ m $ and $ n $ are huge; indeed, when $m$ and $n$ are huge, we may not even be able to explicitly form and compute the $\xi - \iota_* a$ in \eqref{xprox}, not to say to perform the root-finding procedure.
%
	\end{remark}
}

\section{Optimality condition}\label{sec4}
In this section, we discuss optimality conditions and define a stationarity measure for problem \eqref{P0} that are important for our algorithmic development later. We consider the following assumption.
\begin{assumption}\label{assumption1}
In \eqref{P0}, for any $y\in\mathcal{F}$ and $\xi\in\partial P_2(y)$, there exists $\xfeas_{(y, \xi)}\in\X$ so that the following holds:
\begin{equation}\label{assumine}
P_1(\xfeas_{(y, \xi)}) - P_2(y) - \langle\xi, \xfeas_{(y, \xi)} - y\rangle < \sigma.
\end{equation}
\end{assumption}

Note that Assumption~\ref{assumption1} holds in the examples described under Assumption~\ref{assump31} or \ref{assump32}. In fact, if we take $\xfeas_{(y, \xi)} = 0$ for any $y\in\mathcal{F}$ and $\xi\in\partial P_2(y)$, one can see that \eqref{assumine} holds for those $P_1$ and $P_2$. Moreover, it is interesting to note that Assumption~\ref{assumption1} depends on the choices of $P_1$ and $P_2$ in the DC decomposition of the constraint function. In contrast, the validity of gMFCQ is independent of the choices of $P_1$ and $P_2$.


We now study some relationships between gMFCQ and Assumption~\ref{assumption1}, and show in particular {\color{blue}that, under Assumption~\ref{assumption1},} every local minimizer of \eqref{P0} is a stationary point.
\begin{proposition}\label{slaequ}
Consider \eqref{P0}. Then the following statements hold:
\begin{enumerate}[{\rm (i)}]
	\item If Assumption~\ref{assumption1} holds, then the gMFCQ holds at every point in $\mathcal{F}$.
	\item If gMFCQ holds and $-P_2$ is regular at every point in $\{x\in\X:\; P_1(x) - P_2(x) = \sigma\}$, then Assumption~\ref{assumption1} holds. 
\end{enumerate}
\end{proposition}
\begin{remark}
From Proposition~\ref{loc-sta}, when the gMFCQ holds at every point in $\mathcal{F}$, we see that any local minimizer of \eqref{P0} is a stationary point of \eqref{P0}. Then we can deduce from Proposition~\ref{slaequ}(i) that any local minimizer of \eqref{P0} is a stationary point of \eqref{P0} when Assumption~\ref{assumption1} holds.
\end{remark}
\begin{proof}
(i): Suppose to the contrary that the gMFCQ fails at some $\bar{x}\in\mathcal{F}$, that is, there exists $\bar x$ with $P_1(\bar x) - P_2(\bar x) = \sigma$ but $0\in\partial^\circ (P_1 - P_2)(\bar x)$. This implies that $0\in\partial P_1(\bar x) - \partial P_2(\bar x)$, and hence there exists $\bar\xi\in\partial P_2(\bar x)$ satisfying $\bar\xi\in\partial P_1(\bar x)$. Moreover, by Assumption~\ref{assumption1}, there exists $\xfeas_{(\bar x, \bar\xi)}\in\X$ such that
\[
\sigma > P_1(\xfeas_{(\bar x, \bar\xi)}) - P_2(\bar x) - \langle\bar\xi, \xfeas_{(\bar x, \bar\xi)} - \bar x\rangle \overset{\rm (a)}\geq P_1(\bar x) - P_2(\bar x).
\]
where (a) follows from the convexity of $P_1$ and the fact that $\bar\xi\in\partial P_1(\bar x)$. The above display contradicts $P_1(\bar x) - P_2(\bar x) = \sigma$. Hence the gMFCQ holds at every point in $\mathcal{F}$.

(ii): Suppose to the contrary that Assumption~\ref{assumption1} fails. Then there exist $y\in {\cal F}$ and $\xi\in\partial P_2(y)$ such that for all $x\in\X$, $P_1(x) - P_2(y) - \langle\xi, x - y\rangle \geq \sigma$.
In particular, we have $P_1(y) - P_2(y) = P_1(y) - P_2(y) - \langle\xi, y - y\rangle \geq \sigma$, which together with $y\in {\cal F}$ implies
$P_1(y) - P_2(y) = \sigma$.

Since $P_1(y) - P_2(y) = \sigma$, we conclude that $y$ is a minimizer of the function $x\mapsto P_1(x) - P_2(y) - \langle\xi, x - y\rangle $. Then, we deduce from the first-order optimality condition that
\begin{equation*}
\begin{aligned}
0 &\in \partial P_1(y) - \xi\subseteq \partial P_1(y) - \partial P_2(y) \overset{\rm (a)}= \partial^\circ P_1(y) - \partial^\circ P_2(y) \\
& \overset{\rm (b)}= \partial^\circ P_1(y) + \partial^\circ(- P_2)(y) \overset{\rm (c)}= \partial^\circ(P_1 - P_2)(y),
\end{aligned}
\end{equation*}
where (a) follows from \cite[Theorem~6.2.2]{BoLe06}, (b) holds because of \cite[Proposition~2.3.1]{Cl90}, and (c) is true in view of Corollary~1 of \cite[Proposition~2.9.8]{Cl90} and the regularity properties of $-P_2$ (by assumption) and $P_1$ (thanks to convexity). The above display contradicts the gMFCQ. Thus, Assumption~\ref{assumption1} holds.
\qed \end{proof}


Next, we present equivalent characterizations of a stationary point of \eqref{P0}.
\begin{lemma}\label{equistat}
Consider \eqref{P0} and suppose that Assumption~\ref{assumption1} holds. Let $x^*\in {\cal F}$. Then the following statements are equivalent:
\begin{enumerate}[{\rm (i)}]
  \item $x^*$ is a stationary point of \eqref{P0}.
  \item There exists $\xi^*\in\partial P_2(x^*)$ such that $x^*\in\Argmin_{x\in  {\cal F}(x^*,\xi^*)}\{\langle \nabla f(x^*), x\rangle\}$.\footnote{See \eqref{Fyxi} for the definition of ${\cal F}(x^*,\xi^*)$.}
  \item There exist $\xi^*\in\partial P_2(x^*)$ and $u^*\in\Argmin_{x\in {\cal F}(x^*,\xi^*)}\{\langle \nabla f(x^*), x\rangle\}$ such that
	\[
	\langle \nabla f(x^*), u^* - x^*\rangle = 0.
	\]
\end{enumerate}
\end{lemma}
\begin{proof}
(i)$\Leftrightarrow$(ii): By Assumption~\ref{assumption1}, for any $x^*\in\mathcal{F}$ and for any $\xi^*\in\partial P_2(x^*)$, one
can see that the constraint set ${\cal F}(x^*,\xi^*) = \{x:\; P_1(x) - P_2(x^*) - \langle \xi^*, x - x^*\rangle \leq
\sigma\}$ contains a Slater point $\xfeas_{(x^*, \xi^*)}$. Then, in view of \cite[Corollary~28.2.1, Theorem~28.3]{Ro70}, we see that (i) is equivalent to (ii).
	
(ii)$\Rightarrow$(iii): Let $\xi^*$ be as in item (ii). Then $\Argmin_{x\in  {\cal F}(x^*,\xi^*)}\{\langle \nabla f(x^*), x\rangle\}$ is nonempty thanks to \eqref{constrain}. Now, pick any $u^*\in\Argmin_{x\in  {\cal F}(x^*,\xi^*)}\{\langle \nabla f(x^*), x\rangle\}$. {\color{blue} As $x^*\in {\cal F}(x^*,\xi^*)$, we obtain} $P_1(u^*) - P_2(x^*) - \langle \xi^*, u^* - x^*\rangle \leq \sigma$ and
\begin{equation*}
\langle \nabla f(x^*), u^* \rangle \leq \langle \nabla f(x^*), x^*\rangle.
\end{equation*}
 Next, since $x^*\in\Argmin_{x\in  {\cal F}(x^*,\xi^*)}\{\langle \nabla f(x^*), x\rangle\}$, one has
\[
\langle \nabla f(x^*), x^* \rangle \leq \langle \nabla f(x^*), u^*\rangle
\]
Combining the above two displays yields $\langle \nabla f(x^*), u^* - x^*\rangle = 0$.

(iii)$\Rightarrow$(ii): Let $\xi^*$ and $u^*$ be as in item (iii). Then
\[
\langle \nabla f(x^*), x^*\rangle = \langle \nabla f(x^*), u^*\rangle = \min\limits_{x\in {\cal F}(x^*,\xi^*)}\{\langle \nabla f(x^*), x\rangle\}.
\]
Now, since $x^*\in\mathcal{F}$ and hence $x^*\in {\cal F}(x^*,\xi^*)$, we conclude that (ii) holds.
\qed \end{proof}

Finally, we will introduce a stationarity measure for \eqref{P0} which plays an important role in our convergence analysis later. Specifically, we define a merit function $G: \mathcal{F} \rightarrow \mathbb{R}$ by
\begin{equation}\label{defG}
G(x)=\inf_{\xi \in \partial P_2(x)} \max_{y \in \mathcal{F}(x, \xi)} \langle \nabla f(x), x-y\rangle\ \ \ \ \mbox{ for all } x \in \mathcal{F},
\end{equation}
where ${\cal F}(x,\xi)$ is defined as in \eqref{Fyxi}. Note that the maximum in the definition of $G$ is attained.\footnote{
The attainment of the maximum in the definition of $G$ follows from the fact that $\mathcal{F}(x, \xi)$ is a nonempty compact set for all $x \in \mathcal{F}$ and $\xi \in \partial P_2(x)$; see \eqref{constrain}.} In addition, notice that when $ P_2\equiv 0 $, the function $ G $ in \eqref{defG} reduces to the $g$ in \cite[Eq.~(2)]{Jaggi13}: this latter function was used as a measure for optimality in \cite{Jaggi13} where $f$ was assumed to be convex. Here, for problem \eqref{P0} with a possibly nonzero $P_2$, we argue that $G$ can be regarded as a measure for ``proximity to stationarity" for any feasible points of \eqref{P0}, under Assumption~\ref{assumption1}.
\begin{theorem}[Stationarity measure]\label{lemma:1}
Consider \eqref{P0} and suppose that Assumption~\ref{assumption1} holds.  Then, the following statements hold for the $G$ in \eqref{defG}:
\begin{enumerate}[{\rm (i)}]
\item $G(x) \ge 0$ for all $x \in \mathcal{F}$.
\item Let $\{x^k\}\subseteq  \mathcal{F}$.  If $G(x^k) \rightarrow 0$ and $x^k \rightarrow x^*$ for some $x^*$, then $x^*\in {\cal F}$ and is a stationary point of  \eqref{P0}.
\item For $x^* \in \mathcal{F}$, we have $G(x^*)=0$  if and only if $x^*$ is a stationary point of \eqref{P0}.
\end{enumerate}
\end{theorem}
\begin{proof}
Item {\rm (i)} holds in view of \eqref{constrain} and the definition of $G$.

To prove {\rm (ii)}, let $\{x^k\} \subseteq \mathcal{F}$ and $x^k \rightarrow x^*$ with $G(x^k) \rightarrow 0$. Then $x^*\in\mathcal{F}$ because ${\cal F}$ is closed.
Next, notice that
\[
0 \le G(x^k)=\inf_{\xi \in \partial P_2(x^k)} \max_{y \in \mathcal{F}(x^k, \xi)} \langle \nabla f(x^k), x^k-y\rangle \rightarrow 0.
\]
So there exist $\xi^k \in \partial P_2(x^k)$  and $y^k \in {\rm Argmax}_{y \in \mathcal{F}(x^k, \xi^k)}\langle \nabla f(x^k), x^k-y\rangle$ such that
\begin{equation}\label{eq:use1}
\langle \nabla f(x^k), x^k-y^k \rangle \rightarrow 0.
\end{equation}
Note that $\{y^k\}\subseteq {\cal F}$ in view of \eqref{constrain} and is hence bounded. Moreover, $\{x^k\}$ is bounded and hence $\{\xi^k\}$ is bounded in view of \cite[Theorem~2.6]{Tuy98} and the continuity and convexity of $P_2$.
Passing to convergent subsequences if necessary, we may assume that $y^k \rightarrow y^*$ and $\xi^k \rightarrow \xi^*$ for some
 $y^* \in \mathcal{F}$ and $\xi^* \in \partial P_2(x^*)$ (thanks to the closedness of $\partial P_2$). Passing to the limit in \eqref{eq:use1} and using the continuity of $\nabla f$, we have
\begin{equation}\label{eq:use2}
\langle \nabla f(x^*), x^*-y^* \rangle =0.
\end{equation}

We next claim that
\begin{equation}\label{minsub}
x^* \in \Argmin\limits_{x\in\X}\{\langle\nabla f(x^*), x\rangle:\; P_1(x) - P_2(x^*) - \langle\xi^*, x -
x^*\rangle\leq\sigma \}.
\end{equation}
Granting this, in view of Lemma~\ref{equistat}(i), (ii) and the fact that $\xi^*\in \partial P_2(x^*)$, we can then conclude that $x^*$ is a stationary point of \eqref{P0}.
	
It now remains to establish \eqref{minsub}. To this end, we first define the following:
\begin{equation}\label{Omegag}
\begin{array}{ll}
\!\!\!\!\!g_{k}(x)\!:=\!  P_1(x) \!-\! P_2(x^{k}) \!-\!\langle\xi^{k}, x -x^{k}\rangle \!-\! \sigma, & \Omega_{k} \!:=\! \{x\in\X:\, g_{k}(x) \leq 0\}
\ \forall k,\!\!\!\!\!\!\\
\!\!\!\!\!g_{*}(x)\!:=\!  P_1(x) \!-\! P_2(x^*) \!-\!\langle\xi^*, x -x^*\rangle \!-\! \sigma,& \Omega_{*} \!:=\! \{x\in\X:\, g_{*}(x) \leq 0\}.
\end{array}
\end{equation}
Then, because $x^{k}\to x^*$ and $\xi^{k}\to \xi^*$, we have $\lim_{k\to\infty}g_{k}(x) = g_{*}(x)$ for any $x\in\R^n$.

Since $x^*\in\mathcal{F}$ and $\xi^*\in\partial P_2(x^*)$, by Assumption~\ref{assumption1}, there exists $\xfeas_{({x^*, \xi^*})}$ such that $g_{*}(\xfeas_{({x^*, \xi^*})}) < 0$, i.e. there exists $\delta_1 > 0$ such that $g_{*}(\xfeas_{({x^*, \xi^*})}) = -\delta_1 < 0$. Moreover, since $g_{k}(\xfeas_{({x^*, \xi^*})}) \to g_{*}(\xfeas_{({x^*, \xi^*})})$, there exists $N_0 > 0$, such that for any $k > N_0$ we have that $g_{k}(\xfeas_{({x^*, \xi^*})}) \leq g_*(\xfeas_{({x^*, \xi^*})}) + \frac{\delta_1}{2} = -\frac{\delta_1}{2}  < 0$.

We now apply Lemma~\ref{ErrorBounded} with $\Omega := \Omega_{k}$, ${\cal K} := \R_+$, $x^s = \xfeas_{({x^*, \xi^*})}$ and $\delta = \frac{\delta_1}{2}$ to obtain, for each $k > N_0$, that,
\begin{equation}\label{erro}
\begin{aligned}
\d(z, \Omega_{k})&\leq 2\delta_1^{-1}\|z - \xfeas_{(x^*, \xi^*)}\|\d(0, g_{k}(z) + \R_+)\\
&\leq 4M\delta_1^{-1}\d(0, g_{k}(z) + \R_+) = 4M\delta_1^{-1}[g_{k}(z)]_+, \
\forall z\in B(0, M),
\end{aligned}
\end{equation}
where $M$ is defined as in \eqref{constrain}, and the second inequality follows from the fact that $\|\xfeas_{(x^*, \xi^*)}\| \le M$ (thanks to \eqref{constrain} and the definition of $\xfeas_{(x^*, \xi^*)}$).
	
Fix any $u\in \Omega_*$. Since $x^*\in\mathcal{F}$ and $\xi^*\in\partial P_2(x^*)$, by \eqref{constrain}, we have $u\in B(0, M)$.
Moreover, applying \eqref{erro} with $z = u$, we have, for each $k > N_0$,
\[
\|u - {\rm Proj}_{\Omega_{k}}(u)\| = \d(u, \Omega_{k})\leq
4M\delta_1^{-1}[g_{k}(u)]_+,
\]
where ${\rm Proj}_{\Omega_{k}}$ is the projection mapping onto $\Omega_k$.
Since $P_1(u) - P_2(x^*) - \langle\xi^*, u - x^*\rangle \le \sigma $ (thanks to $u\in \Omega_*$), we have $[g_{k}(u)]_+\rightarrow 0$ since $x^{k}\rightarrow
x^*$ and $\xi^{k}\rightarrow \xi^*$. From this relation and the above display, we have shown that
\begin{equation}\label{limitproju}
\mbox{For each }u\in \Omega_*, \mbox{ it holds that } \lim_{k\to \infty}{\rm Proj}_{\Omega_{k}}(u) = u,
\end{equation}
where $\Omega_*$ and $\Omega_{k}$ are as in \eqref{Omegag}. Now, since $y^k \in {\rm Argmax}_{y \in \mathcal{F}(x^k, \xi^k)}\langle \nabla f(x^k), x^k-y\rangle$ and noting that $\Omega_k = \mathcal{F}(x^k, \xi^k)$, we see that for any $u\in \Omega_*$,
\[
\langle\nabla f(x^{k}), y^{k} - x^{k}\rangle \leq \langle\nabla f(x^{k}), {\rm Proj}_{\Omega_{k}}(u) -
x^{k}\rangle.
\]
Passing to the limits and noting \eqref{limitproju}, we have
\[
\langle\nabla f(x^*), y^* - x^*\rangle \leq \langle\nabla f(x^*), u - x^*\rangle.
\]
Moreover, in view of \eqref{eq:use2}, we have
\[
\langle\nabla f(x^*), x^* - x^*\rangle = 0 = \langle\nabla f(x^*), y^* - x^*\rangle \leq \langle\nabla f(x^*), u
- x^*\rangle,
\]
Since $u\in \Omega_*$ is chosen arbitrarily, we deduce that \eqref{minsub} holds.

To see {\rm (iii)}, applying {\rm (ii)} with $x^k\equiv x^* \in \mathcal{F}$, we see that if $G(x^*)=0$ then $x^*$ is a stationary point of \eqref{P0}. Conversely, suppose that $x^*$ is a stationary point of \eqref{P0}. Then, from Lemma \ref{equistat}, there exists $\xi^* \in \partial P_2(x^*)$ such that
\[
x^* \in \Argmin_{y \in \mathcal{F}(x^*, \xi^*)} \langle \nabla f(x^*), y-x^*\rangle = \Argmax_{y \in \mathcal{F}(x^*, \xi^*)} \langle \nabla f(x^*), x^*-y\rangle.
\]
This shows that $G(x^*) \le 0$. Note that $x^* \in \mathcal{F}$ and hence $G(x^*) \ge 0$ from item (i). Therefore, we see that $G(x^*)=0$.
\qed \end{proof}

\section{Algorithm and convergence analysis}\label{sec:convergence}

We present our {\color{blue}(basic) FW-type} algorithm for solving \eqref{P0} as Algorithm~\ref{FW_nonconvex} below, which involves an $\LO$ (defined in Definition~\ref{LOdef}) and a line-search strategy \eqref{linesearch}. Notice that the structure of the algorithm is similar to that of the classical FW algorithm \cite{LaZh16,Lacoste15,Jaggi13}. The main difference is that the constraint sets in the linear-optimization oracles involved in the classic FW algorithm are the same at each step, while the constraint sets in the $\LO$s in  Algorithm~\ref{FW_nonconvex} change with iteration. {\color{blue} We also note that the linear-optimization oracles (in the classic FW algorithm) can be used to compute the feasible initial point for the classic FW algorithm, while our Algorithm 2 needs a feasible point $x^0 \in {\cal F}$ so as to construct the constraint set of the {\em first} $\LO$.}
\begin{algorithm}
	\caption{Frank-Wolfe{\color{blue}-type} algorithm for \eqref{P0} under Assumption~\ref{assumption1}}\label{FW_nonconvex}
	\begin{algorithmic}
		\State \vspace{-0.15 cm}
		\begin{description}
		  \item[\bf Step 0.] Choose $x^0\in \mathcal{F}$, $c,\, \eta\in (0, 1)$, and a sequence $\{\alpha^0_k\}\subseteq(0,1]$ with $\inf_k\alpha^0_k > 0$. Set $k = 0$.  \vspace{0.1 cm}
		  \item[\bf Step 1.] Pick $\xi^k\in\partial P_2(x^k)$ and let $u^k$ be an output of $\LO(\nabla f(x^k), x^k, \xi^k)$ (see Definition~\ref{LOdef}).  Let $ d^k = d^k_{\rm fw} :=u^k-x^k $.

 \noindent If $\langle\nabla f(x^k), d^k\rangle = 0$, terminate.	 \vspace{0.1 cm}
		  \item[\bf Step 2.] Find $ \alpha_k = \alpha_k^0\eta^{j_k} $ with $ j_k $ being the smallest nonnegative integer such that
			\begin{equation}\label{linesearch}
			 f(x^k + \alpha_k d^k) \le f(x^k) + c\alpha_k\langle\nabla f(x^k), d^k \rangle.
			\end{equation}		
		  \item[\bf Step 3.] Set $\widehat x^{k+1} = x^k + \alpha_k d^k$. Choose $x^{k+1}\in\mathcal{F}$ such that $ f(x^{k+1})\leq f(\widehat x^{k+1}) $. Update $k \leftarrow k+1$ and go to Step 1.
		\end{description}
	\end{algorithmic}
\end{algorithm}

Before we analyze the well-definedness and other theoretical properties of Algorithm~\ref{FW_nonconvex}, we first comment on the termination condition in Algorithm~\ref{FW_nonconvex}.
\begin{remark}[Termination condition in Step 1]\label{term:step1}
{\color{blue}Notice that $u^k$ is an output of $\LO(\nabla f(x^k), x^k, \xi^k)$, which means that}
\[
{\color{blue} u^k\in\Argmin_{x\in {\cal F}(x^k,\xi^k)}\{\langle \nabla f(x^k), x - x^k\rangle\}.}
\]
If $\langle\nabla f(x^k), u^k - x^k\rangle = 0$, then we deduce from Lemma~\ref{equistat}(iii) that $x^k$ is a stationary point of \eqref{P0}. Moreover, if $x^k$ is not a stationary point of \eqref{P0}, we have that $\langle\nabla f(x^k), u^k - x^k\rangle < 0$ (since we always have $\langle\nabla f(x^k), u^k - x^k\rangle \leq 0$, thanks to $x^k\in {\cal F}$ and hence $x^k\in\mathcal{F}(x^k,\xi^k)$).
\end{remark}

Next we show that Algorithm~\ref{FW_nonconvex} is well-defined. By Remark~\ref{term:step1}, it suffices to show that, if $x^k\in {\cal F}$ for some $k\ge 0$ is not a stationary point of \eqref{P0} and a $ \xi^k\in\partial P_2(x^k) $ is given, then $\LO(\nabla f(x^k), x^k, \xi^k)$ has an output, the line-search step in Step 2 terminates in finitely many inner iterations, and an $x^{k+1}\in\mathcal F$ can be generated. This together with an induction argument would establish the well-definedness of Algorithm~\ref{FW_nonconvex}.

\begin{proposition}[Well-definedness of Algorithm~\ref{FW_nonconvex}]\label{Prop1}
Consider Algorithm~\ref{FW_nonconvex} for solving \eqref{P0} under Assumption~\ref{assumption1}. Suppose that an $x^k\in\mathcal{F}$ is generated at the end of {\color{blue}an iteration} of Algorithm~\ref{FW_nonconvex} for some $k\geq 0$ and suppose that $ x^k $ is not a stationary point of \eqref{P0}. Then the following
statements hold for this $k$:
\begin{enumerate}[{\rm (i)}]
  \item The $\LO(\nabla f(x^k), x^k, \xi^k)$ is well-defined, i.e. the corresponding linear-optimization problem has an optimal solution.
  \item Step 2 of Algorithm~\ref{FW_nonconvex} terminates in finitely many inner iterations.
  \item $\widehat{x}^{k+1}\in {\cal F}$.
\end{enumerate}
Thus, an $x^{k+1}\in {\cal F}$ can be generated at the end of the $(k+1)$th iteration of Algorithm~\ref{FW_nonconvex}.
\end{proposition}
\begin{proof}
(i): The well-definedness {\color{blue} of $\LO(\nabla f(x^k), x^k, \xi^k)$} follows from Remark~\ref{welldefined}.

(ii): Define $\psi(\alpha) := f(x^k + \alpha(u^k - x^k))$. Since $f$ is continuously differentiable, for each $\alpha > 0$, by the mean value theorem, there exists $t_\alpha\in (0,\alpha)$ such that
\[
\begin{split}
&f(x^k + \alpha(u^k - x^k)) = \psi(\alpha) = \psi(0) + \alpha \psi'(t_\alpha) \\
&= \psi(0) + c\alpha \psi'(0) + \alpha [(1-c)\psi'(0) + \psi'(t_\alpha) - \psi'(0)]\\
& = f(x^k) + c \alpha\langle\nabla f(x^k), u^k - x^k\rangle \\
&\ \ \ \ +\alpha\left((1-c)\langle\nabla f(x^k), u^k - x^k\rangle + [\psi'(t_\alpha) - \psi'(0)]\right)
\end{split}
\]
Noting that $\langle\nabla f(x^k), u^k - x^k\rangle < 0$ (thanks to Remark~\ref{term:step1} and the fact that $x^k$ is not a stationary point of \eqref{P0}), $c\in (0,1)$ and that $\lim_{\alpha\downarrow 0}\psi'(t_\alpha) = \psi'(0)$ (thanks to the continuity of $\nabla f$), we conclude that \eqref{linesearch} is satisfied for all sufficiently small $\alpha > 0$.
	
(iii): Recall from \eqref{constrain} that $x^k\in {\cal F}(x^k,\xi^k)$. Moreover, since $u^k$ is an output of $\LO(\nabla f(x^k), x^k, \xi^k)$, we also have $u^k\in {\cal F}(x^k,\xi^k)$. These together with $\alpha_k\in(0,1]$ and the convexity of ${\cal F}(x^k,\xi^k)$ imply that
\[
\widehat x^{k+1} := x^k + \alpha_k(u^k - x^k) = \alpha_k u^k +(1-\alpha_k)x^k \in {\cal F}(x^k,\xi^k)\subseteq\mathcal{F}.
 \]
Then an $x^{k+1}\in {\cal F}$ can be generated at the end of the $(k+1)$th iteration of Algorithm~\ref{FW_nonconvex}, since we can at least choose $x^{k+1} := \widehat x^{k+1}$. 
\qed \end{proof}
\begin{remark}[Choice of $ x^{k+1} $]
From Proposition~\ref{Prop1}, we see that one can always choose $x^{k+1} = \widehat x^{k+1}$. Here, observing that the constraint functions are all positively homogeneous (as the difference of two norms) in the examples we discussed in Sections~\ref{gl} and \ref{mc}, we introduce a boundary boosting technique to choose $ x^{k+1} $ for those examples. The main idea is to leverage the positive homogeneity of $ P_1-P_2 $. Specifically, if $ c_1 := P_1(\widehat x^{k+1}) - P_2(\widehat x^{k+1})< \sigma $ and $ c_1>0 $, we define $ \widetilde{x}^{k+1} = \frac{\sigma}{c_1}\widehat x^{k+1} $. It follows that $ P_1(\widetilde x^{k+1}) - P_2(\widetilde x^{k+1}) = \frac{\sigma}{c_1}(P_1(\widehat x^{k+1}) - P_2(\widehat x^{k+1})) = \sigma $. Then we choose
\begin{equation}\label{choosexk+1}
	x^{k+1} = \begin{cases}
	\widetilde x^{k+1} & \text{if}\;  c_1>0 \;\text{and}\; f(\widetilde x^{k+1}) \leq f(\widehat x^{k+1}), \\
	\widehat{x}^{k+1} & \text{otherwise}.
	\end{cases}
\end{equation}
\end{remark}

We next show that the sequence $\{x^k\}$ generated by Algorithm~\ref{FW_nonconvex} clusters at a stationary point of \eqref{P0}. Notice from Remark~\ref{term:step1} and Proposition~\ref{Prop1} that $\{x^k\}$ is either an infinite sequence or is a finite sequence that ends at a stationary point of
\eqref{P0}. Without loss of generality, we assume that $\{x^k\}$ is an infinite sequence.
\begin{theorem}[Subsequential convergence]\label{thm-fw}
Consider \eqref{P0} and suppose that Assumption~\ref{assumption1} holds. Let $\{x^k\}$ be an infinite sequence generated by
Algorithm~\ref{FW_nonconvex} and let $G$ be defined as in \eqref{defG}. Then the sequence $\{x^k\}$ is bounded and $G(x^k) \rightarrow 0$ as $k \rightarrow \infty$. Moreover, any accumulation point of $ \{x^k\} $ is a stationary point of \eqref{P0}.
\end{theorem}
\begin{proof}
The boundedness of $\{x^k\}$ follows from the boundedness of $\mathcal{F}$ and the fact that $\{x^k\}\subseteq\mathcal{F}$. Next, we recall from the definition that $u^{k}$ is an output of $\LO(\nabla f(x^{k}), x^{k}, \xi^{k})$. This shows that
\[
0 \le G(x^{k})\le \max_{y \in \mathcal{F}(x^{k}, \xi^{k})} \langle \nabla f(x^{k}), x^{k}-y\rangle = \langle \nabla f(x^{k}), x^{k}-u^{k}\rangle,
\]
where the first inequality follows from Theorem~\ref{lemma:1}(i) and the fact that $x^{k} \in \mathcal{F}$. In addition, in view of Remark~\ref{term:step1}, we deduce from Lemma \ref{Armijolemma} (with $\Gamma:= \mathcal{F}$) that
\[
\lim_{k\to\infty}\langle \nabla f(x^{k}), u^{k}-x^{k}\rangle =  0.
\]
Thus, we have $G(x^{k}) \rightarrow 0$. Now it follows directly from Theorem~\ref{lemma:1}(ii) that any accumulation point of $ \{x^k\} $ is a stationary point of \eqref{P0}. 
\qed \end{proof}

\subsection{The case when $P_1$ is strongly convex}
We now consider the case where $P_1$ in \eqref{P0} is a strongly convex function. We consider the following additional assumption.
\begin{assumption}\label{ass:nonzerograd}
  In \eqref{P0}, it holds that $\nabla f(x)\neq 0$ for all $x\in {\cal F}$ and $\nabla f$ is Lipschitz continuous on ${\cal F}$, i.e., there exists $L > 0$ such that
  \[
   \|\nabla f(x) - \nabla f(y)\|\le L\|x - y\|\ \ {\rm whenever}\ x,y\in {\cal F}.
  \]
\end{assumption}
The Lipschitz continuity requirement on $\nabla f$ in Assumption~\ref{ass:nonzerograd} is standard when it comes to complexity analysis of first-order methods; see, for example, \cite{Ne04}.
Moreover, the condition of nonvanishing gradient in Assumption~\ref{ass:nonzerograd} eliminates the ``uninteresting" situation where a feasible point of problem~\eqref{P0} is a stationary point of the {\em unconstrained} problem $\min_{x \in \X} f(x)$.

\begin{proposition}\label{Prop5.2}
Consider \eqref{P0} with $P_1$ being strongly convex. Suppose that Assumptions~\ref{assumption1} and \ref{ass:nonzerograd} hold. Let $\{\alpha_k\}$ and $\{d^k\}$ be two infinite sequences generated by
Algorithm~\ref{FW_nonconvex}. Then
\[
\sum_{k=0}^{\infty}  \|d^k\|^2 <+\infty\ \ {  and}\ \ \inf_k\alpha_k > 0.
\]
\end{proposition}
\begin{proof}
Since $u^{k}$ is an output of $\LO(\nabla f(x^{k}), x^{k}, \xi^{k})$ and Assumption~\ref{assumption1} holds, by \cite[Corollary~28.2.1, Theorem~28.3]{Ro70}, there exist $\lambda_k \ge 0$ and $w^k \in \partial P_1(u^k)$ such that
\begin{equation}\label{eq:KKT0}
\begin{array}{rl}
  &\nabla f(x^k)+ \lambda_k (w^k- \xi^k)=0,\\ [2 pt]
  &\lambda_k (P_1(u^k)- \langle \xi^k, u^k-x^k \rangle -P_2(x^k)-\sigma)=0.
\end{array}
\end{equation}

We claim that $\underline{\lambda}:=\inf\{\lambda_k: k \in \mathbb{N}\}>0$. Suppose not, by passing to subsequences if necessary, we may assume that $\lambda_k \rightarrow 0$ and $x^k \rightarrow x^*$ for some $x^*$. Note that $x^*\in {\cal F}$. Passing to the limit in the first relation of \eqref{eq:KKT0} and noting that $\{w^k\}$ and $\{\xi^k\}$ are bounded (thanks to the boundedness of $\{x^k\}$, $\{u^k\}$, the continuity of $P_1$ and $P_2$, and \cite[Theorem~2.6]{Tuy98}), we have
\[
\nabla f(x^*)=0.
\]
This contradicts Assumption~\ref{ass:nonzerograd}. Thus, $\underline{\lambda}=\inf\{\lambda_k: k \in \mathbb{N}\}>0$.

Next, let $\rho > 0$ denote the modulus of strong convexity of $P_1$. Using \eqref{eq:KKT0}, we have
\begin{eqnarray}\label{eq:ineq}
&&\langle \nabla f(x^k), u^k-x^k\rangle  =   \langle -\lambda_k (w^k- \xi^k), u^k-x^k\rangle \notag \\
&& =  \lambda_k  \langle w^k, x^k-u^k\rangle + \lambda_k  \langle \xi^k, u^k-x^k\rangle \notag \\
&& \le  \lambda_k (P_1(x^k)-P_1(u^k) -\frac{\rho}{2}\|x^k-u^k\|^2) + \lambda_k (P_1(u^k)-P_2(x^k)-\sigma) \notag \\
&& =  \lambda_k(P_1(x^k)-P_2(x^k)-\sigma) - \frac{\rho \lambda_k}{2}\|x^k-u^k\|^2 \le  - \frac{\rho \underline{\lambda}}{2}\|x^k-u^k\|^2.
\end{eqnarray}
where the first inequality holds in view of the second relation of \eqref{eq:KKT0} and the definition of $\rho$, and the last inequality holds because $x^k \in \mathcal{F}$.
This together with the definition of $x^{k+1}$ and \eqref{linesearch} shows that
\[
f(x^{k+1}) \le f(\widehat{x}^{k+1}) \le f(x^k) + c\alpha_k\langle\nabla f(x^k), u^k - x^k\rangle  \le f(x^k) - c\alpha_k\frac{\rho \underline{\lambda}}{2}\|x^k-u^k\|^2.
\]
Note that $\{x^k\} \subseteq \mathcal{F}$ (so that $\inf_k f(x^k) > -\infty$). Using this and summing the above display from $k = 0$ to $\infty$, we see further that
\begin{equation}\label{hahahaha2}
\sum_{k=0}^{\infty} \alpha_k  \|d^k\|^2=\sum_{k=0}^{\infty} \alpha_k  \|x^k-u^k\|^2 <+\infty.
\end{equation}

We now show that $\inf_k\alpha_k>0$. To see this, first recall from \eqref{constrain} that $x^k\in {\cal F}(x^k,\xi^k)$, and note from the definition that $u^k\in {\cal F}(x^k,\xi^k)$. Moreover, since ${\cal F}(x^k,\xi^k)$ is convex, we have $x^k+ \alpha (u^k-x^k) \in {\cal F}(x^k,\xi^k) \subseteq {\cal F}$ for all $\alpha\in [0,1]$. Then, for any $\alpha\in [0,1]$, with $L$ as in Assumption~\ref{ass:nonzerograd}, we have
\begin{align*}
&f(x^k+ \alpha (u^k-x^k))  \le  f(x^k) + \alpha \langle \nabla f(x^k), u^k-x^k\rangle+ \frac{L \alpha^2}{2}\|u^k-x^k\|^2 \\
& =  f(x^k) + c \alpha \langle \nabla f(x^k), u^k-x^k\rangle \\
& \ \ \ \ +\bigg[(1-c)\alpha \langle \nabla f(x^k), u^k-x^k\rangle + \frac{L \alpha^2}{2}\|u^k-x^k\|^2\bigg] \\
& \le  f(x^k) + c \alpha \langle \nabla f(x^k), u^k-x^k\rangle + \alpha  \|u^k-x^k\|^2 \,  \bigg[- \, \frac{(1-c) \rho \underline{\lambda}}{2}  + \frac{L \alpha}{2}\bigg],
\end{align*}
where the last inequality follows from \eqref{eq:ineq}.
This shows that the line search criterion \eqref{linesearch} will be satisfied as long as
$\alpha \le \frac{(1-c)\rho \underline{\lambda}}{L}$.
Consequently, we have
$\alpha_k \ge \min\{ \alpha_k^0, \frac{(1-c)\rho \underline{\lambda} \eta}{L} \}$,
and so $\inf_k\alpha_k \ge \min\{ \inf_{k}\alpha_k^0, \frac{(1-c)\rho \underline{\lambda} \eta}{L}\}>0$.
The desired conclusion now follows immediately from this and \eqref{hahahaha2}.
\qed \end{proof}

Next, we derive a $o(1/k)$ complexity in terms of the stationarity measure $G$ defined in \eqref{defG} under the strong convexity of $P_1$.
\begin{theorem}
Consider \eqref{P0} with $P_1$ being strongly convex. Suppose that Assumptions~\ref{assumption1} and \ref{ass:nonzerograd} hold. Let $\{x^k\}$ be an infinite sequence generated by
Algorithm~\ref{FW_nonconvex} and let $G$ be defined as in \eqref{defG}. Then
\[
\sum_{k=0}^{\infty}  G(x^k) <+\infty \ \mbox{ and }\ k \big[\min_{0 \le t \le k} G(x^t) \big] \rightarrow 0.
\]
\end{theorem}
\begin{proof}
Recalling that $u^k$ is an output of $\LO(\nabla f(x^{k}), x^{k}, \xi^{k})$, where $\xi^k \in \partial P_2(x^k)$, we have
\[
G(x^k)\le \max_{y \in \mathcal{F}(x^k, \xi^k)} \langle \nabla f(x^k), x^k-y\rangle = \langle \nabla f(x^k), x^k-u^k\rangle
\]
This together with the definition of $x^{k+1}$ implies that
\[
f(x^{k+1})-f(x^k) \le f(\widehat{x}^{k+1})-f(x^k) \le c \alpha^k \langle \nabla f(x^{k}), u^k-x^k\rangle \le - c \alpha^k G(x^k).
\]
Summing the above display from $k = 0$ to $\infty$ and recalling $\inf_k f(x^k) > -\infty$ (since $\{x^k\}\subseteq {\cal F}$), we see that $\sum_{k=0}^{\infty} \alpha_k G(x^k)<\infty$.
Since $\inf_k\alpha_k > 0$ from Proposition~\ref{Prop5.2} and $G(x^k)\ge 0$ (see Theorem~\ref{lemma:1}(i)), we deduce further that
$\sum_{k=0}^{\infty} G(x^k)<\infty$.

Finally, let $a_k=\min_{0 \le t \le k} G(x^t)$. Then $\{a_k\}$ is nonnnegative (see Theorem~\ref{lemma:1}(i)) and non-increasing. On the other hand, the Cauchy criterion for the convergent series gives $\lim_{k\to\infty}\sum_{t=k}^{2k} G(x^t) = 0$. This implies that
\[
(k+1) \, a_{2k} \le G(x^k)+\ldots + G(x^{2k}) = \sum_{t=k}^{2k} G(x^t) \rightarrow 0
\]
Similarly, one can show that $\lim_{k\to\infty}(k+1)a_{2k+1}= 0$.
Thus, we have $\lim_{k\to\infty}k a_{k} = 0$ as desired.
\qed \end{proof}

\section{Away-step Frank-Wolfe{\color{blue}-type} algorithm}\label{sec6}
When $f$ and $C$ in \eqref{P00} are convex and $C = {\rm conv}(\mathcal A) $ for some finite set $ \mathcal A $ (whose elements are called atoms), the so-called away-step technique was proposed in \cite{Guelat86} to accelerate the convergence of FW method; see \cite{Lacoste15,Garber16,Clarkson10} for recent developments. The key ingredient in the away-step technique is to keep track of \emph{a convex decomposition of the current iterate into atoms in $ \mathcal A $}. The away step then selects one atom from the current decomposition that ``differs" most from the gradient direction. The FW method with away {\color{blue}steps} requires to maintain the set of the ``active" atoms used in the aforementioned decomposition and update this active set in each iteration (see \cite[Algorithm~1]{Lacoste15} for more details).

In our nonconvex settings, the feasible set of $\LO$ is not necessarily a convex hull of finitely many atoms, and it can vary from iteration to iteration. In particular, the active atoms used for the current iteration may be infeasible for the subproblem of the next iteration. It seems difficult to maintain the decomposition with a uniform atomic set $ \mathcal A $ and possibly meaningless to update the decomposition based on previous information. In view of this, we do not store atoms used in the decompositions of previous iterates, but generate a new set of atoms in each iteration based on the current iterate. Below we describe the details of how to decompose the iterate and construct away-step oracles for our {\color{blue}FW-type} method for \eqref{P0}.

We first introduce our away-step oracle, which mimics the away {\color{blue}steps} used in the classical FW method {\color{blue} under the convex setting}.

\begin{definition}[Away-step oracles]\label{awdef}
Let $P_1$, $P_2$ and $\sigma$ be defined in \eqref{P0}, $ y\in\mathcal F $ and $ \xi\in\partial P_2(y) $.
Given $ a\in\X\setminus\{0\} $, choose a set
\begin{equation}\label{Sy}
	\mathcal S(y, \xi)=\left\{v_1,\ldots, v_{m_y} \right\}\subseteq \bd {\cal F}(y,\xi)
\end{equation}
so that $ y = \sum_{i=1}^{m_y} c_i v_i $ for some $ c_1,\ldots,c_{m_y}>0$ satisfying $ \sum_{i=1}^{m_y} c_i = 1$, where $1\le m_y\leq n+1 $ with $ n $ being the dimension of $ \X $.
We define the away-step oracle $ \AWO(a, \mathcal S(y,\xi)) $ as
\begin{equation}\label{aw-step}
\begin{array}{cl}
\max\limits_{x\in \X} & \langle a, x \rangle \ \ \ \
{\rm s.t.} \ \ x \in {\rm conv}(\mathcal{S}(y, \xi)).
\end{array}
\end{equation}	
\end{definition}

The choice of the set $ \mathcal S(y, \xi) $ is essential for $ \AWO(a, \mathcal S(y,\xi)) $. Specific strategies of choosing $ \mathcal S(y,\xi) $ and how the $ \AWO $ can be carried out efficiently for some concrete examples will be discussed in Section~\ref{sec:aw3exs}. Here, we first comment on the existence of such $ \mathcal S(y,\xi) $.

\begin{remark}[Existence of $S(y,\xi)$]\label{remark1}
Since $y\in {\cal F}$ and $ \xi\in\partial P_2(y) $ in Definition~\ref{awdef}, we see from \eqref{constrain} that $y \in {\cal F}(y,\xi)$. Since ${\cal F}(y,\xi)$ is compact and convex, it is the convex hull of all its extreme points; in particular, ${\cal F}(y,\xi) = {\rm conv}\, (\bd {\cal F}(y,\xi))$.
The existence of $v_i$ and $c_i$ in Definition~\ref{awdef} now follows from this and the Carath\'{e}odory's theorem.
\end{remark}

Next, we note that $\AWO$ is well-defined, because it is solving a maximization problem with a linear objective and a nonempty compact feasible set. We register this simple observation as our next proposition.
\begin{proposition}[Well-definedness of $\AWO$]
Consider \eqref{P0}. For any fixed $ y\in \mathcal{F} $, $ \xi\in \partial P_2(y) $ and any set $ \mathcal S(y,\xi) $ given as in \eqref{Sy}, the away-step oracle $ \AWO(\nabla f(y), \mathcal S(y,\xi)) $ in \eqref{aw-step} is well-defined.
\end{proposition}

We now present our Frank-Wolfe{\color{blue}-type} algorithm with away {\color{blue}steps} as Algorithm~\ref{alg:aw-fw} to enhance Algorithm~\ref{FW_nonconvex} for solving \eqref{P0} under Assumption~\ref{assumption1}.
\begin{algorithm}
	\caption{Away-step Frank-Wolfe{\color{blue}-type} algorithm for \eqref{P0} under Assumption~\ref{assumption1}}\label{alg:aw-fw}
	\begin{algorithmic}
		\State 
		\begin{description}
			\item[\bf Step 0.] Choose $x^0\in \mathcal{F}$, $ 0<\epsilon< \zeta < \infty$, $c,\,\eta\in (0, 1)$ and a sequence $ \{\alpha_k^0\}\subseteq (0,1] $ with $\inf_k \alpha_k^0 > 0$. Set $k = 0$.\vspace{0.1 cm}

			\item[\bf Step 1.] Pick $\xi^k\in\partial P_2(x^k)$. Compute $ u^k_{\rm fw}$ by calling $\LO(\nabla f(x^k),x^k, \xi^k)$ (see Definition~\ref{LOdef}) and set $ d_{\rm fw}^k := u^k_{\rm fw} - x^k $.

\noindent If $ \langle \nabla f(x^k), u_{\rm fw}^k-x^k\rangle = 0 $, terminate. \vspace{0.1 cm}

			\item[\bf Step 2.] Choose $ \mathcal S(x^k, \xi^k) $ as in \eqref{Sy} and call $ \AWO(\nabla f(x^k),\mathcal S(x^k, \xi^k)) $  (see Definition~\ref{awdef}) to compute $ u^k_{\rm aw} $.
Set $ d^k_{\rm aw} := x^k - u^k_{\rm aw} $ and choose an $ \alpha^k_{\rm aw}\le \max\left\{\alpha\geq 0:\; x^k+\alpha d^k_{\rm aw} \in {\rm conv}(\mathcal S(x^k,\xi^k)) \right\} $.\vspace{0.1 cm}

			\item[\bf Step 3.] If $ \langle \nabla f(x^k), d^k_{\rm fw}\rangle > \langle \nabla f(x^k), d^k_{\rm aw}\rangle $ and $\alpha_{\rm aw}^k\in(\epsilon, \zeta] $, set $ d^k =  d^k_{\rm aw} $ and $ \widetilde\alpha_k^0  = \alpha_{\rm aw}^k $; we declare that an AW step is taken.\vspace{0.1 cm}

\noindent Otherwise,  set $ d^k = d^k_{\rm fw}$ and $\widetilde \alpha^0_k = \alpha^0_k$; we declare that an FW step is taken.\vspace{0.1 cm}
 			
			\item[\bf Step 4.] Find $ \alpha_k = \widetilde \alpha_k^0\eta^{j_k} $ with $ j_k $ being the smallest nonnegative integer such that
			\begin{equation}\label{ls_awfw}
			  f(x^k + \alpha_k d^k) \le f(x^k) + c\alpha_k\langle\nabla f(x^k), d^k \rangle.
			\end{equation}			
			\item[\bf Step 5.] Set $\widehat x^{k+1} = x^k + \alpha_k d^k$. Choose $ x^{k+1}\in\mathcal F $ such that $ f(x^{k+1})\leq f(\widehat x^{k+1}) $. Update $k \leftarrow k+1$ and go to Step 1.
		\end{description}
	\end{algorithmic}
\end{algorithm}

\begin{remark}[{Well-definedness of Algorithm~\ref{alg:aw-fw}}]\label{welldefined:Alg2}
Similar to Proposition~\ref{Prop1}, we can argue the well-definedness of Algorithm~\ref{alg:aw-fw} as follows: Suppose that a nonstationary $x^k\in {\cal F}$ is given for some $k \ge 0$. Note that $\langle \nabla f(x^k), u_{\rm fw}^k-x^k\rangle\le 0$ because $x^k \in {\cal F}(x^k,\xi^k)$. Since $x^k$ is not stationary, we further have $\langle \nabla f(x^k), u_{\rm fw}^k-x^k\rangle< 0$ in view of Lemma~\ref{equistat}. Then the rule of choosing $d^k$ in Step~3 of Algorithm~\ref{alg:aw-fw} yields $ \langle \nabla f(x^k), d^k\rangle < 0 $. Therefore, one can show similarly as in Proposition~\ref{Prop1} that the line-search subroutine in Step~4 of Algorithm~\ref{alg:aw-fw} can terminate in finitely many inner iterations. Furthermore, from the definition of $\alpha_k$, one can deduce that $ \widehat{x}^{k+1} \in\mathcal F(x^k,\xi^k)\subseteq {\cal F} $. Then an $ x^{k+1} \in {\cal F}$ can be generated at the end of the ($k+1$)th iteration of Algorithm~\ref{alg:aw-fw}, since we can at least choose $ x^{k+1}:=\widehat{x}^{k+1} $.
\end{remark}

We next show that the sequence $\{x^k\}$ generated by Algorithm~\ref{alg:aw-fw} clusters
at a stationary point of \eqref{P0}. From the discussion in Remarks~\ref{term:step1} and \ref{welldefined:Alg2}, we see that $\{x^k\}$ is either an infinite sequence or is a finite sequence that terminates at a stationary point of \eqref{P0}. Without loss of generality, we assume that $\{x^k\}$ is an infinite sequence.

\begin{theorem}[Subsequential convergence]\label{thm-fw-aw}
Consider \eqref{P0} and suppose that Assumption~\ref{assumption1} holds. Let $ \{x^k\} $ be an infinite sequence generated by
Algorithm~\ref{alg:aw-fw}. Then $ \{x^k\} $ is bounded and every accumulation point of $ \{x^k\} $ is a stationary point of \eqref{P0}.
\end{theorem}
\begin{proof}
Note that $\{x^k\}$ is bounded because $\{x^k\}\subseteq {\cal F}$. Similarly, $\{u^k_{\rm aw}\}$ and $\{u^k_{\rm fw}\}$ are bounded, and we also have the boundedness of $\{\xi^k\}$ in view of \cite[Theorem~2.6]{Tuy98}, the continuity of $P_2$ and the boundedness of $\{x^k\}$.

Next, observe from Step 3 of Algorithm~\ref{alg:aw-fw}, the descent property of the FW direction (i.e., $\langle \nabla f(x^k), u^k_{\rm fw} -x^k \rangle \leq \langle \nabla f(x^k), x^k -x^k\rangle = 0$ for all $k$) and the assumption that $\{x^k\}$ is an infinite sequence (so that $\langle \nabla f(x^k), u^k_{\rm fw} -x^k \rangle < 0$ according to Step 1)
that $ d^{k} $ is a descent direction for every $ k $, i.e., $ \langle \nabla f(x^k), d^k\rangle < 0 $ for every $ k $. In view of this, \eqref{ls_awfw} and Lemma~\ref{Armijolemma} (with $\Gamma:= \mathcal{F}$), we have
\begin{equation}\label{limitf}
\lim_{k\to\infty}\langle \nabla f(x^k), d^k\rangle = 0.
\end{equation}

Now, let $x^*$ be an accumulation point of $\{x^k\}$. Then, in view of the boundedness of $ \{(x^k,\xi^k,u^k_{\rm aw},u^k_{\rm fw})\} $, there exists a subsequence $ \{(x^{k_t},\xi^{k_t},u^{k_t}_{\rm aw},u^{k_t}_{\rm fw})\} $ that
converges to $ (x^*, \xi^*, u^*_{\rm aw} , u^*_{\rm fw}) $ for some $\xi^*\in\partial P_2(x^*)$ (thanks to the closedness of $\partial P_2$), $u^*_{\rm aw}\in {\cal F}$ and $u^*_{\rm fw}\in {\cal F}$. We consider two cases.

{\bf Case 1}: Suppose that $ \{x^{k_t}\} $ is followed by infinitely many FW steps and finitely many AW steps. By passing to a further subsequence, we assume without loss of generality that $ x^{k_t} $ is followed by an FW step for all $t$, i.e., $d^{k_t} = u^{k_t}_{\rm fw} - x^{k_t}$ for all $t$. Since $u^k_{\rm fw}$ is an output of $\LO(\nabla f(x^k),x^k,\xi^k)$, we have from the definition of $G$ in \eqref{defG} and \eqref{limitf} that
\[
0 \le G(x^{k_t})\le \langle \nabla f(x^{k_t}), x^{k_t}-u^{k_t}_{\rm fw}\rangle = - \langle \nabla f(x^{k_t}), d^{k_t}\rangle \to 0.
\]
This together with Theorem~\ref{lemma:1}(ii) shows that $x^*$ is a stationary point of \eqref{P0}.

{\bf Case 2}: Suppose that the AW step is invoked infinitely many times in $ \{x^{k_t}\} $.
Passing to a further subsequence, we assume without loss of generality that $ x^{k_t} $ is followed by an AW step for all $ t $, i.e., $d^{k_t} = x^{k_t} - u^{k_t}_{\rm aw}$ for all $t$. Then in view of Step 3, we have
\begin{equation}\label{dirule}
\langle \nabla f(x^{k_t}), x^{k_t}-u^{k_t}_{\rm aw}\rangle = \langle \nabla f(x^{k_t}), d^{k_t}\rangle < \langle \nabla f(x^{k_t}), u^{k_t}_{\rm fw}-x^{k_t}\rangle.
\end{equation}
Now, recalling the definition of $G$ in \eqref{defG} and the facts that $\xi^{k_t} \in \partial P_2(x^{k_t})$ and $u^{k_t}_{\rm fw}$ is an output of $\LO(\nabla f(x^k),x^k, \xi^k)$, we deduce further that
\[
\begin{aligned}
0\le G(x^{k_t})&\le \max_{y \in \mathcal{F}(x^{k_t}, \xi^{k_t})} \langle \nabla f(x^{k_t}), x^{k_t}-y\rangle = \langle \nabla f(x^{k_t}), x^{k_t}-u^{k_t}_{\rm fw}\rangle \\
& \le - \langle \nabla f(x^{k_t}), x^{k_t}-u^{k_t}_{\rm aw}\rangle = - \langle \nabla f(x^{k_t}), d^{k_t}\rangle \rightarrow 0,
\end{aligned}
\]
where the second inequality follows from \eqref{dirule} and the last relation follows from \eqref{limitf}.
This together with Theorem~\ref{lemma:1}(ii) implies that $ x^* $ is a stationary point of problem~\eqref{P0}. This completes the proof.
\qed \end{proof}

\subsection{{$ \AWO $} for $P_1$ and $P_2$ as in Assumption~\ref{assump31} or \ref{assump32}}\label{sec:aw3exs}
In this section, we discuss how to construct the set $ \mathcal S(y,\xi) $ and obtain a solution for $\AWO$ when $ P_1 $ and $ P_2 $ are described as in Assumption~\ref{assump31} or \ref{assump32}.

Note that the $ P_1 $ in Assumption~\ref{assump31} or \ref{assump32} are ``atomic norms" that can be written as a gauge function of the form:
\begin{equation}\label{atomic}
\|x\|_\mathcal{A} = \inf\{t\geq 0:\; x\in t\;{\rm conv}(\mathcal{A}) \}
\end{equation}
for some symmetric compact atomic set $\mathcal{A}\subseteq \X $ with $0 \in {\rm conv}({\cal A})$. For more discussions on atomic norm, we refer the readers to \cite{Chandrasekaran10}. Below, we list the atomic sets for the $P_1$ in the two scenarios discussed in Assumptions \ref{assump31} and \ref{assump32}.
\begin{itemize}
	\item {\bf Scenario~1.} $ \mathcal{A}_{\rm gl} = \bigcup_{J\in{\cal J}}\mathcal{A}_J$ with
\[
\mathcal A_{J} = \{x:\,\|x_{J}\| = 1,\; x_I= 0, \forall I\in \mathcal J\setminus \{J\} \}
\]
for each $ J\in\mathcal J $, where $ \mathcal J $ is a partition of $ \{1,\ldots, n\} $. This corresponds to $ P_1(x) = \sum_{J\in\mathcal J} \|x_{J}\| $ (see \cite[Corollary~2.2]{RaoRecht12}) in Assumption~\ref{assump31}.
	\item {\bf Scenario~2.} $ \mathcal A_* = \{uv^T:\; u\in\R^m, v\in\R^n \;{\rm with}\;\|u\| = \|v\| = 1 \} $. This corresponds to $ P_1(X) = \|X\|_* $ (see \cite[Section~2.2]{Chandrasekaran10}) in Assumption~\ref{assump32}.
\end{itemize}
For the rest of this section, we will focus on \eqref{atomic} with ${\cal A} = {\cal A}_{\rm gl}$ or ${\cal A}_*$.
Based on these two atomic sets, we first construct a set that provides potential choices for the elements of $ \mathcal S(y,\xi) $ in $\AWO$ for problems in Sections~\ref{gl} and \ref{mc}.
\begin{proposition}\label{atomprop}
	Consider \eqref{P0}. Let $ P_1 = \|\cdot\|_{\mathcal A} $ as in \eqref{atomic} with ${\cal A} = {\cal A}_{\rm gl}$ or ${\cal A}_*$, and let $ P_2 $ be a norm function such that $ P_2\leq \mu P_1$ for some $ \mu \in[0,1) $. Given $ y\in\mathcal F $ and $ \xi\in\partial P_2(y) $, let
	\[
	\mathcal V(y,\xi) = \left\{\frac{{\sigma}s}{1-\langle \xi, s\rangle}:\; s\in {\cal A} \right\}.
	\]
    Then the following statements hold:
	\begin{enumerate}[{\rm (i)}]
		\item For any $ s \in {\cal A}$, we have $ 1-\langle\xi, s\rangle >0 $;
		\item For any $ v\in\mathcal{V}(y,\xi) $, it holds that
\[\|v\|_\mathcal{A} - P_2(y) - \langle \xi, v - y\rangle = \|v\|_\mathcal{A} - \langle \xi, v \rangle =\sigma.\]
	\end{enumerate}
\end{proposition}
\begin{proof}
	For notational simplicity, we write $ P_2(x)= \mu \gamma(x) $ for some norm function $\gamma$, and use $ \gamma^\circ $ to denote the dual norm of $ \gamma $. Since $ \xi\in\partial P_2(y) $, we see that $ \gamma^\circ(\xi)\leq \mu $. Since we have $\|s\|_{\cal A} = 1$ for all $s\in {\cal A}$ (recall that ${\cal A} = {\cal A}_{\rm gl}$ or ${\cal A}_*$), it follows that
	\begin{equation*}
	\begin{split}
	&1-\langle\xi, s\rangle \geq \|s\|_\mathcal{A} - \gamma^\circ(\xi)\gamma(s)\geq \|s\|_\mathcal{A}-\mu \gamma(s) = P_1(s)- P_2(s) \\
	&\overset{\rm (a)}\geq P_1(s)-\mu P_1(s) = 1-\mu >0,
	\end{split}
	\end{equation*}
	where the first inequality and the last equality holds because $P_1(s) = \|s\|_\mathcal{A}=1 $, and (a) holds because $ P_2\leq \mu P_1  $. This proves (i).
	
	We now prove (ii). Fix any $ v\in\mathcal{V}(y,\xi) $. Then we have
	\begin{equation*}
	\begin{split}
	&\|v\|_\mathcal{A} - P_2(y)- \langle\xi, v-y\rangle= \frac{{\sigma}\left(\|s\|_\mathcal{A} -\langle \xi, s\rangle \right)}{1-\langle \xi, s\rangle} -P_2(y) +\langle \xi, y\rangle \\
	&\overset{\rm (a)}= \sigma - P_2(y) +\langle\xi,y\rangle \overset{\rm (b)}= \sigma,
	\end{split}
	\end{equation*}
	where (a) holds since $ \|s\|_\mathcal{A} = 1 $, and (b) holds as $P_2$ is a norm and $\xi\in \partial P_2(y)$. 
\qed \end{proof}
In the next two subsections, we will present a strategy of choosing $ \mathcal{S}(y,\xi) $ in {\bf Scenarios 1} and {\bf 2}, and discuss how to determine the stepsize $ \alpha_{\rm aw} $ in Step~2 in Algorithm~\ref{alg:aw-fw}. Here, we first exclude the case $ y=0 $: In this case, the away-step direction $ d_{\rm aw} $ will always be inferior to $ d_{\rm fw} $ in Step~3 if we choose $ \xi=0 $ (note that $ 0\in\partial P_2(0) $ since $P_2$ is a norm). To see this, note that $ \mathcal{F}(0,0)=\{x:\; \|x\|_\mathcal{A}\leq \sigma \} $. Therefore, we have
\[
\langle \nabla f(0), 0 - u_{\rm aw} \rangle = \langle \nabla f(0), (-u_{\rm aw}) - 0\rangle \geq \langle \nabla f(0), u_{\rm fw}-0\rangle,
\]
where we use the facts that $ -u_{\rm aw}\in\mathcal{F}(0,0) $ (thanks to the symmetry of $ \mathcal{F}(0,0) $) and $ u_{\rm fw} $ is the output of $\LO(\nabla f(0),0,0)$ for the last inequality. In other words, if $x^k = 0$ and is not stationary, and if we set $\xi^k = 0$, then $d^k = u^k_{\rm fw} - x^k$.

\subsubsection{Constructing $ \mathcal S(y, \xi) $ in $\AWO$ in Step 2 of Algorithm~\ref{alg:aw-fw}} \label{sec:syxi}

Let $ P_1 $ and $ P_2 $ be as in Assumption~\ref{assump31} or \ref{assump32}. As discussed before, these $ P_1 $ can be written as $ P_1=\|\cdot\|_\mathcal A $ with ${\cal A} = {\cal A}_{\rm gl}$ or ${\cal A}_*$.

We now discuss the choice of ${\cal S}(y,\xi)$ in \eqref{aw-step} when $y\in\mathcal{F}\setminus\{0\}$. To this end, fix any $ y\in\mathcal{F}\setminus\{0\}$ and pick $ \xi\in\partial P_2(y)$. From the definition of the atomic norm \eqref{atomic}, there exist a positive integer $ m\leq n+1 $, a set of atoms $ \{s_i\}_{i=1}^m\subset \mathcal A $ and nonnegative coefficients $ \{r_i\}_{i=1}^m $ such that
\begin{equation}\label{atomrep}
 y = \|y\|_\mathcal{A} \sum\limits_{i=1}^m r_i s_i \;\text{and}\;  \sum\limits_{i=1}^m r_i =1 .
\end{equation}
Specifically, for ${\cal A} = {\cal A}_{\rm gl}$ and ${\cal A}_*$ in {\bf Scenarios 1} and {\bf 2}, respectively, we can derive the representation in \eqref{atomrep} with
\begin{enumerate}[{\rm (i)}]
	\item $ m = \lvert\mathcal{J}\rvert, $ $ r_i = \|y_{J_i}\|/\sum_{i=1}^m \|y_{J_i}\| $ and $ s_i = {\rm Sgn}(y_{J_i}) $ for $ i = 1,\ldots,m $, where we number the elements of ${\cal J}$ as $\{J_1,\ldots,J_m\}$,  $ y=[y_{J_i}]_{i=1}^m\in\R^n $ and ${\rm Sgn}$ is defined in \eqref{sign};
	\item $ m = {\rm rank}(Y)$, $ r_i = \lambda_i/\|Y\|_* $ and $ s_i = u_iw_i^T $ for $ i=1, \ldots, {\rm rank}(Y) $, where $ u_i $'s (resp., $ w_i $'s) are columns of $ U $ (resp., $ W $) from the thin SVD of $ Y= U\Lambda W^T $, and $ \lambda_i $'s are the diagonal entries of $ \Lambda $.
\end{enumerate}

With respect to \eqref{atomrep}, we define $ \mathcal I_y:=\{i:\, r_i\neq 0, i=1, \ldots, m\} $ and
\begin{equation}\label{vicidef}
v_i := \frac{\sigma s_i}{1-\langle\xi, s_i\rangle}, \;{\rm and}\;   c_i := \dfrac{r_i\|y\|_\mathcal{A}(1-\langle\xi, s_i\rangle)}{\sigma},\quad  \forall i=1, \ldots, m.
\end{equation}
It then follows that
\begin{equation}\label{ydcomp1}
y = \sum\limits_{i=1}^m c_i v_i = \sum\limits_{i\in\mathcal I_y} c_i v_i.
\end{equation}
Now, using Proposition~\ref{atomprop} and the definition of $ c_i $ and $ v_i $ in \eqref{vicidef}, and recalling that $ y\neq 0 $ and $ r_i>0 $ for every $ i\in\mathcal{I}_y $, one can deduce that $ v_i \in \bd \mathcal F(y,\xi) $ (defined in \eqref{Fyxi}) and $ c_i>0$ for every $ i\in\mathcal{I}_y $. Moreover, it holds that
\begin{equation}\label{sumci}
\sum\limits_{i\in\mathcal I_y} c_i = \sum\limits_{i\in\mathcal I_y} \frac{r_i\|y\|_\mathcal A - \|y\|_\mathcal A \langle\xi,r_i s_i\rangle }{\sigma} \overset{\rm (a)} = \frac{\|y\|_\mathcal A -\langle \xi, y \rangle}{\sigma} \overset{\rm (b)} \le 1,
\end{equation}
where (a) uses \eqref{atomrep} and the fact $ \sum_{i\in\mathcal I_y} r_i = \sum_{i=1}^m r_i=1 $, and (b) holds because $ y\in\mathcal F (y,\xi) $ (see \eqref{constrain}), $\xi\in \partial P_2(y)$ and $P_2$ is a norm. When $ \sum_{i\in {\mathcal{I}_y}} c_i = 1$, the decomposition \eqref{ydcomp1} of $ y $ satisfies the conditions in Definition~\ref{awdef}. Therefore, we can choose $ \mathcal{S}(y,\xi)=\{v_i:\;i\in\mathcal{I}_y\} $ {\color{blue} if $ \sum_{i\in {\mathcal{I}_y}} c_i = 1$}. Otherwise, choose any $ i_0\in\mathcal{I}_y $ and set $ \bar c=1-\sum_{i\in\mathcal{I}_y} c_i > 0$. Then
\begin{equation}\label{ydcomp2}
\begin{split}
	y = c_{i_0}v_{i_0} + \sum_{i\in\mathcal{I}_y\setminus\{i_0\}}c_i v_i \overset{\rm (a)}=& \sum_{i\in\mathcal{I}_y\setminus\{i_0\}}c_i v_i +  \left(c_{i_0}+\frac{\bar c(1-\langle\xi, s_{i_0}\rangle )}{2}\right)v_{i_0} \\
	&+ \frac{\bar c(1+\langle\xi, s_{i_0}\rangle )}{2}\cdot \frac{-\sigma s_{i_0}}{(1+\langle\xi, s_{i_0}\rangle)},
\end{split}
\end{equation}
where (a) uses the fact that $ v_{i_0}=\frac{\sigma s_{i_0}}{1-\langle\xi, s_{i_0}\rangle} $.
Note that $ \frac{-\sigma s_{i_0}}{1+\langle\xi, s_{i_0}\rangle}\in\mathcal{V}(y,\xi) $ since $-s_{i_0}\in {\cal A}$ and $ \|-s_{i_0}\|_{\cal A} = 1$. Furthermore, we can check directly that the decomposition of $ y $ in \eqref{ydcomp2} is a convex combination of elements in $ \mathcal{V}(y,\xi)\subset \mathcal F(y,\xi) $ (see Proposition~\ref{atomprop}(i) for the positivity of $1 \pm \langle\xi,s_{i_0}\rangle$). In summary, we can choose the set $ \mathcal S(y,\xi) $ in \eqref{aw-step} as
\begin{equation}\label{Syxi}
\mathcal{S}(y, \xi) = \begin{cases}
\left\{\frac{\sigma s_i}{1-\langle\xi, s_i\rangle}: i\in \mathcal I_y\right\} & {\rm if}\; \sum\limits_{i\in\mathcal{I}_y} c_i=1;\\
 \left\{\frac{\sigma s_i}{1-\langle\xi, s_i\rangle}: i\in \mathcal I_y\right\}\cup \left\{\frac{-\sigma s_{i_0}}{1+\langle\xi, s_{i_0}\rangle}\right\} & {\rm otherwise,}
\end{cases}
\end{equation}
where $i_0$ is an arbitrarily chosen element in ${\mathcal{I}}_y$.

Now we give a solution $u_{\rm aw}$ of \eqref{aw-step}, which can be found by maximizing $ \langle a, x\rangle $ over the finite discrete set $ \mathcal S(y, \xi) $. In particular, a solution of \eqref{aw-step} is
\begin{equation}\label{uaw}
u_{\rm aw} =
\begin{cases}
\frac{\sigma s_{i_*}}{1-\langle\xi, s_{i_*}\rangle}  & {\rm if}\; \sum\limits_{i\in\mathcal{I}_y} c_i=1 \;{\rm or}\; \langle a, \frac{-\sigma s_{i_0}}{1+\langle \xi, s_{i_0}\rangle}\rangle < \langle a, \frac{\sigma s_{i_*}}{1-\langle \xi, s_{i_*}\rangle}\rangle;  \\
\frac{-\sigma s_{i_0}}{1+\langle \xi, s_{i_0}\rangle} & {\rm otherwise,}
\end{cases}
\end{equation}
where $ i_*\in\Argmax_{i\in\mathcal I_y} \left\langle a, \frac{\sigma s_i}{1-\langle\xi, s_i\rangle}\right\rangle  $. Recall that $i_0$ can be chosen to be any element in ${\cal I}_y$. In our numerical experiments in Section~\ref{sec7} below, for simplicity, we choose $i_0 = i_*$ in \eqref{uaw}.

\subsubsection{Choosing the $ \alpha_{\rm aw} $ in Step 2 of Algorithm~\ref{alg:aw-fw}}\label{subsec612}

Let $ P_1 $ and $ P_2 $ be as in Assumption~\ref{assump31} or \ref{assump32}.
Given $ y \in\mathcal F\setminus\{0\}$ and $ \xi\in\partial P_2(y)$, with $ \mathcal S(y,\xi) $ and $ u_{\rm aw} $ determined as in \eqref{Syxi} and \eqref{uaw}, we discuss how to find an $ \alpha_{\rm aw} \le \max\{\alpha\geq 0:y+\alpha d_{\rm aw}\in {\rm conv}(\mathcal S(y,\xi))\} $ along the away-step direction $ d_{\rm aw}  = y - u_{\rm aw}$. Note that the sets $ \mathcal S(y, \xi) $ in Section~\ref{sec:syxi} are all discrete sets with finite elements. For notational simplicity, we write
\[
\mathcal S(y,\xi) = \{v_1, \ldots, v_{q} \}
\]
with $ q=\lvert\mathcal{I}_y\rvert $ or $ \lvert\mathcal{I}_y\rvert+1 $, where $\mathcal{I}_y$ is defined as in Section~\ref{sec:syxi}. According to \eqref{ydcomp2} and \eqref{uaw}, we see that $ u_{\rm aw}\in\mathcal S(y,\xi) $ and
\begin{equation}\label{yrep}
	y = \sum_{i=1}^{q} \widetilde c_i v_i
\end{equation}
with $\widetilde c_i> 0 $ for every $ i=1,\ldots,q $ and $ \sum_{i=1}^{q}\widetilde c_i=1 $. Suppose that $ u_{\rm aw} = v_{i_1} $ for some $  v_{i_1}\in \mathcal S(y,\xi)$. Then one can choose
\begin{equation*}
	\alpha_{\rm aw} \le \max\{\alpha\geq 0:\; y+\alpha(y - v_{i_1}) \in {\rm conv} (\mathcal S(y, \xi))\}.
\end{equation*}
Define $ y(\alpha) = y+\alpha(y-v_{i_1}) $. Using \eqref{yrep}, we have
\[
\begin{split}
y(\alpha) = (1+\alpha)\sum\limits_{i\neq i_1}\widetilde c_i v_i + ((1+\alpha)\widetilde c_{i_1} -\alpha) v_{i_1}.
\end{split}
\]
To ensure that $ y(\alpha)\in {\rm conv}(S(y,\xi)) $, we impose the following conditions on $ \alpha $:
\begin{align}
	&(1+\alpha)\sum_{i\neq i_1}\widetilde c_i + (1+\alpha)\widetilde c_{i_1}-\alpha = 1, \label{alpha_ineq2} \\
    &(1+\alpha)\widetilde c_{i_1} -\alpha \geq 0. \label{alpha_ineq1}
\end{align}
Since $ \sum_{i=1}^{q} \widetilde c_i=1 $ (thanks to \eqref{sumci} and \eqref{ydcomp2}), we have $\sum_{i\neq i_1}\widetilde c_i =1 - \widetilde c_{i_1}$. It follows that
\[
(1+\alpha)\sum_{i\neq i_1}\widetilde c_i + (1+\alpha)\widetilde c_{i_1}-\alpha= (1+\alpha)(1-\widetilde c_{i_1})+(1+\alpha)\widetilde c_{i_1} - \alpha  = 1.
\]
That is, \eqref{alpha_ineq2} holds automatically for any $ \alpha\geq 0 $. So \eqref{alpha_ineq1} is sufficient to ensure that $ y(\alpha)\in{\rm conv}(\mathcal S(y,\xi)) $. We can thus choose
\begin{equation*}
\alpha_{\rm aw} = \min\left\{\frac{\widetilde c_{i_1}}{1-\widetilde c_{i_1}},\zeta\right\},
\end{equation*}
where $\zeta$ is given in Step 0 of Algorithm~\ref{alg:aw-fw}.
Here, we would like to mention that the expression $\frac{\widetilde c_{i_1}}{1-\widetilde c_{i_1}}$ coincides with the formula of the feasible stepsize $ \gamma_{\rm max} $ given in the FW method with {\color{blue}away steps} in convex setting proposed in \cite{Lacoste15}.

{\color{blue}
\section{Numerical experiments on matrix completion}\label{sec7}
In this section, we illustrate the performance of our methods for solving matrix completion problems (MC), which aims to recover a low rank matrix given some observed entries. A typical convex model \cite{FrGM17,Jaggi10} for MC is of the form:
\begin{equation}\label{mc:convex}
\begin{array}{cl}
\min\limits_{x\in\R^{m\times n}} & \frac12\sum\limits_{(i,j)\in\Omega}(x_{ij}-\bar x_{ij})^2\\
\text{s.t.} & \|x\|_* \leq \sigma,
\end{array}
\end{equation}
where $\Omega$ is a collection of indices, {\color{blue} $\bar x \in \R^{m \times n}$ is a matrix such that, for all $(i,j) \in \Omega$, $\bar x_{ij}$ are the given observations, {\color{blueblue} and} $\sigma > 0$; here the nuclear norm is used to induce a low rank structure. Recently, the DC regularizer $ \|x\|_* - \|x\|_F $  is used for low rank matrix completion and its numerical advantage was discussed in \cite{MaLou17}. Motivated by this, we consider in this section an MC model with a DC regularizer in the constraint as follows:
\begin{equation}\label{mc:ncvx}
\begin{array}{cl}
\min\limits_{x\in\R^{m\times n}} & f(x):=\frac12\sum\limits_{(i,j)\in\Omega}(x_{ij}-\bar x_{ij})^2 \\
\text{s.t.} & \|x\|_* - \mu \|x\|_F \leq \sigma,
\end{array}
\end{equation}
where $\mu \in (0,1)$.\footnote{\color{blueblue} Note that \eqref{mc:ncvx} satisfies  Assumption \ref{assumption1} (see the discussion after Assumption \ref{assumption1}), which further implies  the gMFCQ  in view of   Proposition \ref{slaequ}(i). This together with Proposition \ref{loc-sta} shows that any local minimizer of \eqref{mc:ncvx} is a stationary point of \eqref{mc:ncvx}.}
For the above two matrix completion models, we adapt \textbf{IF-$ (0, \infty)$} in \cite[Section~4]{FrGM17} to solve \eqref{mc:convex}, and adapt our Algorithms~\ref{FW_nonconvex} and \ref{alg:aw-fw} to solve \eqref{mc:ncvx} with $P_1(x) = \|x\|_*$ and $P_2(x) = \mu \|x\|_F$.\footnote{\color{blue}We do not consider the use of a strongly convex $P_1(x) = \|x\|_* + \frac{\rho}2\|x\|_F^2$ for some $\rho > 0$ as described in Section~\ref{sec:scsubp} because, as discussed in Remark~\ref{rem4}, the corresponding subproblem then involves computing the prohibitively expensive proximal operator in \eqref{xprox} and forming the huge matrix $\xi-\iota_* a$.} All algorithms are implemented in MATLAB,\footnote{\color{blue}We also borrow some C-codes from the implementation of \textbf{IF-$ (0, \infty)$} in \cite[Algorithm~3]{FrGM17}, which incorporated these C implementations as mex files.}
and the experiments are performed in {Matlab {\color{blueblue}2022b} on a 64-bit PC with an Intel(R) Core(TM) {\color{blueblue}i7-4790 CPU (@3.60GHz, 3.60GHz) and 32GB of RAM.}}
The MATLAB solver of the in-face extended Frank-Wolfe method \cite[Algorithm~3]{FrGM17} that implements \textbf{IF-$ (0, \infty)$} is downloaded from:\\ \url{https://github.com/paulgrigas/InFaceExtendedFW-MatrixCompletion}. \\
For convenience, we refer to Algorithms~\ref{FW_nonconvex} and \ref{alg:aw-fw} (for solving \eqref{mc:ncvx}) as \textbf{FW}$_{\bf ncvx}$ and \textbf{AFW}$_{\bf ncvx}$ respectively, and refer to the in-face extended Frank-Wolfe solver (for solving \eqref{mc:convex}) as {\sf InFaceFW-MC} below.


{\color{blueblue}
In the tests, we use the standard datasets {\sf MovieLens10M, MoviLens20M, MovieLens32M} and {\sf  Netflix Prize} that contain ratings data  for movies given by  users. See Table \ref{table:dataset} for details about these datasets.}\footnote{
{\color{blueblue}The raw datasets {\sf MovieLens10M, MovieLens20M} and {\sf MovieLens32M}
have $(m,n)=(71567,10681)$, $(138493,27278)$ and $(200971,87585)$, respectively.}
Here, we have removed user indices and movie indices without any rating records as in \cite{FrGM17}.
{\color{blueblue}The MovieLens and Netflix Prize datasets can be downloaded from \href{https://files.grouplens.org/datasets/movielens/}{https://files.grouplens.org/datasets/movielens/} and
\href{https://archive.org/download/nf\_prize\_dataset.tar}{https://archive.org/download/nf\_prize\_dataset.tar}, respectively.}}
\begin{table}[ht]
\caption{Datasets}
\centering
{\color{blueblue}
\begin{tabular}{c c c c c c}
\hline
Datasets & $m$(users)  & $n$(movies)  & $|\Omega|$(ratings)  &  $|\Omega_{\rm test}|$    & $|\Omega_{\rm all}|$ \\
\hline
MovieLens10M & 69,878 &  10,677    &  7,001,117 & 2,998,937  & 10,000,054\\
MovieLens20M & 138,493 & 26,744  & 13,998,676 & 6,001,587  & 20,000,263 \\
MovieLens32M & 200,971 & 84,349   & 22,405,373 & 9,594,709  & 32,000,082 \\
Netflix Prize & 480,189 & 17,770    & 70,339,414 & 30,141,093 & 100,480,507\\
\hline
\end{tabular} \label{table:dataset}
}
\end{table}
Following the implementations in {\sf InFaceFW-MC}, we perform a pre-processing to centralize\footnote{\color{blue}The original code is from the file \texttt{read\_movielens\_10M.m} in {\sf InFaceFW-MC}.} the rating data, and then split the processed data into two parts \footnote{\color{blue}We use the \texttt{split\_matcomp\_instance.m} file in {\sf InFaceFW-MC} to split the dataset.} with 70\%  {\color{blueblue}data} used as the training dataset (i.e., the observation $ \bar x $ in \eqref{mc:convex} and \eqref{mc:ncvx}), and the remaining $ |\Omega_{\rm test}| $ (=$|\Omega_{\rm all}| - |\Omega|$) data points used as the testing dataset to check the recovery efficiency of the algorithms. For the convex model \eqref{mc:convex} considered in \cite{FrGM17}, a cross-validation was performed in {\sf InFaceFW-MC} to learn the quantity $ \sigma$ and {\color{blueblue}the same $\sigma$ is also used in model \eqref{mc:ncvx}.\footnote{\color{blueblue}We use the codes \texttt{InFace\_Extended\_FW\_sparse\_path.m} and \texttt{find\_best\_delta.m}  in {\sf InFaceFW-MC} to find  $\sigma$.}}
 {\color{blueblue}Moreover,  we also set $\mu=0.5, 0.5$ and $0.75$  in \eqref{mc:ncvx} for testing the effects of this parameter}.}

We now present the algorithmic settings for \textbf{IF-$ (0, \infty)$}, \textbf{FW}$_{\bf ncvx}$ and \textbf{AFW}$_{\bf ncvx}$. The parameters used in \textbf{IF-$ (0, \infty)$} are in default settings as the {\sf InFaceFW-MC} solver except for the maximal time limit, which we will specify below. Next we discuss the implementation details of our algorithms. We set $ c=10^{-4} $, $ \eta=1/2 $ and choose $ \alpha^0_0=1 $ and, for $k\geq 1 $,
\begin{equation*}
\alpha_k^0 = \begin{cases}
\max\{ 10^{-8}, \min\{2\alpha^{\rm fw}_k, 1\}\} & \text{if }  d^{k-1} = d^{k-1}_{\rm fw} \text{ and } j_{k-1}=0,\\
\max\{ 10^{-8}, \min\{\alpha^{\rm fw}_k, 1\}\} & \text{otherwise},
\end{cases}
\end{equation*}
where $ \alpha_k^{\rm fw} $ is the stepsize used in the most recent FW step.\footnote{\color{blue}When no FW step has been used yet, this quantity is set to $1$.} In the tests of \textbf{FW}$_{\bf ncvx}$ and \textbf{AFW}$_{\bf ncvx}$ for \eqref{mc:ncvx}, we compute a closed-form solution of $\LO$ according to Theorem~\ref{cfs-mat} with $ a=\nabla f(x^k) $ and $ \xi=\xi^k:=\argmin_{u\in \mu \partial \|x^k\|_F} \{\|u\|\} $. The generalized eigenvalue problem involved in Theorem~\ref{cfs-mat} is solved via {\sf eigifp}. When $ k\geq 1 $, we use the output of {\sf eigifp} in the previous iteration as a warm-start, and choose the $e/{\sqrt{m+n}}$ as the initial vector of {\sf eigifp} when $ k=0 $, where $e$ is the vector of all ones of dimension $ m+n $. The accuracy tolerance of {\sf eigifp} is set as $ 10^{-6} $, and we use its default maximum number of iterations. Following the implementations in {\sf InFaceFW-MC}, we do not form $ x^k $ in each iteration of \textbf{FW}$_{\bf ncvx}$ and \textbf{AFW}$_{\bf ncvx}$, but maintain the thin SVD triple $ (R^k, \Lambda^k, W^k) $ of $ x^k  = R^k\Lambda^k (W^k)^T$, and we form $ \nabla f(x^k) $ by computing $ x^k_{ij} $ for $ (i,j)\in\Omega $ only. Note that both $ u^k_{\rm fw} $ and $ u^k_{\rm aw} $ in \textbf{FW}$_{\bf ncvx}$ and \textbf{AFW}$_{\bf ncvx}$ are rank-one matrices thanks to Theorem~\ref{cfs-mat} and item (ii) under \eqref{atomrep}. Therefore, the thin SVD of $ x^{k+1}= x^k + \alpha_k d^k$ with $ d^k=u^k_{\rm fw}-x^k $ or $ x^k-u^k_{\rm aw} $ can be obtained efficiently via the rank-one SVD update technique proposed in \cite{Brand06}. Here in our implementation, we use \texttt{svd\_rank\_one\_update1.m} in {\sf InFaceFW-MC} for performing the SVD rank-one update. Moreover, when applying \textbf{AFW}$_{\bf ncvx}$ for solving \eqref{mc:ncvx}, we set $ \epsilon=10^{-5} $ and $ \zeta=10^5 $, and we compute a solution of $ \AWO(x^k, \mathcal S(x^k,\xi^k)) $ through \eqref{uaw} with $ a=\nabla f(x^k) $, $ \xi=\xi^k:=\argmin_{u\in \mu \partial \|x^k\|_F} \{\|u\|\}$ and $ \mathcal{I}_y = \{i:\Lambda^k_{ii} > 10^{-6} \} $.
Furthermore, leveraging the positive homogeneity of $ \|\cdot\|_* - \mu\|\cdot\|_F $, we adopt the strategy in \eqref{choosexk+1} to determine $ x^{k+1} $ in both \textbf{FW}$_{\bf ncvx}$ and \textbf{AFW}$_{\bf ncvx}$.\footnote{\color{blue}Note that $\|\cdot\|_* - \mu\|\cdot\|_F$ can be computed efficiently since we maintain the thin SVD of the iterates.} Finally, we terminate \textbf{FW}$_{\bf ncvx}$ and \textbf{AFW}$_{\bf ncvx}$ when the computational time reaches a given upper bound $ T^{\max} $, or the iteration number reaches 40000, or the criterion $ |\langle \nabla f(x^k), d^k\rangle| < 10^{-6} \max\{|f(x^k)+\langle \nabla f(x^k), d^k\rangle|, 1 \} $ is satisfied.


In the test, we initialize both \textbf{FW}$_{\bf ncvx}$ and \textbf{AFW}$_{\bf ncvx}$ at the zero matrix.  Regarding the initial point of \textbf{IF-$ (0, \infty)$}, we use its default options.\footnote{\color{blue}Specifically, it does not start from the zero matrix, {\color{blueblue} but from $ x^0$, where $ x^0 $ is an approximate solution of the problem $ \min\{\langle\nabla f(0), x\rangle:\, \|x\|_*\leq \sigma  \} $.}}  {\color{blueblue}Figures~\ref{fig:10M}--\ref{fig:Netflix}} show the training error $f(x^k) $ and the testing error $ \frac12\sum_{(i,j)\in \Omega_{\rm test}} ({\color{blue} x_{ij}^k} - \bar{x}_{ij})^2 $ of the three algorithms {\color{blueblue} for  datasets  in  Table \ref{table:dataset} with $\mu$ taking different values 0.25, 0.5 and 0.75} as the algorithm progresses, where $ \Omega_{\rm test} $ collects the indexes of the testing data points.
{\color{blueblue}For all the datasets, with the lapse of time,}
we see that the algorithms \textbf{FW}$_{\bf ncvx}$ and \textbf{AFW}$_{\bf ncvx}$ achieve better accuracy (with respect to both the training and test data set) than \textbf{IF-$ (0, \infty)$},  {\color{blueblue}except that the test errors of the Nexflix dataset of our methods are slightly worse than those of \textbf{IF-$ (0, \infty)$} when $\mu = 0.25$; see Figure \ref{fig:Netflix}.}
{\color{blueblue}On the other hand, as $\mu$ grows from $0.25$ to $0.75$, one can observe that the performances of \textbf{FW}$_{\bf ncvx}$ and \textbf{AFW}$_{\bf ncvx}$ become better.}

In Table~\ref{table:1}, we show the quantity
\[
\widehat\varepsilon= |\langle \nabla f(x^{k^*-1}), d^{k^*-1}\rangle| /\max\{|f(x^{k^*-1})+\langle \nabla f(x^{k^*-1}), d^{k^*-1}\rangle|, 1 \},
\]
as well as the rank\footnote{\color{blue}We count the number of singular values that exceed $10^{-6}$.} and the root mean squared error {\color{blueblue} on the testing data points (RMSE)} of $ x^{\rm out} $ {\color{blueblue} for each dataset with $\mu=0.25$, $0.5$ and $0.75$}, where $ \widehat\varepsilon $ can be seen as an approximate optimality measure,  $ k^* $ represents the number of iterations executed within {\color{blueblue}$T^{\max}$} and $ x^{\rm out} $ is computed from $ x^* $ using \emph{a reverse operation} of the centralizing operation in the pre-processing of \textbf{MovieLens10M} as in \cite{FrGM17}.\footnote{\color{blue}We note that this is different from the errors plotted in  {\color{blueblue}Figures~\ref{fig:10M}--\ref{fig:Netflix}}, where the $\bar x$ was obtained by centralizing the rating data.} The table demonstrates that our methods always obtain better recovery results within the same computational time. On the other hand, by comparing \textbf{AFW}$_{\bf ncvx}$ and \textbf{FW}$_{\bf ncvx}$, we see that the ranks of the solutions obtained by \textbf{AFW}$_{\bf ncvx}$ are much smaller, i.e., \textbf{AFW}$_{\bf ncvx}$ is a better option when a solution of small rank is preferred.\footnote{{\color{blueblue}Here, we give an intuitive explanation of the observation that \textbf{AFW}$_{\bf ncvx}$ yields a solution of lower rank. Indeed, in view of the construction of away-step oracle \eqref{aw-step}, \eqref{Syxi}, \eqref{uaw} and the update formula for $x^{k+1}$ in Step 5 of Algorithm \ref{alg:aw-fw} in Section \ref{sec6}, the next iterate $x^{k+1}$ satisfies either ${\rm rank}(x^{k+1})= {\rm rank}(x^{k})$ or ${\rm rank}(x^{k+1})= {\rm rank}(x^{k})-1$ when the AW step is accepted in the $k$th iteration; see Section~\ref{subsec612}. On the other hand, if the FW step is used in the $k$th iteration and $\alpha_k < 1$ (this condition almost always holds in our experiments), then the rank can change by at most $1$, and empirically we observe that the rank keeps increasing.
We have to emphasize that although it is observed numerically that \textbf{AFW}$_{\bf ncvx}$ yields an approximate solution with a lower rank compared with that of \textbf{FW}$_{\bf ncvx}$, there is no theoretical guarantee for the relative magnitude of the ranks of the solutions returned by these methods.} }




\captionsetup[figure]{labelfont={color=blueblue},font={color=blueblue}}
\begin{figure}[h]
	\caption{Training error and test error of \textbf{IF}-$ (0, \infty)$, \textbf{FW}$_{\bf ncvx}$ and \textbf{AFW}$_{\bf ncvx}$ for {\sf MovieLens10M}.  The figures in the first, second and third row correspond to  $\mu = 0.25$, $0.5$ and $0.75$, respectively.}
	\label{fig:10M}
	\centering
	\includegraphics[scale=0.191]{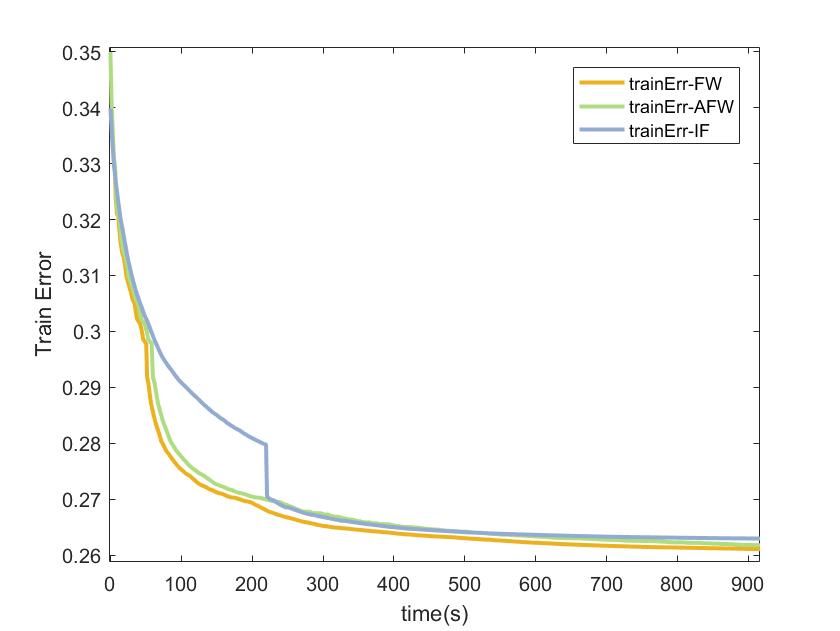}
	\includegraphics[scale=0.191]{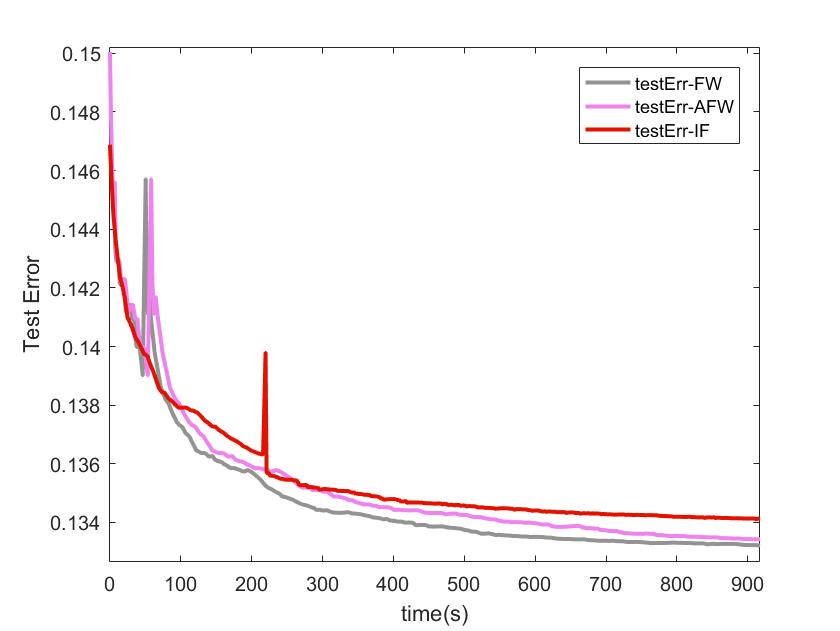}

	\includegraphics[scale=0.191]{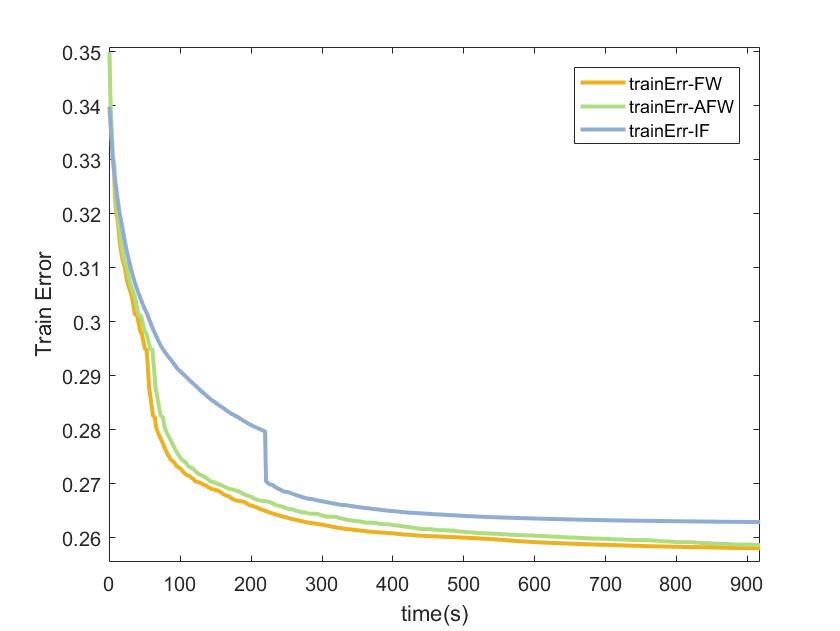}
	\includegraphics[scale=0.191]{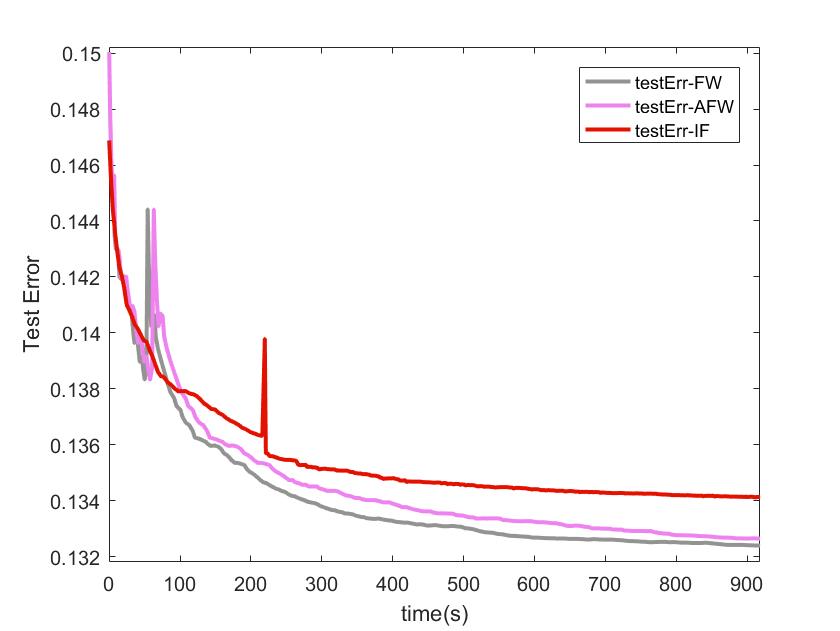}

	\includegraphics[scale=0.191]{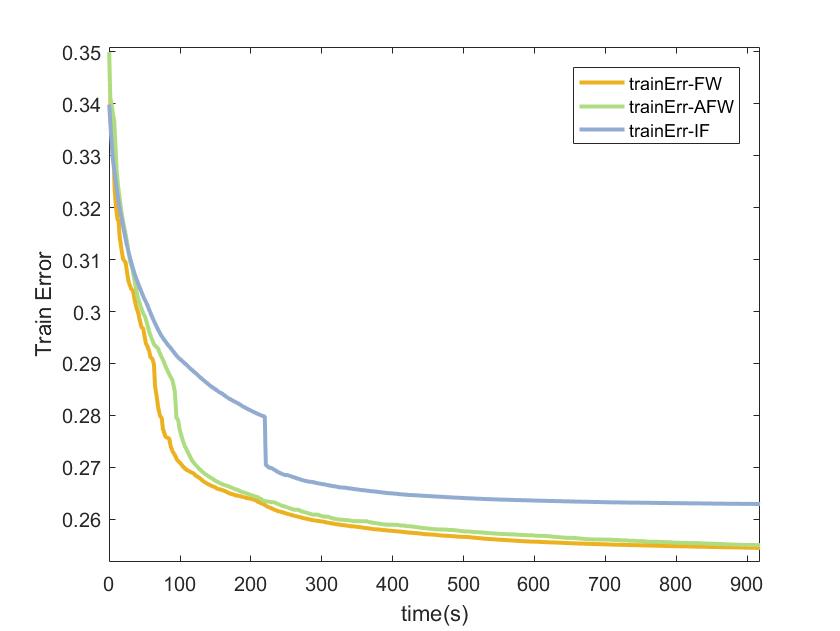}
	\includegraphics[scale=0.191]{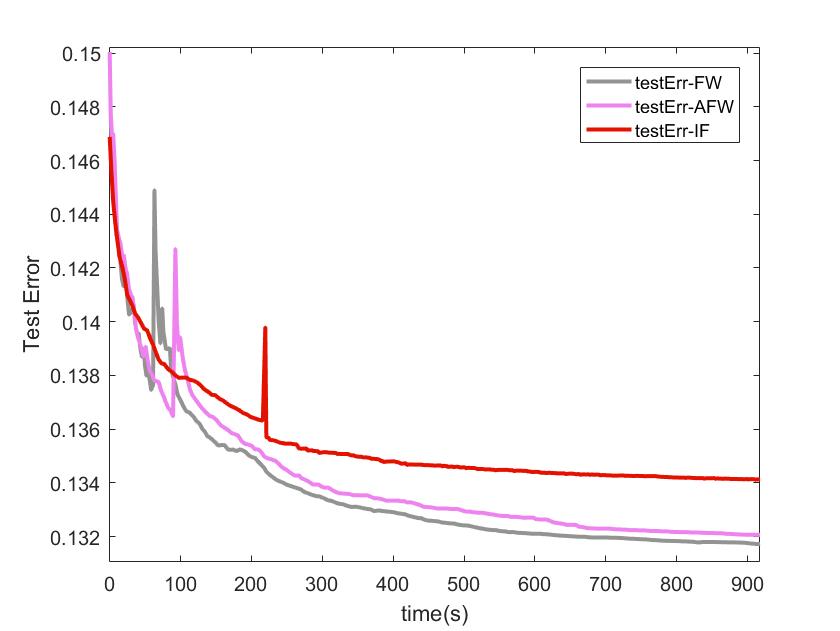}
\end{figure}

\captionsetup[figure]{labelfont={color=blueblue},font={color=blueblue}}
\begin{figure}[h]
	\caption{Training error and test error of \textbf{IF}-$ (0, \infty)$, \textbf{FW}$_{\bf ncvx}$ and \textbf{AFW}$_{\bf ncvx}$ for {\sf MovieLens20M}.  The figures in the first, second and third row correspond to  $\mu = 0.25$, $0.5$ and $0.75$, respectively.}
	\label{fig:20M}
	\centering
	\includegraphics[scale=0.191]{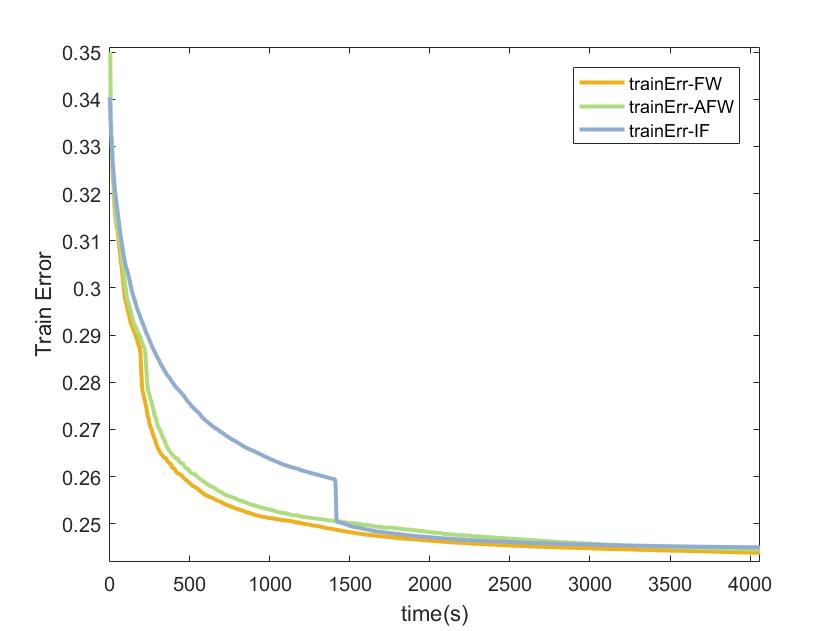}
	\includegraphics[scale=0.191]{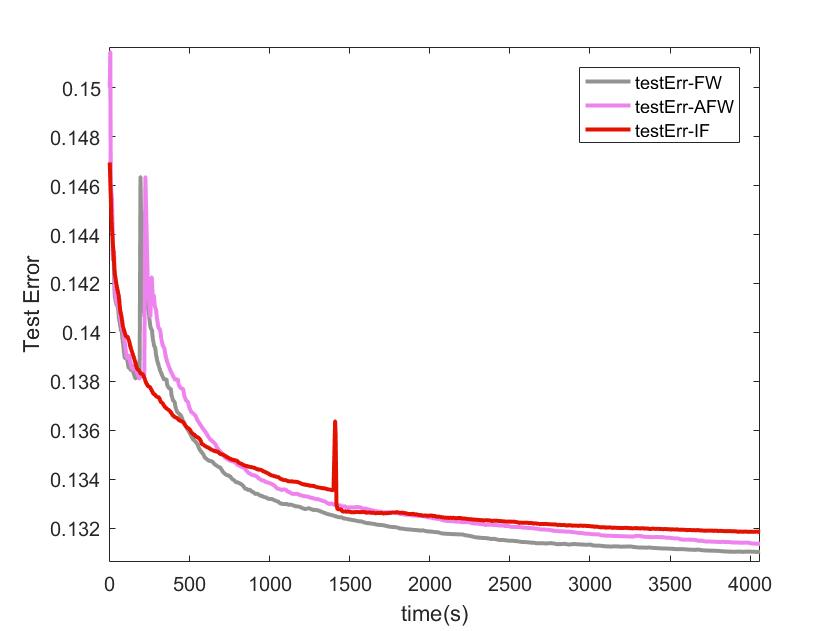}

	\includegraphics[scale=0.191]{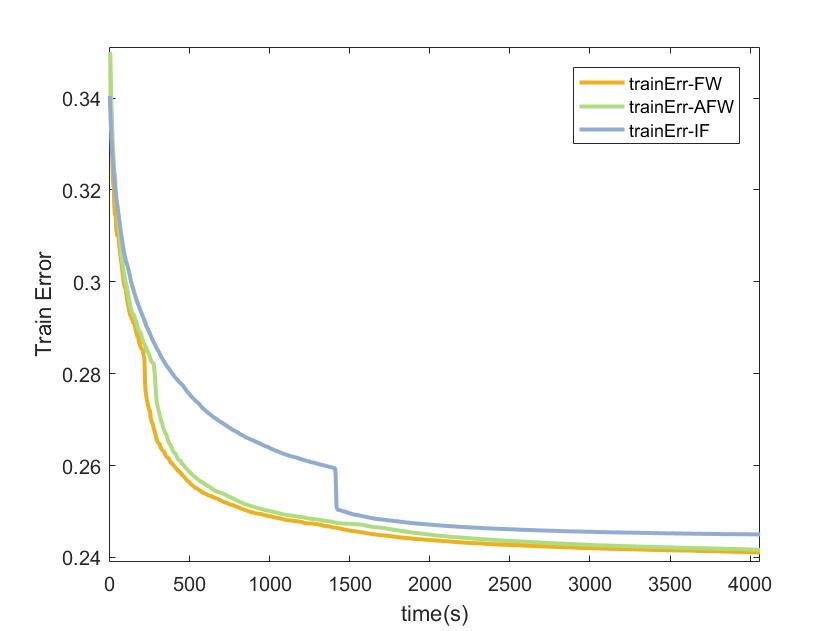}
	\includegraphics[scale=0.191]{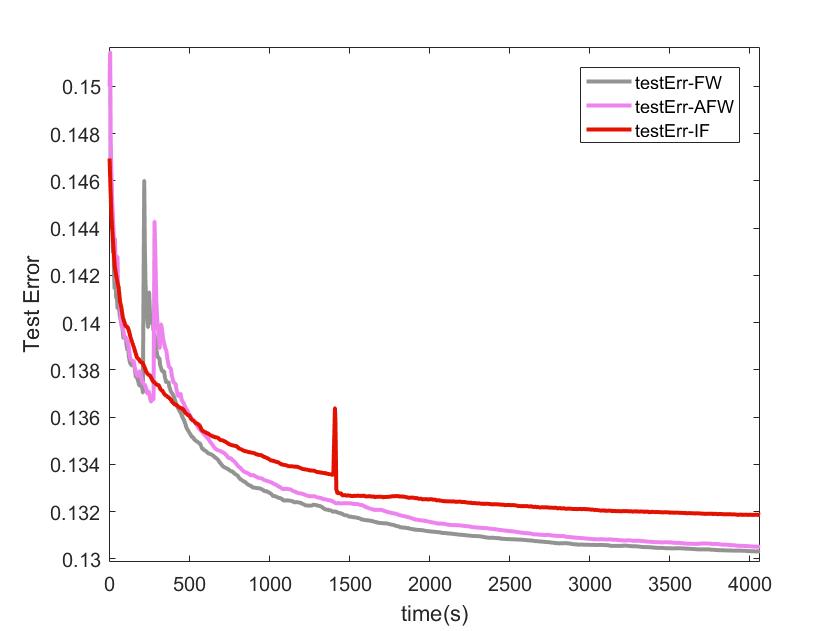}

	\includegraphics[scale=0.191]{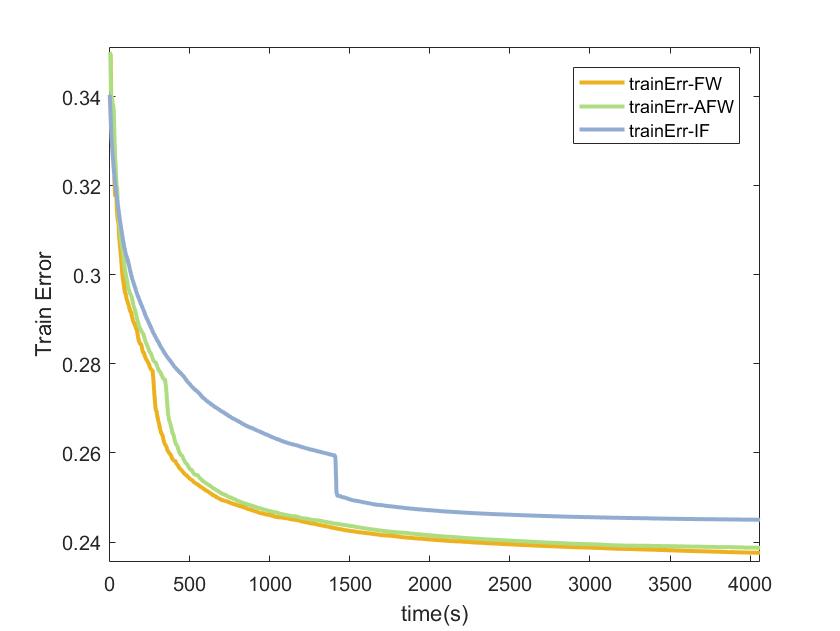}
	\includegraphics[scale=0.191]{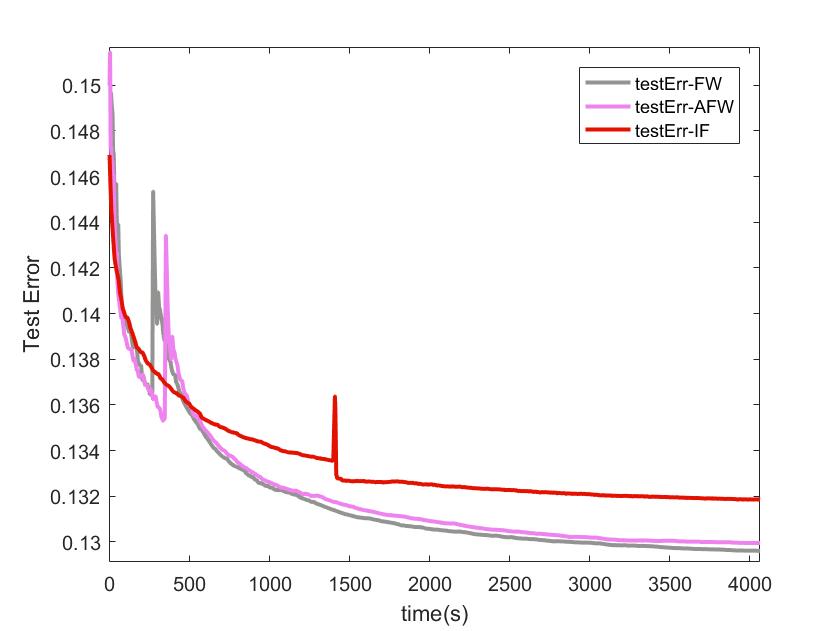}
\end{figure}

\captionsetup[figure]{labelfont={color=blueblue},font={color=blueblue}}
\begin{figure}[h]
	\caption{Training error and test error of \textbf{IF}-$ (0, \infty)$, \textbf{FW}$_{\bf ncvx}$ and \textbf{AFW}$_{\bf ncvx}$ for {\sf MovieLens32M}. The figures in the first, second and third row correspond to  $\mu = 0.25$, $0.5$ and $0.75$, respectively.}
	\label{fig:32M}
	\centering
	\includegraphics[scale=0.191]{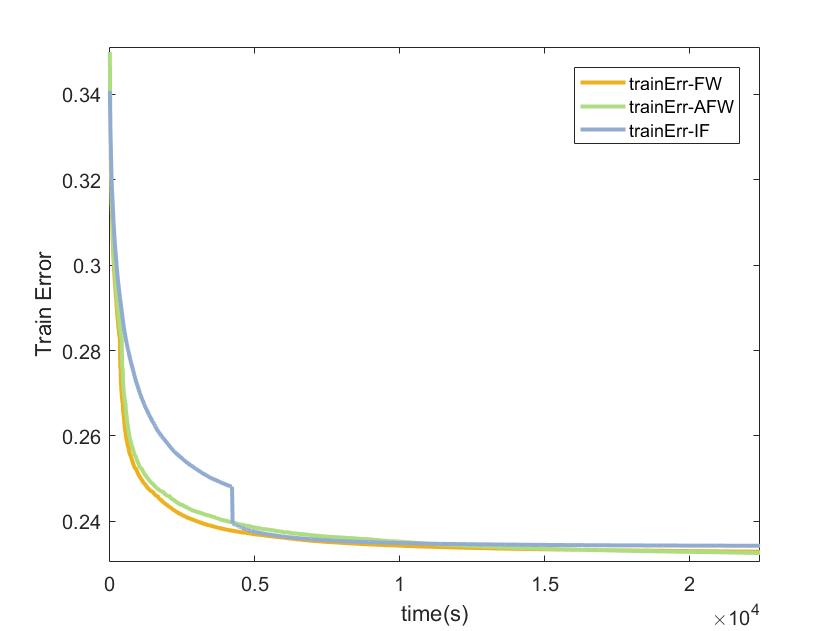}
	\includegraphics[scale=0.191]{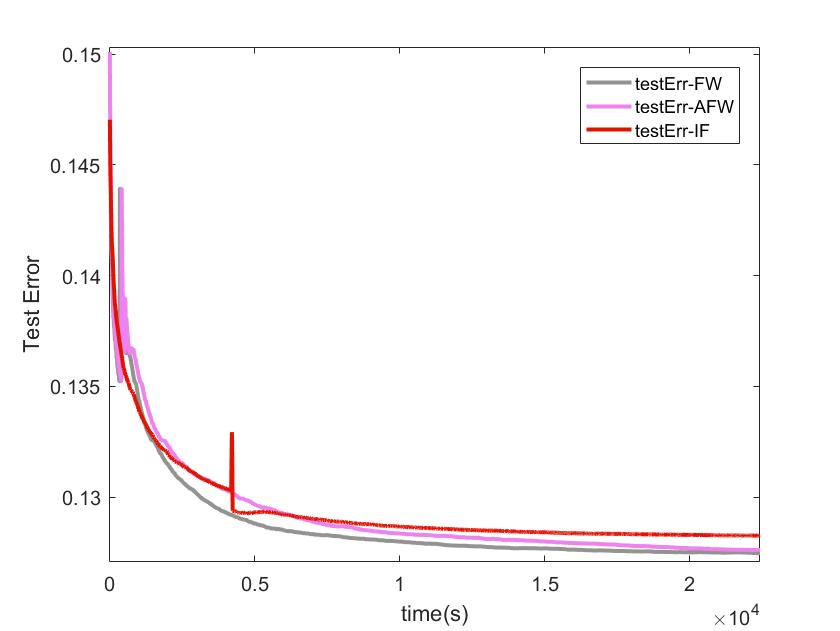}

	\includegraphics[scale=0.191]{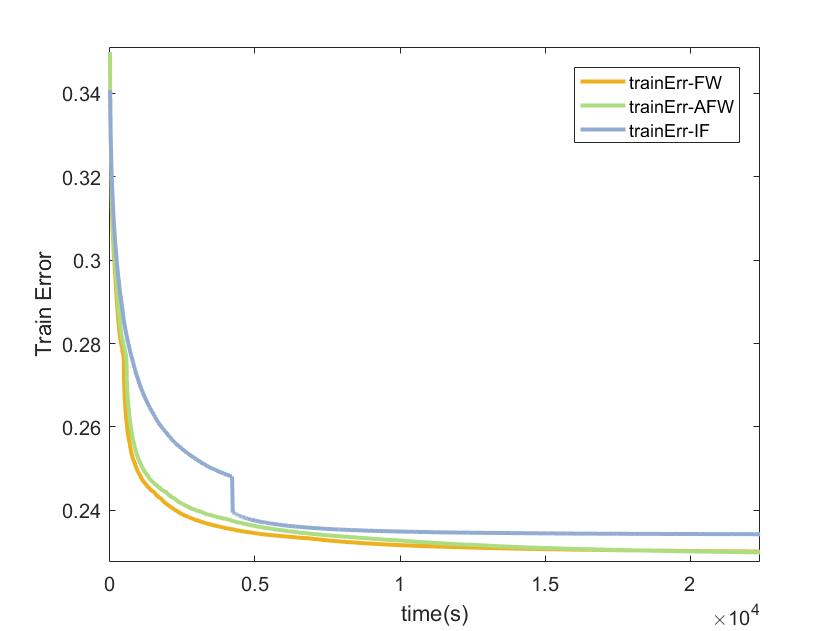}
	\includegraphics[scale=0.191]{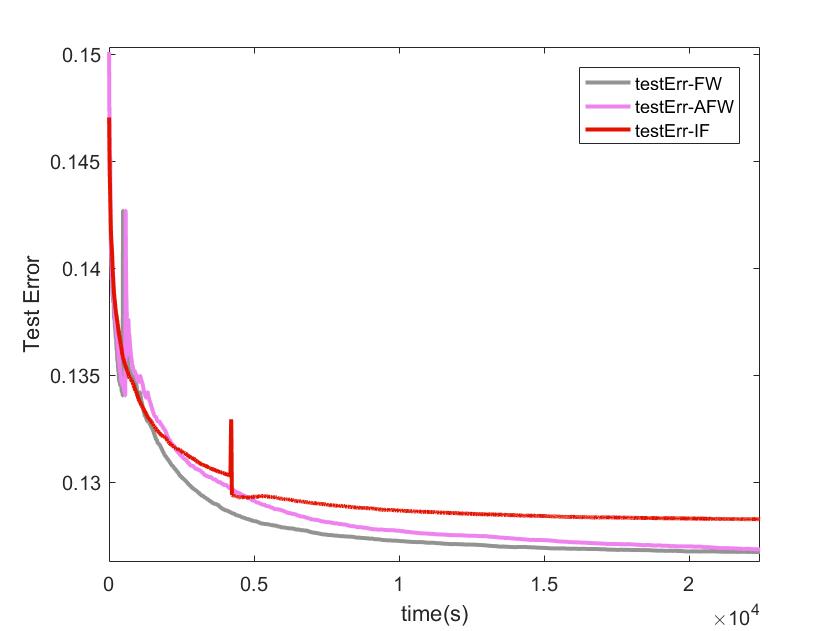}

    \includegraphics[scale=0.191]{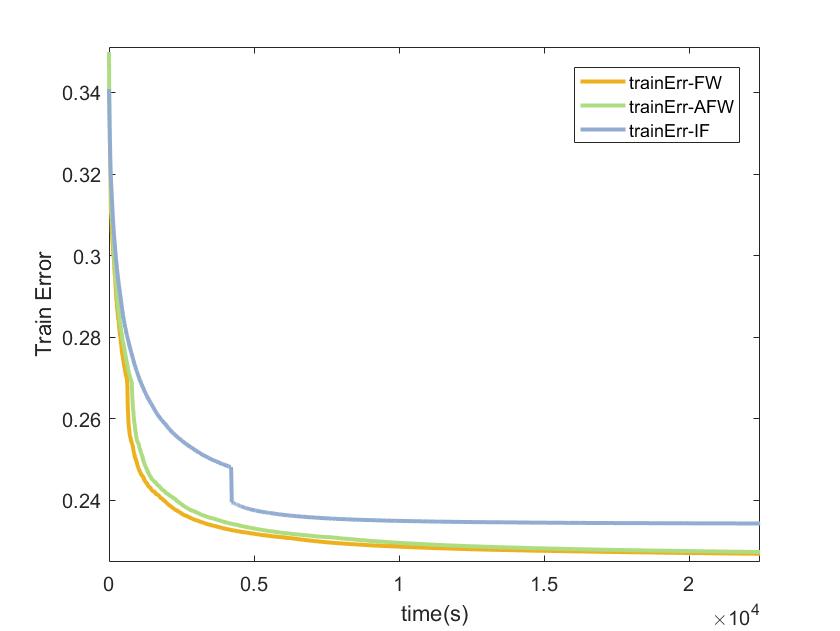}
	\includegraphics[scale=0.191]{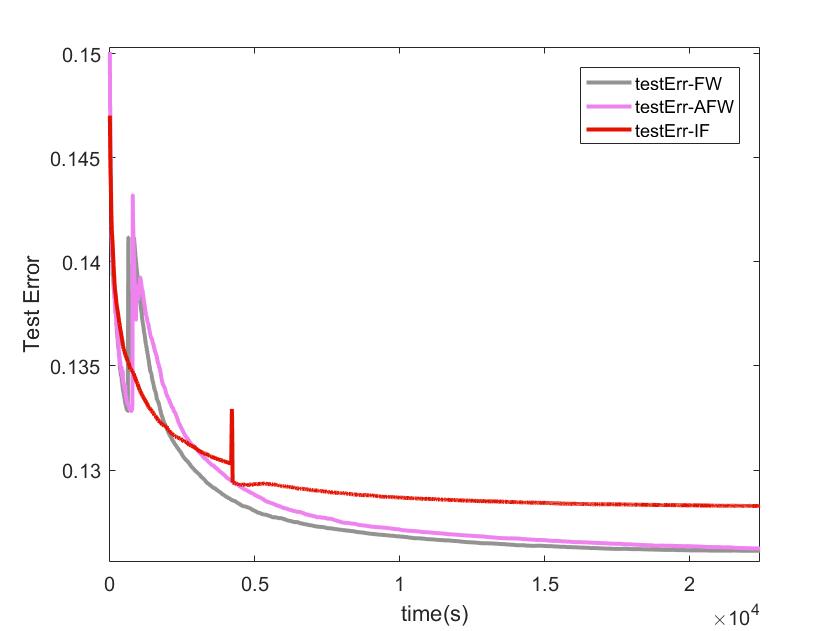}
\end{figure}

\captionsetup[figure]{labelfont={color=blueblue},font={color=blueblue}}
\begin{figure}[h]
	\caption{Training error and test error of \textbf{IF}-$ (0, \infty)$, \textbf{FW}$_{\bf ncvx}$ and \textbf{AFW}$_{\bf ncvx}$ for {\sf Netflix Prize}. The figures in the first row, the second row and the third row correspond to  $\mu = 0.25$, $0.5$ and $0.75$, respectively.}
	\label{fig:Netflix}
	\centering
	\includegraphics[scale=0.191]{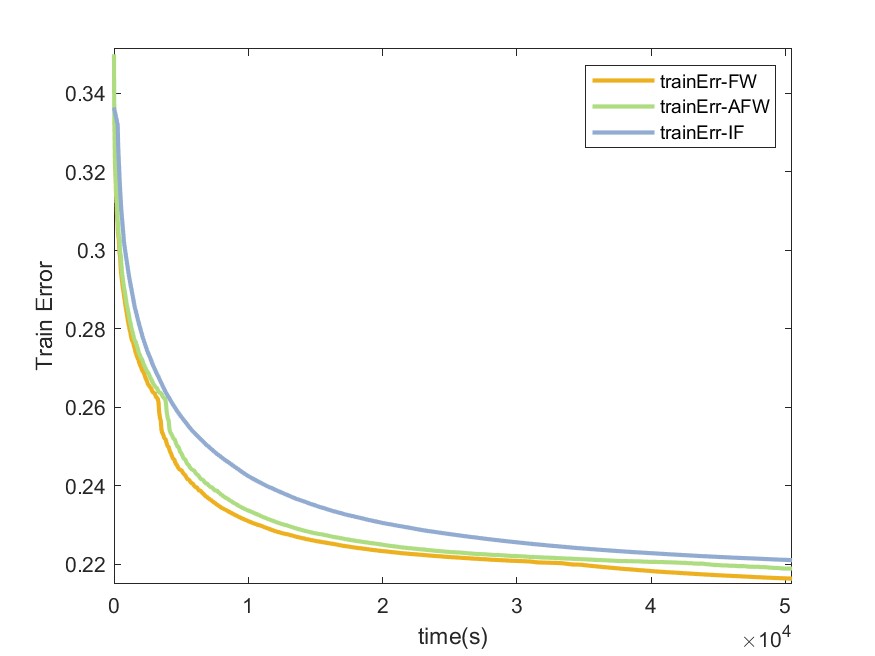}
	\includegraphics[scale=0.191]{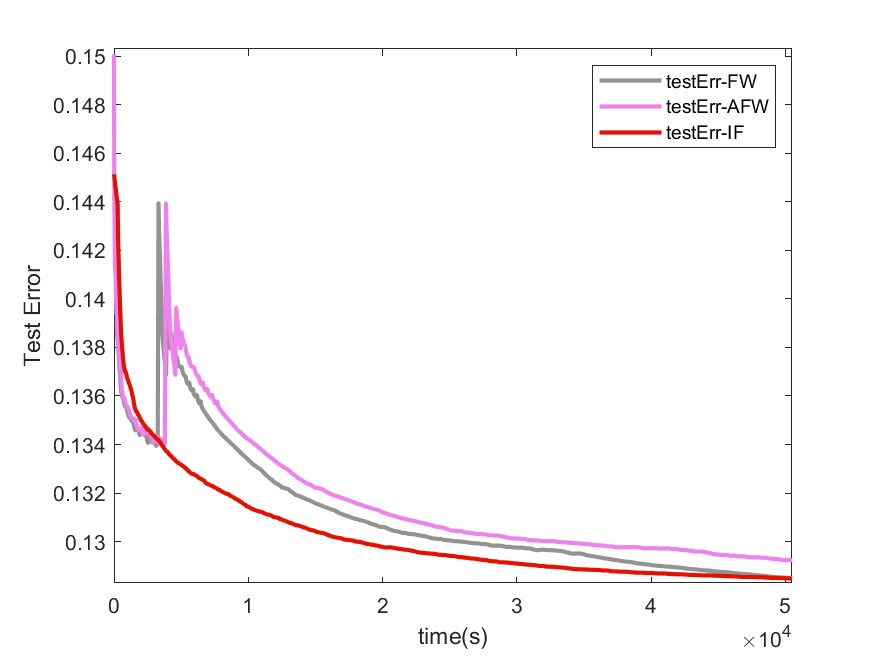}

	\includegraphics[scale=0.191]{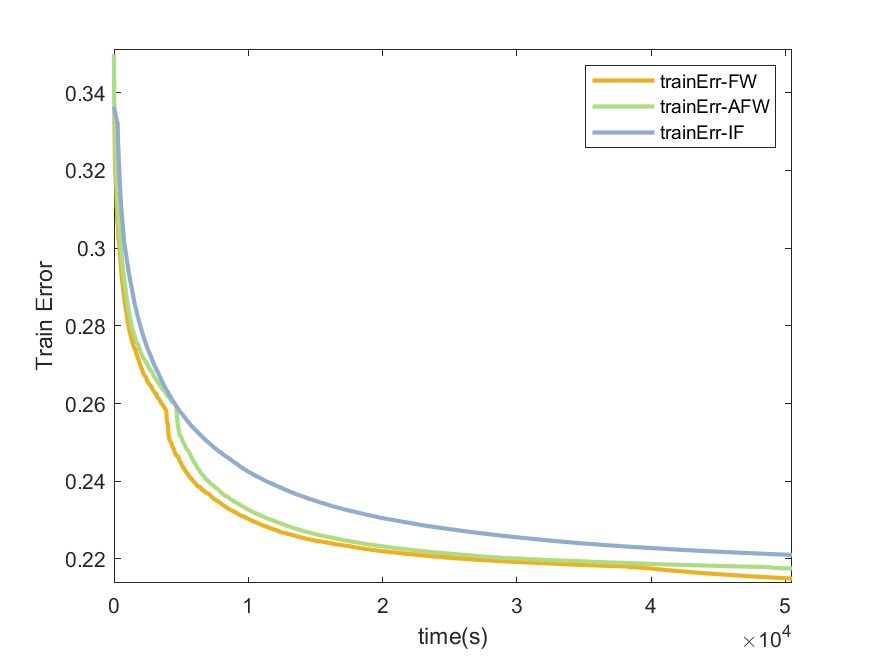}
	\includegraphics[scale=0.191]{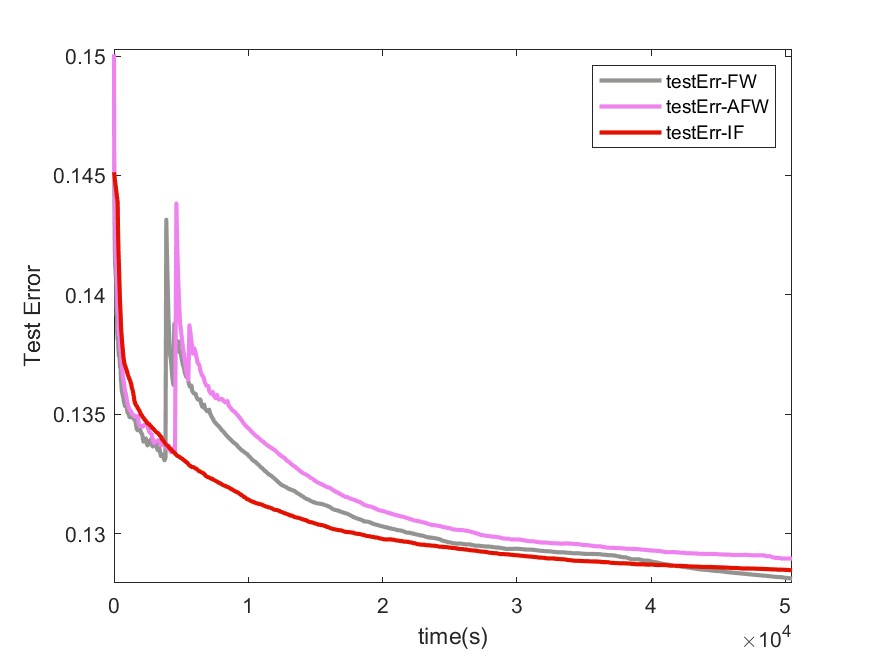}

    \includegraphics[scale=0.191]{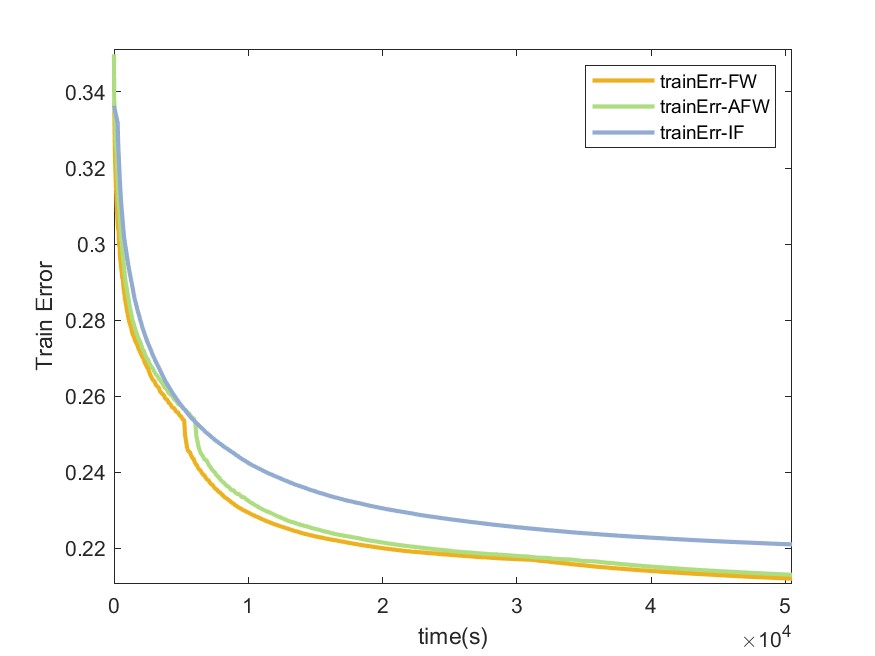}
	\includegraphics[scale=0.191]{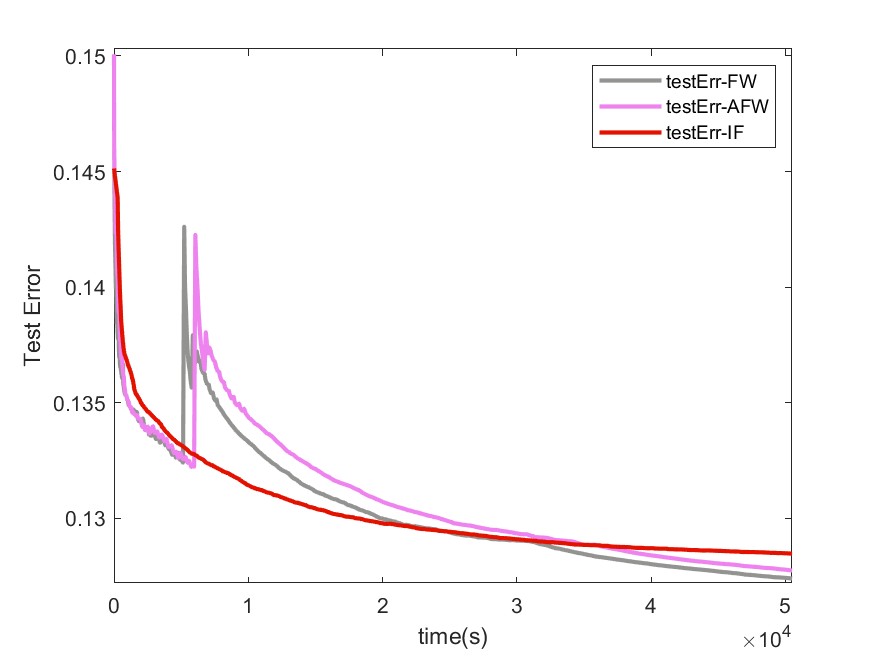}
\end{figure}


\begin{sidewaystable}
	\captionsetup{labelfont={color=blue},font={color=blue}}
	\caption{Relative optimality measure ($ \widehat\varepsilon $), Rank and RMSE for \textbf{IF}-$ (0, \infty)$, \textbf{FW}$_{\bf ncvx}$ and \textbf{AFW}$_{\bf ncvx}$ within different maximal computational time $ T^{\max} $ {\color{blueblue}for different $\mu$}} \label{table:1}
	{ \color{blueblue}
\setlength{\tabcolsep}{0.42mm}
{
		\begin{tabular}{c ccc ccc ccc ccc ccc ccc ccc}
			\toprule
             & & & & \multicolumn{6}{c}{$\mu=0.25$} & \multicolumn{6}{c}{$\mu=0.5$} & \multicolumn{6}{c}{$\mu=0.75$} \\
             \cmidrule(lr){5-10} \cmidrule(lr){11-16} \cmidrule(lr){17-22}
			& \multicolumn{3}{c}{\textbf{IF-$ (0, \infty)$}} & \multicolumn{3}{c}{\textbf{FW}$_{\bf ncvx}$} & \multicolumn{3}{c}{\textbf{AFW}$_{\bf ncvx}$}
                                                             & \multicolumn{3}{c}{\textbf{FW}$_{\bf ncvx}$} & \multicolumn{3}{c}{\textbf{AFW}$_{\bf ncvx}$}
                                                             & \multicolumn{3}{c}{\textbf{FW}$_{\bf ncvx}$} & \multicolumn{3}{c}{\textbf{AFW}$_{\bf ncvx}$}\\
			\cmidrule(lr){2-4} \cmidrule(lr){5-7} \cmidrule(lr){8-10}  \cmidrule(lr){11-13} \cmidrule(lr){14-16}  \cmidrule(lr){17-19} \cmidrule(lr){20-22}
			\multicolumn{22}{c}{{\sf MovieLens10M}  Dataset} \\[2pt]
            \cmidrule(lr){2-4} \cmidrule(lr){5-7} \cmidrule(lr){8-10}  \cmidrule(lr){11-13} \cmidrule(lr){14-16}  \cmidrule(lr){17-19} \cmidrule(lr){20-22}
			$ T^{\max}(s) $ &  $ \widehat\varepsilon $ &     rank &     RMSE &   $ \widehat\varepsilon $ &     rank &     RMSE &   $ \widehat\varepsilon $ &     rank &     RMSE
                                                                             &   $ \widehat\varepsilon $ &     rank &     RMSE &   $ \widehat\varepsilon $ &     rank &     RMSE
                                                                             &   $ \widehat\varepsilon $ &     rank &     RMSE &   $ \widehat\varepsilon $ &     rank &     RMSE \\
			\cmidrule(lr){2-4} \cmidrule(lr){5-7} \cmidrule(lr){8-10}  \cmidrule(lr){11-13} \cmidrule(lr){14-16}  \cmidrule(lr){17-19} \cmidrule(lr){20-22}
1000 &  1.2e-02 &    128 & 0.8090 &  7.8e-03 &    191 & 0.8063 &   1.2e-02 &     90 & 0.8070&  8.4e-03 &    188 & 0.8037 &   1.2e-02 &     94 & 0.8045&  9.2e-03 &    188 & 0.8016 &   1.0e-02 &     96 & 0.8028  \\
2000 &  5.2e-03 &    144 & 0.8083 &  4.1e-03 &    283 & 0.8055 &   4.7e-03 &    117 & 0.8060&  4.2e-03 &    283 & 0.8030 &   4.8e-03 &    114 & 0.8035&  4.9e-03 &    281 & 0.8007 &   5.0e-03 &    116 & 0.8013  \\
3000 &   3.3e-03 &    147 & 0.8081 &  3.8e-03 &    355 & 0.8052 &   3.6e-03 &    133 & 0.8054&  3.5e-03 &    356 & 0.8027 &   1.0e-02 &    {\bf 127} & 0.8029&  3.5e-03 &    349 & {\bf 0.8004} &   3.9e-03 &    128 & 0.8006  \\
		\cmidrule(lr){2-4} \cmidrule(lr){5-7} \cmidrule(lr){8-10}  \cmidrule(lr){11-13} \cmidrule(lr){14-16}  \cmidrule(lr){17-19} \cmidrule(lr){20-22}
		\end{tabular}
	}

\setlength{\tabcolsep}{0.42mm}
{
		\begin{tabular}{c ccc ccc ccc ccc ccc ccc ccc}
			\multicolumn{22}{c}{{\sf MovieLens20M}  Dataset} \\[2pt]
			\cmidrule(lr){2-4} \cmidrule(lr){5-7} \cmidrule(lr){8-10}  \cmidrule(lr){11-13} \cmidrule(lr){14-16}  \cmidrule(lr){17-19} \cmidrule(lr){20-22}
			$ T^{\max}(s) $ &  $ \widehat\varepsilon $ &     rank &     RMSE &   $ \widehat\varepsilon $ &     rank &     RMSE &   $ \widehat\varepsilon $ &     rank &     RMSE
                                                                             &   $ \widehat\varepsilon $ &     rank &     RMSE &   $ \widehat\varepsilon $ &     rank &     RMSE
                                                                             &   $ \widehat\varepsilon $ &     rank &     RMSE &   $ \widehat\varepsilon $ &     rank &     RMSE \\
			\cmidrule(lr){2-4} \cmidrule(lr){5-7} \cmidrule(lr){8-10}  \cmidrule(lr){11-13} \cmidrule(lr){14-16}  \cmidrule(lr){17-19} \cmidrule(lr){20-22}
1000 &  1.8e-01 &     83 & 0.8032 &  3.3e-02 &    117 & 0.8003 &   3.6e-02 &     99 & 0.8023&  3.5e-02 &    114 & 0.7994 &   3.7e-02 &    101 & 0.8006&  3.7e-02 &    118 & 0.7980 &   4.0e-02 &    102 & 0.7988  \\
7000 &  8.1e-03 &    223 & 0.7954 &  7.5e-03 &    373 & 0.7929 &   8.6e-03 &    165 & 0.7936&  7.8e-03 &    365 & 0.7907 &   8.5e-03 &    321 & 0.7910&  8.2e-03 &    369 & 0.7884 &   9.8e-03 &    163 & 0.7893  \\
13000 &  4.5e-03 &    231 & 0.7950 &  4.2e-03 &    504 & 0.7925 &   6.3e-03 &    197 & 0.7927&  4.5e-03 &    494 & 0.7901 &   6.6e-03 &    331 & 0.7903&  4.5e-03 &    494 & {\bf 0.7879} &   4.5e-03 &   {\bf 196} & 0.7881  \\
		\cmidrule(lr){2-4} \cmidrule(lr){5-7} \cmidrule(lr){8-10}  \cmidrule(lr){11-13} \cmidrule(lr){14-16}  \cmidrule(lr){17-19} \cmidrule(lr){20-22}
		\end{tabular}
	}

\setlength{\tabcolsep}{0.42mm}
{
		\begin{tabular}{c ccc ccc ccc ccc ccc ccc ccc}
			\multicolumn{22}{c}{{\sf MovieLens32M}  Dataset} \\[2pt]
			\cmidrule(lr){2-4} \cmidrule(lr){5-7} \cmidrule(lr){8-10}  \cmidrule(lr){11-13} \cmidrule(lr){14-16}  \cmidrule(lr){17-19} \cmidrule(lr){20-22}
			$ T^{\max}(s) $ &  $ \widehat\varepsilon $ &     rank &     RMSE &   $ \widehat\varepsilon $ &     rank &     RMSE &   $ \widehat\varepsilon $ &     rank &     RMSE
                                                                             &   $ \widehat\varepsilon $ &     rank &     RMSE &   $ \widehat\varepsilon $ &     rank &     RMSE
                                                                             &   $ \widehat\varepsilon $ &     rank &     RMSE &   $ \widehat\varepsilon $ &     rank &     RMSE \\
			\cmidrule(lr){2-4} \cmidrule(lr){5-7} \cmidrule(lr){8-10}  \cmidrule(lr){11-13} \cmidrule(lr){14-16}  \cmidrule(lr){17-19} \cmidrule(lr){20-22}
10000 &  2.3e-02 &    230 & 0.7844 &  1.2e-02 &    328 & 0.7823 &   1.7e-02 &    187 & 0.7834&  1.3e-02 &    332 & 0.7801 &   1.8e-02 &    188 & 0.7815&  1.3e-02 &    334 & 0.7787 &   1.6e-02 &    281 & 0.7797  \\
30000 &  5.7e-03 &    291 & 0.7830 &  5.9e-03 &    555 & 0.7806 &   8.8e-03 &    242 & 0.7808&  6.2e-03 &    554 & 0.7783 &   6.0e-03 &    232 & 0.7785&  6.2e-03 &    551 & 0.7762 &   7.6e-03 &    493 & 0.7764  \\
50000 &  4.1e-03 &    294 & 0.7827 &  4.8e-03 &    680 & 0.7802 &   3.0e-03 &    265 & 0.7802&  4.9e-03 &    679 & 0.7778 &   4.3e-03 &   {\bf 261} & 0.7778&  5.4e-03 &    676 & {\bf 0.7756} &   5.7e-03 &    616 & 0.7759  \\
		\cmidrule(lr){2-4} \cmidrule(lr){5-7} \cmidrule(lr){8-10}  \cmidrule(lr){11-13} \cmidrule(lr){14-16}  \cmidrule(lr){17-19} \cmidrule(lr){20-22}
		\end{tabular}
	}

\setlength{\tabcolsep}{0.42mm}
{
		\begin{tabular}{c ccc ccc ccc ccc ccc ccc ccc}
			\multicolumn{22}{c}{{\sf Netflix Prize}  Dataset} \\[2pt]
			\cmidrule(lr){2-4} \cmidrule(lr){5-7} \cmidrule(lr){8-10}  \cmidrule(lr){11-13} \cmidrule(lr){14-16}  \cmidrule(lr){17-19} \cmidrule(lr){20-22}
			$ T^{\max}(s) $ &  $ \widehat\varepsilon $ &     rank &     RMSE &   $ \widehat\varepsilon $ &     rank &     RMSE &   $ \widehat\varepsilon $ &     rank &     RMSE
                                                                             &   $ \widehat\varepsilon $ &     rank &     RMSE &   $ \widehat\varepsilon $ &     rank &     RMSE
                                                                             &   $ \widehat\varepsilon $ &     rank &     RMSE &   $ \widehat\varepsilon $ &     rank &     RMSE \\
			\cmidrule(lr){2-4} \cmidrule(lr){5-7} \cmidrule(lr){8-10}  \cmidrule(lr){11-13} \cmidrule(lr){14-16}  \cmidrule(lr){17-19} \cmidrule(lr){20-22}
10000 &  3.5e-01 &    131 & 0.8575 &  6.6e-02 &    144 & 0.8640 &   7.4e-02 &    131 & 0.8667&  7.0e-02 &    147 & 0.8636 &   7.8e-02 &    134 & 0.8673&  7.7e-02 &    149 & 0.8638 &   9.0e-02 &    136 & 0.8673  \\
30000 &  1.6e-01 &    298 & 0.8498 &  4.3e-02 &    271 & 0.8520 &   4.6e-02 &    230 & 0.8532&  4.5e-02 &    270 & 0.8507 &   4.6e-02 &    240 & 0.8520&  4.7e-02 &    274 & 0.8497 &   4.8e-02 &    246 & 0.8507  \\
50000 &  1.1e-01 &    419 & 0.8479 &  2.8e-02 &    357 & 0.8479 &   3.7e-02 &   {\bf 271} & 0.8503&  3.1e-02 &    356 & 0.8468 &   4.2e-02 &    291 & 0.8494&  2.9e-02 &    360 & {\bf 0.8443} &   3.1e-02 &    324 & 0.8455  \\
			\bottomrule
		\end{tabular}
	}
}
\end{sidewaystable}


\section{Conclusion and outlook}
{\color{blue} In this paper, we proposed a new Frank-Wolfe-type method for minimizing a smooth function over a compact set defined as the level set of a single DC function. The proposed method relies on new generalized linear-optimization oracles that can be efficiently computed for several important nonconvex optimization models arising from compressed sensing and matrix completion. To improve its numerical performance empirically, we also introduced an away-step variant of the proposed method. We analyzed the convergence of these new methods. Finally, we applied the proposed FW-type method and its ``away-step" variant to solve a large-scale matrix completion problem on some standard datasets, and compared them with a popular method \cite{FrGM17} in the literature.

Our results suggest the following interesting future research directions.
\begin{itemize}
\item First, as in most of the existing literature of Frank-Wolfe-type methods, our construction of the proposed methods requires the feasible sets of the underlying optimization problem to be compact. While the compactness assumption is natural, it may limit the applicability of the proposed methods from a wider range of applications. In the very recent work \cite{WLM2022}, the authors proposed a new variant of Frank-Wolfe method which can be applied to {\em convex models with non-compact feasible regions}. It is an interesting future research direction to see how our current work can be extended to allow noncompact constraint sets.
    \item Second, in several important contributions including \cite{CPS18,LuZhou2019,PRA17,R2014}, the authors have established algorithms which can converge to a stronger version of stationary points, called {\em d-stationary points}, for a general optimization problems with DC structure. The d-stationary points do not only enjoy nice theoretical properties \cite{CCHP20} but also often lead to better solution qualities \cite{LuZhou2019}. Therefore, it is also of great interest to see how one could combine the techniques in these articles to construct Frank-Wolfe-type algorithms which converge to d-stationary points.
\end{itemize} }

{\color{blue}
\paragraph{Statements and declarations:} The authors declare that there are no competing interests.
}

\end{document}